\documentclass[12pt]{amsart}
\usepackage[utf8]{inputenc}
\usepackage[margin=2.5cm]{geometry}
\usepackage{tikz-cd} 
\tikzset{
    labl/.style={anchor=north, rotate=90, inner sep=.5mm}
}
\usepackage{cancel}
\usepackage[english]{babel}
\usepackage{stmaryrd}
\usepackage{latexsym,amsmath,amsfonts,amscd,amssymb}
\usepackage{fancyhdr}
\usepackage{hyperref}
\usepackage{xypic}

\usepackage{array}   
\usepackage[mathscr]{eucal} 
\usepackage{mathrsfs}
\usepackage{graphicx}
\usepackage{calligra}
\usepackage[T1]{fontenc}
\DeclareFontFamily{OT1}{pzc}{}
\DeclareFontShape{OT1}{pzc}{m}{it}{<-> s * [1.10] pzcmi7t}{}
\DeclareMathAlphabet{\mathpzc}{OT1}{pzc}{m}{it}

\DeclareFontFamily{OT1}{rsfs}{}
 \DeclareFontShape{OT1}{rsfs}{n}{it}{<->rsfs10}{}
 \DeclareMathAlphabet{\curly}{OT1}{rsfs}{n}{it}

\theoremstyle{plain}  
\newtheorem{theorem}{Theorem}[section]

\newtheorem*{theorem*}{Theorem}
\newtheorem{corollary}[theorem]{Corollary}
\newtheorem{lemma}[theorem]{Lemma}
\newtheorem{proposition}[theorem]{Proposition}

\theoremstyle{definition}
\newtheorem{definition}[theorem]{Definition}

\theoremstyle{remark}

\newtheorem{remark}[theorem]{Remark}
\newtheorem*{remark*}{Remark}

\newtheorem*{claim*}{Claim}

\newcommand{\forevery}{\;\text{for}\;\text{every}\;}
\newcommand{\andd}{\quad\text{and}\quad}

\newcommand{\suhthat}{\mid}

\renewcommand{\le}{\leqslant}

\renewcommand{\ge}{\geqslant}

\newcommand{\R}{\mathbb{R}}

\newcommand{\Z}{\mathbb{Z}}
\newcommand{\C}{\mathbb{C}}

\newcommand{\PPP}{\curly{P}}

\newcommand{\cC}{\mathcal{C}}

\newcommand{\cM}{\mathcal{M}}
\newcommand{\wcM}{\widetilde{\mathcal{M}}}
\newcommand{\mdl}{\mathcal{M}}
\newcommand{\isoc}{H^1}

\newcommand{\calR}{\mathcal{R}}
\newcommand{\wcalR}{\widetilde{\mathcal{R}}}

\newcommand{\lie}{\mathfrak}

\newcommand{\PGL}{\mathrm{PGL}}
\newcommand{\slm}{\SL(n,\C)^{\theta}}
\newcommand{\nslm}{\SL(n,\C)_{\theta}}

\newcommand{\U}{\mathrm{U}}

\newcommand{\GL}{\mathrm{GL}}
\newcommand{\SL}{\mathrm{SL}}

\newcommand{\SO}{\mathrm{SO}}
\newcommand{\NSO}{\mathrm{NSO}}

\newcommand{\Sp}{\mathrm{Sp}}
\newcommand{\NSp}{\mathrm{NSp}}
\newcommand{\Spin}{\mathrm{Spin}}

\DeclareMathOperator{\ad}{ad}
\DeclareMathOperator{\Ad}{Ad}

\DeclareMathOperator{\Hom}{Hom}
\DeclareMathOperator{\End}{End}

\DeclareMathOperator{\id}{Id}
\DeclareMathOperator{\Id}{Id}

\DeclareMathOperator{\Aut}{Aut}
\DeclareMathOperator{\Int}{Int}
\DeclareMathOperator{\Out}{Out}

\DeclareMathOperator{\Nmd}{Nm_{\Delta}}
\DeclareMathOperator{\gal}{Gal}

\renewcommand{\phi}{\varphi}

\newcommand{\gl}{\mathfrak{gl}}
\newcommand{\sll}{\mathfrak{sl}}
\newcommand{\lieg}{\mathfrak{g}}

\renewcommand{\phi}{\varphi}

\newcommand{\alg}{\alpha_{\gamma}}

\newcommand{\llambda}{\Gamma}
\newcommand{\ambda}{\gamma}

\newcommand{\oo}{\mathcal{O}}

\newcommand{\autg}{\Aut(G)}
\newcommand{\homm}[2]{\Hom\left(#1,#2\right)}
\newcommand{\outg}{\Out(G)}
\newcommand{\homgh}{\homm{\Gamma}{\autg}}

\newcommand{\intt}[1]{\Int_{#1}}
\newcommand{\gt}{G^{\theta}}
\newcommand{\gtt}{G^{\theta'}}
\newcommand{\gs}{G_{\theta}}
\newcommand{\pt}{p_{\theta}}

\newcommand{\tg}{\theta_{\gamma}}
\newcommand{\etag}{\eta_{\gamma}}
\newcommand{\ot}{\overline{\theta}}

\newcommand{\zg}{z_{\gamma}}

\newcommand{\cct}{c_{\theta}}
\newcommand{\ct}{c^{\theta}}
\newcommand{\qt}{q_{\theta}}
\newcommand{\qqt}{\qt}
\newcommand{\hg}{h_{\gamma}}
\newcommand{\oh}{\overline{h}}

\newcommand{\ohg}{\overline{h}_{\gamma}}

\newcommand{\fg}{f_{\gamma}}

\newcommand{\taug}{\tau_{\gamma}}
\newcommand{\mug}{\mu_{\gamma}}

\newcommand{\ag}{a_{\gamma}}
\newcommand{\sg}{s_{\gamma}}
\newcommand{\taut}{\tau^{\theta}}

\newcommand{\ctt}{\tilde c_{\theta}}
\newcommand{\gamt}{\Gamma_{\theta}}
\newcommand{\gamtt}{\widehat{\Gamma}_{\theta}}
\newcommand{\gamtz}{\widehat{\Gamma}^{\theta}}
\newcommand{\liegm}{\mathfrak{g}^{\theta}_{\mu}}
\newcommand{\liegt}{\mathfrak{g}_{\theta}}
\newcommand{\zk}{\lie z_{\lie k}}

\newcommand{\zt}{Z^1_{\theta}(\Gamma,Z)}

\DeclareMathOperator{\Fun}{Fun}
\newcommand{\fun}[2]{\Fun(#1,#2)}

\newcommand{\fung}{\fun{\Gamma}{G/Z}}

\newcommand{\ga}{G_{Y}}

\newcommand{\cent}{Z_{\Gamma}(\gam)}

\newcommand{\wt}{\widetilde{\theta}}

\newcommand{\wf}{\widetilde{f}}

\newcommand{\pa}{p_{Y}}
\newcommand{\xa}{Y}

\newcommand{\gam}{\Gamma_{Y}}

\newcommand{\xd}{X_{\Delta}}
\newcommand{\pg}{p_{\Gamma}}
\newcommand{\pd}{p_{\Delta}}

\newcommand{\rhog}{\rho_{\Gamma}}
\newcommand{\rhomu}{\rho_{\mu}}
\newcommand{\rhotm}{\rho_{\tilde\tau,\mu}}
\newcommand{\cdotv}[1]{\bullet \rhog(#1)}

\DeclareMathOperator{\irr}{irr}

\newcommand{\homtc}{\Hom_{\theta,\tau,c}(\wgame,\Aut(G))}
\newcommand{\outc}{\Hom_{\theta,\tau,c}(\wgame,\Out(G))}
\newcommand{\etc}{e_{\tau,c}}

\newcommand{\wgam}{\widehat{\Gamma}_{Y}}
\newcommand{\wgamm}{\widehat{\Gamma}}
\newcommand{\wga}{\widehat{\gamma}}
\newcommand{\wgame}{\wgamm_{\eta}}

\newcommand{\otau}{\overline{\tau}}
\newcommand{\zzgtw}{Z^2_{\tau}(\wgam,Z(\gt_0))}
\newcommand{\zzgt}{Z^2_{\tau}(\gamt,Z(\gt_0))}
\newcommand{\omu}{\overline{\mu}}

\newcommand{\wtg}{\widetilde{\theta}_{\gamma}}
\newcommand{\wtb}{\widetilde{\theta}_{\beta}^{-1}}

\newcommand{\wtaug}{\tilde{\tau}_{\gamma}}

\newcommand{\fb}{\eta_{\beta}}
\newcommand{\hb}{h_{\beta}}
\newcommand{\mub}{\mu_{\beta}}
\newcommand{\ttg}{\theta'_{\gamma}}

\newcommand{\halpha}{\langle\alpha\rangle}

\newcommand{\olambda}{\overline{\Lambda}}
\newcommand{\gtl}{\gt_{\Lambda}}
\newcommand{\gtll}{\gt_{\olambda}}

\newcommand{\lambdaa}{*(\gamma\lambda\gamma^{-1})}
\newcommand{\lambdaaa}{*\lambda\langle\alg,\lambda\rangle^{-1}}

\title[Finite group actions on Higgs bundle moduli spaces]{Finite group actions on Higgs bundle moduli spaces}
\author{Guillermo Barajas, Suratno Basu and Oscar García-Prada}
\date{\today}



\begin{document}

\begin{abstract}
Let $\cM(X,G)$ be the moduli space of $G$-Higgs bundles over a compact 
Riemann  surface $X$, where $G$ is a semisimple complex Lie group with centre $Z$. We describe the fixed points of  
the action of a finite group $\Gamma$ on $\cM(X,G)$, induced by  
holomorphic actions of $\Gamma$ on $X$ and $G$, a character of $\Gamma$ and a homomorphism from $\Gamma$ to the group of $Z$-bundles over $X$.
Two important ingredients in this study are provided by the theory of twisted $\Gamma$-equivariant bundles developed by Barajas--García-Prada--Gothen--Mundet i Riera, and the Prym--Narasimhan--Ramanan construction given by Barajas--García-Prada.
Via the non-abelian Hodge correspondence, our results provide a description of the fixed-point subvarieties of certain finite group actions on the $G$-character variety of the fundamental group of $X$.
\end{abstract}

\maketitle

\section{Introduction}

\subsection{Background and statement of the problem}\label{section-background}
\noindent The aim of this paper is to study fixed points of finite group actions on moduli spaces of $G$-Higgs bundles over a compact Riemann surface $X$. Here $G$ is a connected semisimple complex Lie group, but most of our results are valid for $G$ reductive. A $G$-Higgs bundle is a pair $(E,\phi)$ consisting of a principal $G$-bundle $E$ over $X$ and a global section $\phi$ --- the Higgs field --- of $E(\mathfrak{g})\otimes K_X$, where $E(\mathfrak{g})$ is the bundle associated to $E$
via the adjoint representation of $G$ in $\lieg$ and $K_X$ is the canonical bundle of $X$. There is a moduli space $\mdl(X,G)$ classifying isomorphism classes of polystable $G$-Higgs bundles --- see Section \ref{section-higgs-bundles}. When $G$ is a classical group, $G$-Higgs bundles correspond to vector bundles equipped with some extra structure and a compatible endomorphism twisted by $K_X$.


We consider several natural group actions on $\mdl(X,G)$. First, the group $\C^*$ acts on $\mdl(X,G)$ by rescaling the Higgs field. 
In this paper we are only concerned with actions of finite subgroups of $\C^*$. The simplest example, already considered by Hitchin \cite{hitchin1987} for $G=\SL(2,\C)$, and in \cite{PR} for general $G$, is the involution $\iota$ of $\mdl(X,G)$ sending $(E,\phi)$ to $(E,-\phi)$. The fixed points of $\iota$ correspond to Higgs pairs with structure group $\gt$, where $\theta$ runs over the inner involutions of $G$. Here $\gt$ is the subgroup of fixed points of $\theta$ and the Higgs field takes values in the $-1$-eigenspace of the automorphism of $\lie g$ induced by $\theta$. Via the non-abelian Hodge correspondence, these fixed points are in bijection with representations of the fundamental group of $X$ in the real form $G^{\sigma}$, where $\sigma=\theta\tau$ for some compact conjugation $\tau$ of $G$ that commutes with $\theta$ --- note that $\sigma$ is of Hodge type, since $\theta$ is inner. 

Secondly, the group of automorphisms $\Aut(G)$ of $G$ acts on $\mdl(X,G)$ by extension of structure group. More precisely, given a holomorphic automorphism $\theta$ of $G$ and a $G$-Higgs bundle $(E,\phi)$ over $X$, we may twist the $G$-action on the total space of $E$ by $\theta^{-1}$ to obtain a new $G$-bundle, which we denote by $\theta(E)$, and the isomorphism of vector bundles $E(\lie g)\cong \theta(E)(\lie g)$ induced by $\theta$ produces a Higgs field $\theta(\phi)$ for $\theta(E)$. The action of the subgroup of inner automorphisms $\Int(G)\le\Aut(G)$ respects isomorphism classes: for each $g\in G$, the map $E\to E$ given by multiplication by $g$ induces an isomorphism from $(E,\phi)$ to $(\Int_g(E),\Int_g(\phi))$. We thus get a left action of $\Out(G):=\Aut(G)/\Int(G)$. 

For example, if $a$ is a non-trivial element of $\Out(G)$ such that $a^2=1$, the fixed points correspond to $G^{\theta}$-Higgs bundles, where $\theta$ is an outer involution of $G$ lifting $a$. Since the actions of $\Out(G)$ and $\C^*$ commute, we may also consider the involution combining both $a$ and $-1$ sending $(E,\phi)$ to $(\theta(E),-\theta(\phi))$, in which case fixed points correspond to representations of $\pi_1(X)$ in certain real forms of $G$ which are no longer of Hodge type. In fact, varying $a$, all the possible real forms of $G$ appear --- see \cite{PR}. 

Another important group acting on the moduli space of $G$-Higgs bundles is $H^1(X,Z)$, the group of isomorphism classes of $Z$-bundles, where $Z$ is the centre of $G$. Given a $G$-Higgs bundle $(E,\phi)$ and a $Z$-bundle $L$, we may define the tensorization $E\otimes L$ to be the $G$-bundle obtained from the $G\times Z$-bundle $E\times L$ via extension of structure group by the multiplication homomorphism. Since the adjoint action of $Z$ on $\lie g$ is trivial, the Higgs field $\phi$ may also be regarded as a Higgs field of $E\otimes L$. 

It can be seen that the actions of $\Out(G)$ and $H^1(X,Z)$ do not commute, but rather the former twists the latter by extension of structure group. We thus get a combined right action of $H^1(X,Z)\rtimes\Out(G)\times\C^*$ on $\mdl(X,G)$, where the semidirect product is defined using the aforementioned action of $\Out(G)$ on $H^1(X,Z)$
.
A description of the fixed points of the action of a general finite cyclic subgroup of $H^1(X,Z)\rtimes\Out(G)\times\C^*$ on $\mdl(X,G)$ is given in \cite{PR}.

An important construction related to the action of $H^1(X,Z)$ in the context of vector bundles, which we generalize to Higgs bundles, is the well-known Narasimhan--Ramanan description of fixed points of the action of a finite cyclic subgroup $\langle L\rangle$ of the Jacobian, where $L$ is a line bundle of order $r$, on the moduli space of vector bundles of rank $n$ and degree $d$, where $r$ divides $n$. Narasimhan and Ramanan \cite{narasimhan-ramanan} find that the intersection of the fixed point locus with the stable locus is equal to the pushforward of the moduli space of stable vector bundles over $X_L$ of rank $n/r$ and degree $d$ satisfying certain additional conditions, where $X_L$ is the étale cover of $X$ determined by $L$. Nasser \cite{nasser} improves this result by showing that the whole polystable fixed point locus is equal to the pushforward of the whole moduli space of polystable vector bundles with the same invariants. 

The last action on $\mdl(X,G)$ that we consider is that of the group of holomorphic automorphisms $\Aut(X)$ of $X$ via pullback. This twists the action of $H^1(X,Z)$ via pullback, hence ultimately we get a right action of the group
 \begin{equation*}
     H^1(X,Z)\rtimes(\Aut(X)\times \Out(G))\times \mathbb{C}^*
     \tag{$\star$}
 \end{equation*}
on $\mdl(X,G)$, where the semidirect product is defined by the actions of $\Aut(X)$ and $\Out(G)$ on $H^1(X,Z)$ given by pullback and extension of structure group, respectively. This is given explicitly by
$$
  (E,\phi)\cdot (\alpha,\eta,a,\mu):=(\eta^*\theta^{-1}(E\otimes\alpha), \mu\eta^*\theta^{-1}(\phi)),
$$
for any element $(\alpha,\eta,a,\mu)$ in ($\star$) and any automorphism  $\theta$ of $G$ lifting $a$.  

Fixed points for the action of a finite subgroup of $\Aut(X)$ have been studied by several authors, including Andersen--Grove \cite{andersen-fixed-points}, Andersen \cite{andersen}, Heller--Schaposnik \cite{heller-schaposnik}, and in \cite{ow}. Fixed points of general finite subgroups of $\Aut(X)\times \Out(G)\times \C^*$ are described in the 2020 version of this paper by the second and third authors
\cite{oscar-suratno} in terms of twisted equivariant Higgs bundles.
In Section \ref{section-twisted-equivariant-higgs-pairs} of the current paper we merge those results with new results based on the 2023 PhD thesis of the 
first author \cite{barajas-thesis}.

This paper undertakes the task of describing the fixed points of the action of an arbitrary finite subgroup $\Gamma$ of ($\star$) on $\mdl(X,G)$. In particular, we generalize \cite{PR} to any finite subgroup. Although the action of $H^1(X,Z)$ is already considered in \cite{PR}, the fixed point description there is given in terms of reductions of the structure group $G$ to a possibly non-connected subgroup. In this paper we go one step further, showing that these reductions correspond to Higgs pairs over certain \'etale covers of $X$ whose structure group is the connected component of the aforementioned subgroup. This extends to $G$-Higgs bundles the Prym--Narasimhan--Ramanan construction given in \cite{prym-narasimhan-ramanan} for principal $G$-bundles, which in turn generalizes results of Narasimhan--Ramanan \cite{narasimhan-ramanan} for vector bundles.

Our answer to the general problem is given by Theorem \ref{th-prym-narasimhan-ramanan-general}. In order to give the reader a good grasp of it, we have isolated its main ingredients into few special cases where their role can be more clearly understood. The particular case of finite group actions of $Z$-bundles on the moduli space of holomorphic $G$-bundles is treated in \cite{prym-narasimhan-ramanan}. The formalism of twisted equivariant principal bundles, which plays a central role in this study, is developed in \cite{BGGM}. 

\subsection{Action of a finite subgroup \texorpdfstring{$\Gamma$}{Gamma} of \texorpdfstring{$\Aut(X)\times\Out(G)\times\C^*$}{without involving tensorization}.}\label{section-intro-no-tensorization}

\noindent We first describe the fixed points of a finite subgroup $\Gamma\le \Aut(X)\times\Out(G)\times\C^*$, that is, an arbitrary finite group of ($\star$) not involving tensorization by elements of $H^1(X,Z)$. In this situation, the fixed point variety $\mdl(X,G)^{\Gamma}$ is given in terms of twisted equivariant Higgs bundles over $X$, which are introduced in Section \ref{section-twisted-equivariant-higgs-pairs}. The corresponding results are explained in Section \ref{section-alpha-trivial}. Let $\eta$, $a$ and $\mu$ be the natural projection homomorphisms from $\Gamma$ to $\Aut(X)$, $\Out(G)$ and $\C^*$, and take a homomorphism $\theta:\Gamma\to\Aut(G)$ lifting $a$. 

Take a 2-cocycle $c\in Z^2(\Gamma,Z)$ --- a map $c:\Gamma\times\Gamma\to Z$ satisfying certain conditions --- see (\ref{eq-2-cocycle}). A (right) $(\theta,c)$-twisted $\Gamma$-equivariant action on a $G$-bundle $E$ over $X$, which we usually denote by $\bullet$, is a lift of the right action of $\Gamma$ on $X$ (determined by $\eta$) to a map from $\Gamma$ to the group of holomorphic automorphisms of the total space of $E$. This action must further twist the bundle $G$-action via $\theta$, and the composition of operating by $\gamma_1$ and $\gamma_2\in\Gamma$ must be equal to operating by $\gamma_1\gamma_2$ multiplied by $c(\gamma_1,\gamma_2)$. The pair $(E,\bullet)$ is called a $(\theta,c)$-twisted $\Gamma$-equivariant $G$-bundle. More generally, we may define (left or right) $(\theta,c)$-twisted $\Gamma$-actions on a set $M$ equipped with a group $G$-action. We review the necessary results on the theory of twisted equivariant actions in Section \ref{section-twisted-equivariant} --- for more details, see \cite{BGGM}.

Given a $(\theta,c)$-twisted $\Gamma$-equivariant $G$-bundle $(E,\bullet)$ and a character $\mu:\Gamma\to\C^*$, the adjoint bundle $E(\lie g) $ inherits an honest group action of $\Gamma$ making it a $\Gamma$-equivariant vector bundle. More precisely, each $\gamma\in\Gamma$ sends $(e,v)\in E(\lie g)$ to $(e\bullet\gamma,\mug\tg^{-1}(v))$. 
Hence there is a notion of $(\theta,c,\omu\theta)$-twisted $\Gamma$-equivariant $G$-Higgs bundle, namely a triple $(E,\bullet,\phi)$ such that $(E,\bullet)$ is a $(\theta,c)$-twisted $\Gamma$-equivariant $G$-bundle and the Higgs field $\phi$ is $\Gamma$-invariant. There are notions of (poly, semi)stability for twisted equivariant Higgs bundles and a moduli space $\mdl(X,G,\Gamma,\theta,c,\omu\theta)$ classifying isomorphism classes of polystable objects, see Section \ref{section-twisted-equivariant-higgs-pairs}. There is also a natural forgetful morphism $\mdl(X,G,\Gamma,\theta,c,\omu\theta)\to\mdl(X,G)$ omitting the twisted $\Gamma$-action, whose image we denote by $\widetilde{\mdl}(X,G,\Gamma,\theta,c,\omu\theta)$.

\vspace{5 pt}
{\bf Theorem A} (Theorem \ref{main}). {\it  Let ${\cM}_*(X,G)\subset \cM(X,G)$ be the open subvariety of $\cM(X,G)$ 
consisting of those $G$-Higgs bundles
 which are stable and simple. Then 
$$ \bigcup_{[c]\in H^2_a(\Gamma,Z)} 
\widetilde{\cM}(X,G,\Gamma,\theta,c,\omu\theta)
\subset
\cM(X,G)^{\Gamma}
$$
and 
 $$
{\cM}_*(X,G)^{\Gamma}\subset \bigcup_{[c]\in H^2_a(\Gamma,Z)} 
\widetilde{\cM}(X,G,\Gamma,\theta,c,\omu\theta).
$$
Here $\theta:\Gamma\to \Aut(G)$ is 
any  lift of $a:\Gamma\to \Out(G)$, and $H^2_a(\Gamma,Z)$ is the second cohomology group of $\Gamma$ with values in $Z$ with $\Gamma$-action induced by $a$.  Moreover, the subvarieties
\begin{equation*}
    {\cM}_*(X,G)\cap\widetilde{\cM}(X,G,\Gamma,\theta,c,\omu\theta)={\cM}_*(X,G)^{\Gamma}\cap\widetilde{\cM}(X,G,\Gamma,\theta,c,\omu\theta)
\end{equation*}
are all disjoint for different cohomology classes $[c]\in H^2_a(\Gamma,Z)$.}
\vspace{5 pt}

\subsection{Fixed points for trivial action on \texorpdfstring{$X$}{X}}\label{section-intro-no-X}
\noindent In Section \ref{section-prym-narasimhan-ramanan} we consider the action of an arbitrary finite subgroup $\Gamma$ of $H^1(X,Z)\rtimes\Out(G)\times \mathbb{C}^*$. Projections on $\Out(G)$ and $\C^*$ provide homomorphisms $a$ and $\mu$, and projection on $H^1(X,Z)$ yields a map $\alpha:\Gamma\to H^1(X,Z)$ which satisfies
$
\alpha_{\gamma\gamma'}=\alpha_\gamma \alpha^{\gamma}_{\gamma'}
$
for each $\gamma$ and $\gamma'$ in $\Gamma$,
where the superscript denotes the left action of $\Gamma$ on $H^1(X,Z)$ induced by $a$. Maps $\alpha$ satisfying this equation are called 1-cocycles, and we write $Z^1_a(\Gamma,H^1(X,Z))$ for the group of 1-cocycles. In general, given a left action $b:\Gamma\to\Aut(A)$ on a group $A$, there is a notion of 1-cocycle and we denote the set of 1-cocycles by $Z^1_{b}(\Gamma,A)$.

It is well known --- see e.g. \cite{BGGM} --- that there exists a homomomorphism $\theta:\Gamma\to\Aut(G)$ lifting $a$. Define the subgroups
$$\gt:=\{g\in G\suhthat\tg(g)=g\forevery\gamma\in\Gamma\}\le G$$
and
$$\gs:=\{g\in G\suhthat\tg(g)=z(\gamma,g)g\,\forevery\gamma\in\Gamma\;\text{and some}\;z(\gamma,g)\in Z\}\le G.$$
These fit into a short exact sequence of groups
\begin{equation*}
    1\to \gt\to\gs\xrightarrow{\cct} \zt.
\end{equation*}
Therefore, there is an induced map from the set of isomorphism classes of $\gs$-bundles $H^1(X,\underline{\gs})$ to $H^1(X,Z^1_a(\Gamma,Z))$. Since $G$ is semisimple, $Z$ is finite and so $X$ and $\Gamma$ may be exchanged, thus providing a map $\ctt:H^1(X,\underline{\gs})\to Z^1_a(\Gamma,H^1(X,Z))$ which may be thought of as a ``characteristic class'' map which measures the obstruction of each $\gs$-bundle to having a reduction of structure group to $\gt$. We may also introduce the $\mu$-weight space $\liegm$ of the automorphism $\theta$ of $\liegm$. The restriction of the adjoint action to $\gs$ preserves $\liegm$ and so there is a notion of $(\gs,\liegm)$-Higgs pair, namely a pair $(F,\psi)$ consisting of a $\gs$-bundle $F$ and a Higgs field $\psi\in H^0(X,F(\liegm)\times K_X)$. There is a moduli space $\mdl_{\alpha}(X,\gs,\liegm)$ of $(\gs,\liegm)$-Higgs pairs $(F,\psi)$ such that $\ctt(F)\cong\alpha$. 

There is also an extension of structure group morphism $\cM_{\alpha}(X,\gs,\liegm)\to\mdl(X,G)$, whose image we denote by $\widetilde{\cM}(X,\gs,\liegm)$. This image is independent of the class $[\theta]$ of $\theta$ in $\Hom(\Gamma,\Aut(G))/\sim$, where $\sim$ is the equivalence relation such that $\theta\sim \theta'$ if and only there exists $g$ in $G$ satisfying $\theta\sim \Int_g\theta'\Int_g^{-1}$. There is a bijection between lifts $\beta\theta$ of $a$ and 1-cocycles $\beta\in Z^1_{\theta}(\Gamma,\Int(G))$ in the sense of group cohomology \cite{serre-galois}, inducing a bijection between the set of classes of lifts in $\Hom(\Gamma,\Aut(G))/\sim$ and the first group cohomology set $H^1_{\theta}(\Gamma,\Int(G))$. Here $\theta$ is any fixed lift. 

\vspace{5 pt}
{\bf Theorem B} (Theorem \ref{th-fixed-points-oscar-ramanan-higgs}). {\it Fix $\theta\in\homgh$ lifting $a$. The following relations between moduli spaces hold:
\begin{enumerate}
    \item $$\bigcup_{[\beta]\in H^1_{\theta}(\Gamma,\Int(G))}\widetilde{\cM}_{\alpha}(X,G_{\beta\theta},\lieg^{\beta\theta}_{\mu})\subset\cM(X,G)^{\Gamma}. $$
    
    \item $$\cM_*(X,G)^{\Gamma}\subset\bigcup_{[\beta]\in H^1_{\theta}(\Gamma,\Int(G))}\widetilde{\cM}_{\alpha}(X,G_{\beta\theta},\lieg^{\beta\theta}_{\mu}).$$
    
\end{enumerate}

Moreover, the intersections 
$$\cM_*(X,G)\cap\widetilde{\cM}_{\alpha}(X,G_{\beta\theta},\lieg^{\beta\theta}_{\mu})=\cM_*(X,G)^{\Gamma}\cap\widetilde{\cM}_{\alpha}(X,G_{\beta\theta},\lieg^{\beta\theta}_{\mu})$$
are disjoint for different $[\beta]\in H^1_{\theta}(\Gamma,\Int(G))$.}
    
\vspace{5 pt}

\subsection{The Prym--Narasimhan--Ramanan construction}\label{section-intro-prym-narasimhan-ramanan}

\noindent In Theorem \ref{th-prym-narasimhan-ramanan-oscar-ramanan-higgs} we take the results of Theorem B one step further, applying the Prym--Narasimhan--Ramanan construction given in \cite{prym-narasimhan-ramanan}.

Let $\gt_0<\gs$ be the connected component of the identity and $\gamtt:=\gs/\gt_0$ the group of connected components. Let $\tau:\gamtt\to\Aut(\gt_0)$ be a lift of the characteristic homomorphism of the extension $\gs$. Then there is a 2-cocycle $c\in Z^2_{\tau}(\Gamma,Z(\gt_0))$, where $Z(\gt_0)$ is the centre of $\gt_0$, and an isomorphism $\gs\cong \gt_0\times_{(\tau,c)}\gamtt$ --- see Proposition \ref{prop-extensions-isomorphic-twisted-group}. Here $\gt_0\times_{(\tau,c)}\gamtt$ is the set $\gt_0\times\gamtt$ equipped with a group multiplication involving $\tau$ and $c$ --- e.g., if $c=1$ then this is a semidirect product, see (\ref{eq-def-twisted-product}).

Now, given a $\gs$-bundle $p_E:E\to X$, the bundle projection morphism $p_E$ can be factored as $E\to E/\gt_0\to X$, where $E/\gt_0\to X$ is a $\gamtt$-bundle over $X$. Let $Y$ be a connected component of $E/\gt_0$ with structure group $\gam\le\gamtt$. The reductiveness of $\gs$ implies that $\gamtt$ is finite, hence $Y$ may be regarded as an étale cover of $X$ with Galois group $\Gamma_Y$. Using the description of $\gs$ as a twisted product we may further equip $E$ with a $(\tau,c)$-twisted $\gamtt$-equivariant action --- see Section \ref{section-intro-no-tensorization}. This provides a bijection between the set of isomorphism classes of $\gs$-bundles over $X$ and a finite group-quotient of the set of isomorphism classes of $(\tau,c)$-twisted $\Gamma_Y$-equivariant $\gt_0$-bundles over $Y$ --- Proposition \ref{prop-prym-narasimhan-ramanan}. The finite group is equal to the centralizer of $\Gamma_Y$ in $\gamtt$, which coincides with the centre of $\gamtt$ if  $E/\gt_0$ is connected. All this theory was first developed in \cite{BGGM}. 

An analogous result for Higgs pairs is given by Proposition \ref{prop-twisted-equivariant-bundles-one-to-one}. This involves the notion of a $(\tau,c,\rho_Y)$-twisted $\Gamma_Y$-equivariant $(\gt_0,\liegm)$-Higgs pair over $Y$, where $\rho_{Y}:\Gamma_Y\to \GL(\liegm)$ is the $(\tau,c)$-twisted representation given by the composition of the map $\Gamma_Y\to\gs$ and the adjoint representation of $G$ on $\lie g$. This is a triple $(F,\bullet,\psi)$ consisting of a $(\tau,c)$-twisted $\Gamma_Y$-equivariant $\gt_0$-bundle $(F,\bullet)$ and a Higgs field $\psi\in H^0(Y,F(\liegm)\otimes K_Y)$ invariant with respect to the $\Gamma_Y$-action on $F(\liegm)$ induced by $\bullet$ and $\rho_Y$. There is a moduli space of $(\tau,c,\rho_Y)$-twisted $\Gamma_Y$-equivariant $(\gt_0,\liegm)$-Higgs pair over $Y$, which we denote $\mdl(Y,\gt_0,\Gamma_Y,\tau,c, \lieg^{\theta}_{\mu},\rho_Y)$.

\vspace{5 pt}
{\bf Theorem C} (Theorem \ref{th-prym-narasimhan-ramanan-oscar-ramanan-higgs}). {\it Write $\taut$ and $c^{\theta}$ instead of $\tau$ and $c$ to emphasize their dependence on $\theta$.  Then, for each homomorphism $\theta:\Gamma\to\Aut(G)$ lifting $a$ there is an isomorphism
\begin{equation}
    \bigsqcup_{\qqt(Y)\cong \alpha}\mdl(Y,\gt_0,\Gamma_Y,\tau^{\theta},c^{\theta}, \lieg^{\theta}_{\mu},\rho_Y)/Z_{\gamtt}(\Gamma_Y)\cong \mdl_{\alpha}(X,\gs,\liegm),
\end{equation}
where $Z_{\gamtt}(\Gamma_Y)$ is the centralizer of $\Gamma_Y$ in $\gamtt$, which acts on $\mdl(Y,\gt_0,\Gamma_Y,\tau^{\theta},c^{\theta}, \lieg^{\theta}_{\mu},\rho_Y)$ by extension of structure group --- see Proposition \ref{prop-action-centralizer}.}
\vspace{5 pt}

The combination of Theorems B and C is what we call Prym--Narasimhan--Ramanan construction, which we first introduced in \cite{prym-narasimhan-ramanan}. This may be regarded as a correspondence between fixed points for the $\Gamma$-action on $\cM(X,G)$, and fixed points in $\bigsqcup_{\qqt(Y)=\llambda}{\mdl}(Y,\gt_0, \liegm)$ for the action of each Galois group $\gam$ on ${\mdl}(Y,\gt_0, \liegm)$, such that $\gamma\in\gam$ sends a $(\gt_0, \liegm)$-Higgs pair $(F,\psi)$ over $Y$ to $\gamma^{*-1}\taug(F,\psi)$.

In Section \ref{section-example-generalize-narasimhan} we show how Theorems B and C generalize \cite{narasimhan-ramanan} to $\GL(n,\C)$-Higgs bundles. Let $G=\GL(n,\C)$ and $\Gamma\le J(X)$ be generated by a line bundle $L$ of finite order $r$. In this situation the only class $[\theta]\in\Hom(\Gamma,\Int(\GL(n,\C)))/\sim$ in the decomposition of the fixed point locus of Theorem B is represented by $\theta(L)=\Int_D$, where $D$ is the diagonal matrix whose eigenvalues are all the $r$-th roots of unity, each having multiplicity $m:=n/r$. In particular, $r$ must divide $n$. In this setting $\GL(n,\C)^{\theta}\cong\GL(m,\C)^{\times r}$ and $\GL(n,\C)_{\theta}\cong\GL(n,\C)^{\theta}\rtimes_{\tau}(\Z/r\Z)$, where the action $\tau$ of $\Z/r\Z$ on $\GL(m,\C)^{\times r}$ permutes the different copies of $\GL(m,\C)$. 

Let $p_L:X_L\to X$ be the étale cover determined by $L$, which has Galois group $\Z/r\Z$. Theorem C implies that $\mdl_*(X,\GL(n,\C))^{L}$ is isomorphic to an open subvariety in the moduli space $\mdl(X_L,\GL(m,\C)^{\times r},\Z/r\Z,\tau,1)/(\Z/r\Z)$. Let us translate this into the language of vector bundles with no consideration of the Higgs field: let $E$ be a $(\tau,1)$-twisted $\Z/r\Z$-equivariant $\GL(m,\C)^{\times r}$-bundle over $X_L$. The associated vector bundle is a direct sum $E_1\oplus\dots\oplus E_r$, where $E_i$ is a vector bundle of rank $m$. There is an induced $\Z/r\Z$-equivariant action which permutes the summands, hence $E_i\cong\zeta^{*i}E_1$, where $\zeta$ is a generator of $\gal(X_L/X)\cong \Z/r\Z$. The quotient of $E$ by this action is the pushforward $p_{L*}E_1$, which recovers the result of Narasimhan and Ramanan. 

\subsection{Fixed points for general \texorpdfstring{$\Gamma$}{Gamma}}

\noindent The fixed point locus $\mdl(X,G)^{\Gamma}$ for a general finite subgroup $\Gamma$ of ($\star$) is studied in Section \ref{section-general}. The study of the fixed points in this general setting is naturally much more involved than the particular cases mentioned above. However, we sketch the main ingredients --- see Section \ref{section-general} for more details. Projections on $\Aut(X)$, $\Out(G)$ and $\C^*$ provide homomorphisms $\eta:\Gamma\to\Aut(X)$, $a:\Gamma\to\Out(G)$ and $\mu:\Gamma\to\C^*$, whereas projection on $H^1(X,Z)$ yields a 1-cocycle $\alpha\in Z^1_{a,\eta}(\Gamma,H^1(X,Z))$, where $\Gamma$ acts on $H^1(X,Z)$ by sending $L$ to $\eta^{*-1}a(L)$. Our strategy is to push Theorem A to account for the potential non-injectivity of $\eta$, so that we first perform a reduction of structure group applying Theorem B to the action of $\ker\eta$ and then we get a twisted equivariant action of
the finite group of automorphisms of an étale cover of $X$ lifting $\eta(\Gamma)$. The main result is Theorem \ref{th-prym-narasimhan-ramanan-general}, where the fixed point subvariety inside of the smooth locus of $\mdl(X,G)$ is described as a union of components as in Theorems A and B. The fixed point components are parametrized by four parameters $[\beta],Y,[\tau^{\beta\theta}],[c^{\beta\theta}]$ and $[\tilde\tau]$, described as follows.
\begin{itemize}
    \item Given a fixed homomorphism $\theta:\ker \eta\to\Aut(G)$ lifting the restriction of $a$ to the kernel of $\eta$, the class $[\beta]$ is an element of $H^1_{\theta}(\ker\eta,\Int(G))$, generalizing the classes $[\beta]$ appearing in Theorem B.
    \item $Y$ runs over certain étale covers of $X$, determined by a condition which depends on $\alpha$ and $\beta$. These generalize the \'etale covers $Y$ appearing in Theorem $C$.
    \item Let $\wgam$ be the group consisting of automorphisms of $Y$ descending to automorphisms of $X$ in $\eta(\Gamma)$. Then $[\tau^{\beta\theta}]\in \Hom(\wgam,\Out(G^{\beta\theta}_0))$ and $[c^{\beta\theta}]\in H^2_{\tau^{\beta\theta}}(\wgam,Z(G^{\beta\theta}_0))$ are classes subject to a condition depending on $\beta$. These generalize the classes of the elements $\taut$ and $c^{\theta}$ appearing in Theorem C, after replacing $\theta$ with $\beta\theta$.
    \item $[\tilde\tau]$ runs over certain set of homomorphisms depending on the previous parameters. This parameter only appears in this general formulation. Together with $\mu$, it determines a $(\tau,c)$-twisted $(\gt_0,\wgam)$-action on $\liegm$ which we call $\rhotm$.
\end{itemize}

\vspace{5 pt}

{\bf Theorem D} (Theorem \ref{th-prym-narasimhan-ramanan-general}).The following relations between moduli spaces hold:
\begin{enumerate}
    \item $$\bigcup_{[\beta],Y,[\tau^{\beta\theta}],[c^{\beta\theta}],[\tilde\tau]}\wcM(Y_{\alpha},G^{\beta\theta}_0,\wgam,\tau^{\beta\theta},c^{\beta\theta},\tilde\tau,\lie g^{\beta\theta}_{\mu},\rhotm)
    \subset\cM(X,G)^{\Gamma}. $$
    
    \item $$\cM_*(X,G)^{\Gamma}\subset
    \bigcup_{[\beta],Y,[\tau^{\beta\theta}],[c^{\beta\theta}],[\tilde\tau]}\wcM(Y_{\alpha},G^{\beta\theta}_0,\wgam,\tau^{\beta\theta},c^{\beta\theta},\tilde\tau,\lie g^{\beta\theta}_{\mu},\rhotm).$$
    
\end{enumerate}

\subsection{Fixed points in character varieties}

\noindent The non-abelian Hodge correspondence provides a homeomorphism between $\mdl(X,G)$ and the character variety $\calR(X,G)$ parametrizing $G$-conjugacy classes of reductive representations $\pi_1(X)\to G$. The theory has an analogue for twisted equivariant Higgs bundles which we explain next. Fix an action of a finite group $\Gamma$ on $X$, a homomorphism $\theta:\Gamma\to\Aut(G)$ and a 2-cocycle $c\in Z^2_{\theta}(\Gamma,Z)$. Let $\calR(X,G,\Gamma,\theta,c)$ be the moduli space of $G$-conjugacy classes of reductive representations of the equivariant fundamental group $\pi_1(X,\Gamma,x)$ of $X$ in $G\times_{(\theta,c)}\Gamma$. Then $\calR(X,G,\Gamma,\theta,c)$ is homeomorphic to $\mdl(X,G,\Gamma,\theta,c)$. We may thus translate our results to give a description of the fixed point loci of certain finite group actions on character varieties. 

More precisely, an element $(\alpha,\eta,a)$ in $H^1(X,Z)\rtimes(\Aut(X)\times\Out(G))$ sends a representation $\rho:\pi_1(X,x)\to G$ to $\theta^{-1}\circ(\eta_*(\rho\otimes\alpha))$. Here $\theta\in\Aut(G)$ is a lift of $a$, $\eta_*$ is induced by the homomomorphism $\eta_*:\pi_1(X,x)\to\pi_1(X,\eta(x))$ and we are calling $\alpha$ to its holonomy representation $\pi_1(X,x)\to Z$ by abuse of notation. 
In Theorem \ref{th-prym-narasimhan-ramanan-character-general} we state that the variety of irreducible and simple representations $\calR_{ss}(X,G)^{\Gamma}$ which are fixed by $\Gamma$ is contained in a union of images of moduli spaces $\calR(Y,G^{\theta}_0,\wgam,\tau,c)$, and this union is contained in $\calR(X,G)^{\Gamma}$. Particular cases of this theorem can be found in Sections \ref{section-prym-narasimhan-ramanan-character-varieties} and \ref{section-character-variety-alpha-trivial}. 


\subsection{Outline of the paper}

\noindent In Section \ref{section-big-higgs-bundles} we introduce the moduli spaces of Higgs bundles, the non-abelian Hodge correspondence and the group action of ($\star$) on $\mdl(X,G)$. In Section \ref{section-lie-theory} we develop some technical group theoretic notions which are necessary throughout the rest of the paper. In Section \ref{section-twisted-equivariant} we define twisted equivariant structures on Higgs pairs and their relation to Higgs pairs with non-connected structure group. Section \ref{section-alpha-trivial} is devoted to the proof of Theorem A and its application to some examples. In Section \ref{section-trivial-eta} we prove Theorems B and C, and apply them to some examples. Finally, in Section \ref{section-general} we treat the general case, first assuming that $\alpha$ is trivial and then proving Theorem D, which is the general result.

As mentioned above, this paper merges results of the 2020 version \cite{oscar-suratno} by the second and third authors with results based on the the 2023 PhD thesis of the first author \cite{barajas-thesis}. 

\subsection{Further developments}

\noindent An important  motivation for this paper is  
the  identification of  hyperk\"ahler and Lagrangian subvarieties of $\cM(X,G)$, which are the support of branes in the context of mirror 
symmetry and Langlands duality as introduced by Kapustin and Witten \cite{kapustin-witten}. 
For example, if the projection of a finite subgroup $\Gamma$ of ($\star$) on $\C^*$ is trivial, the smooth fixed point locus is hyperkähler and so it is the potential support of BBB-branes. However, if $\Gamma$ has order 2 and the corresponding projection on $\C^*$ is non-trivial --- i.e., the generator multiplies the Higgs field by $-1$ ---, the smooth fixed point locus is the potential support of BAA-branes. We expect the Prym--Narasimhan--Ramanan construction to provide examples of fully equipped branes: this has been achieved by Franco--Gothen--Oliveira--Peón-Nieto \cite{franco-branes} in the case of the action of a finite cyclic subgroup of the Jacobian.


There are no instances of supports for branes of type ABA and AAB in this paper. Some of these arise from antiholomorphic involutions of $G$ and $X$, which we plan to study in the future. A study of real $G$-Higgs bundles, which corresponds to $\Gamma=\Z/2\Z$ acting antiholomorphically both on $X$ and $G$, is carried on in \cite{biswas-García-prada,bhp,biswas-calvo-García-Prada}. Other references are Baraglia--Schaposnik \cite{baraglia-schaposnik}, \cite{biswas-García-prada} and \cite{bhp,biswas-García-Prada-hurtubise2,biswas-García-Prada-hurtubise3}.

We expect that our fixed point description may be applied to the study of cohomological invariants associated to automorphisms of moduli spaces of Higgs bundles, via localization. An example of this is \cite{narasimhan-ramanan}, where the Atiyah--Singer fixed point theorem is used to calculate topological invariants of moduli spaces of vector bundles. Very recently, Andersen--Mistegård \cite{hitchin-equivariant-index} have studied the Hitchin equivariant index associated to the action of a cyclic automorphism $f\in\Aut(X)$ on the cohomology of the determinant line bundle over the moduli space of Higgs bundles of rank $2$ and degree 1. This is a topological invariant of the mapping torus of $f$. The analysis of the fixed points of $f$ in the moduli space is crucial in this work. A more explicit calculation of the invariant will appear in future joint work of the first author, joint with Andersen and Mistegård \cite{andersen-barajas}.

Whenever $\mdl(X,G)$ may be realized as a Hitchin fibration, we expect our fixed point descriptions to be 
useful for the study of the Hitchin fibres of different fixed point subvarieties. Some references in this direction are Heller--Schaposnik \cite{heller-schaposnik}, Schaffhauser \cite{schaffhauser} and Schaposnik \cite{schaposnik-thesis}.

We also plan to extend our fixed point description to parabolic Higgs bundles in future work. These correspond, via non-abelian Hodge theory, to representations of the fundamental group of punctured surfaces with fixed conjugacy classes around the punctures \cite{simpson-noncompact,oscar-biquard-mundet}. The parabolic setup is geometrically richer and perhaps more natural from the physics point of view of mirror symmetry \cite{gukov-witten,kapustin-witten}.

Finally, our description of fixed points in moduli spaces of Higgs pairs may be applied to the study of fixed point subvarieties for Higgs bundles associated to real forms, which may lead to results on the topology of higher Teichmüller spaces \cite{cayley}.

\noindent{\bf Acknowledgements}. We wish to thank Peter Gothen and 
S. Ramanan for very useful discussions.

\section{Higgs bundles}\label{section-big-higgs-bundles}

Throughout this section $G$ is a connected semisimple Lie group with centre $Z$ and $X$ is a compact Riemann surface with canonical bundle $K_X$.

\subsection{Moduli space of \texorpdfstring{$G$}{G}-Higgs bundles}\label{section-higgs-bundles}

A $G$-Higgs bundle over $X$ is a pair $(E,\phi)$ where $E$ is a holomorphic principal $G$-bundle $E$ and $\phi$ is a holomorphic section of $E(\mathfrak{g})\otimes K_X$, 
where $E(\mathfrak{g}):=E\times_{G} \mathfrak{g}$ is the associated bundle corresponding to the adjoint action of $G$ on the Lie 
algebra $\mathfrak{g}$. We sometimes denote the adjoint bundle
$E(\mathfrak{g})$ by $\ad(E)$.

Two $G$-Higgs bundles $(E,\phi)$ and $(E',\phi')$ are isomorphic if there is an isomorphism
$f:E\to E'$ such that the induced isomorphism $\ad(f)\otimes \id_{{K_X}}:E(\mathfrak{g})\otimes K_X\to E'(\mathfrak{g})\otimes K_X$
sends $\phi$ to $\phi'$.

Given a $G$-Higgs bundle $(E,\phi)$, each element of the centre $Z\subset G$ determines an automorphism of $(E,\phi)$ via the $G$-action. The $G$-Higgs bundle $(E,\phi)$ is said to be {\bf simple} if $\Aut(E,\phi)= Z$ where $\Aut(E,\phi)$
is the group of Higgs bundle automorphisms of $(E,\phi)$.

Let us recall the definitions of stability, semistability and polystability for $G$-Higgs bundles over $X$. Following \cite{PBI}, fix a maximal compact subgroup $K$ of $G$ with Lie algebra $\lie k$. Choose a $G$-invariant non-degenerate symmetric bilinear form $\langle\cdot,\cdot\rangle$ on $\lie g$. Every element $s\in i\lie k$ determines a parabolic subgroup $P_s$ of $G$, namely
\begin{equation}\label{eq-def-Ps}
    P_s:=\{g\in G\suhthat e^{ts}ge^{-ts}\,\text{remains bounded as}\;t\to\infty\}.
\end{equation}
If $L_s$ is its Levi subgroup then $K_s:=K\cap L_s$ is a maximal compact subgroup of $L_s$ and its inclusion in $P_s$ is a homotopy equivalence. Now let $E$ be a $G$-bundle with a holomorphic reduction $\sigma\in H^0(X,E(G/P_s))$, where $E(G/P_s)$ is the $G/P_s$-bundle associated to $E$ via the $G$-action given by multiplication on the left. We denote by $E_{\sigma}$ the corresponding $P_s$-bundle $\sigma^*(E\to E(G/P_s))$. Then there is a smooth reduction $\sigma'\in\Omega^0(X,E_{\sigma}/K_s)$, and we may equip the corresponding $K_s$-bundle with a connection $A$ with curvature $F_A$. We define
\begin{equation}\label{eq-def-deg}
    \deg E(\sigma,s):=\frac{i}{2\pi}\int_X\chi_s(F_A),
\end{equation}
where $\chi_s$ is the image of $s$ under the isomorphism $\lie g\cong\lie g^*$ induced by the non-degenerate pairing.

\begin{definition}\label{def-stability-higgs-bundle}
A $G$-Higgs bundle $(E,\phi)$ over $X$ is:
\begin{itemize}
    \item \textbf{Semistable} if $\deg E(\sigma,s)\ge 0$ for any $s\in i\lie k$ and any reduction of structure group $\sigma\in H^0(X,E(G/P_s))$ such that $\phi\in H^0(E_{\sigma}(\lie p_s)\otimes K_X)$, where $\lie p_s$ is the Lie algebra of $P_s$.
    \item \textbf{Stable} if $\deg E(\sigma,s)> 0$ for any $s\in i\lie k$ and any reduction of structure group $\sigma\in H^0(X,E(G/P_s))$ such that $\phi\in H^0(E_{\sigma}(\lie p_s)\otimes K_X)$.
    \item \textbf{Polystable} if it is semistable and, if $\deg E(\sigma,s)=0$ for some $s\in i\lie k$ and a reduction $\sigma\in H^0(X,E(G/P_s))$ such that $\phi\in H^0(E_{\sigma}(\lie p_s)\otimes K_X)$, there is a further holomorphic reduction of structure group $\sigma'\in H^0(X,E_{\sigma}(P_s/L_s))$ with $\phi\in H^0(E_{\sigma'}(\lie l_s)\otimes K_X)$, where $\lie l_s$ is the Lie algebra of $L_s$ and $E_{\sigma'}$ is the reduction of $E_{\sigma}$ to $L_s$.
\end{itemize}
\end{definition}


Let  $\cM(X,G)$  be the  {moduli space of isomorphism classes of polystable $G$-Higgs bundles}.
By Schmitt's general Geometric Invariant Theory construction 
\cite{schmitt:2008}, the space
$\cM(X,G)$ has the structure of a quasi-projective variety. Its smooth locus is the open subvariety of stable and simple points $\cM_*(X,G)$. For related  constructions
see \cite{nitsure,simpson:1994,simpson:1995}.
If we fix  the topological class $c$  of $E$ we can consider  
$\cM_c(X,G)\subset \cM(X,G)$, the  moduli space of $G$-Higgs bundles
with fixed topological class $c$. Since $G$ is connected, the topological 
class is given by an element  $c\in \pi_1(G)$.
In this situation it is well-known 
that $\cM_c(X,G)$ is non-empty 
and connected \cite{donagi-pantev,li}.  A Morse-theoretic  proof is given
in \cite{García-Prada-oliveira}, where the connectedness 
and non-emptiness  of  $\cM_c(X,G)$ 
is also proved when $G$ is a non-connected complex reductive Lie group.

\subsection{Higgs bundles and Hitchin equations}\label{section-higgs-bundles-and-hitchin-equations}

As above, let $G$ be a connected semisimple complex Lie group and let $K\subset G$ be a maximal compact subgroup
of $G$. Let $(E,\phi)$ be a $G$-Higgs bundle on $X$. Let $h$ be a reduction of structure group of $E$ from $G$
 to $K$, and $F_h$ be the curvature of  the {Chern--Singer connection} --- the unique connection on $E$ compatible with $h$ and the holomorphic structure of $E$.
 Let $\sigma_h:\Omega^{1,0}(X,E(\lieg))\to \Omega^{0,1}(X,E(\lieg))$
be the $\C$-antilinear map defined by the reduction $h$ and the conjugation 
between $(1,0)$ and $(0,1)$-forms on $X$.
Consider {Hitchin's equation} 
\begin{equation}\label{hitchin-equation}
F_h-[\phi,\sigma_h(\phi)]=0.
\end{equation}
By the works of Hitchin \cite{hitchin1987} and Simpson \cite{simpson}, the following holds. 
\begin{theorem}\label{EH1}
Let $(E,\phi)$ be a $G$-Higgs bundle on $X$.
Then $(E,\phi)$ is polystable if and only if the $G$-bundle $E$ 
admits a reduction of structure group $h$ to $K$ satisfying (\ref{hitchin-equation}). 
\end{theorem}

\subsection{Higgs bundles and representations of the fundamental group}
\label{higgs-reps}

By a {\bf representation} we mean a 
 homomorphism $\rho:\pi_{1}(X)\to G$. The set $\Hom(\pi_{1}(X),G)$ of all such homomorphisms is an analytic variety \cite{goldman}. The group $G$ acts on $\Hom(\pi_{1}(X),G)$ via
$ g\cdot \rho(\gamma):=g\rho(\gamma)g^{-1}$. for $g\in G$, $\rho\in \Hom(\pi_{1}(X),G)$ and $\gamma\in \pi_{1}(X)$.
 A representation  $\rho\in \Hom(\pi_{1}(X),G)$ is called
 {\bf reductive} if its
 composition with the adjoint representation of $G$ in the
 Lie algebra $\mathfrak{g}$  decomposes as a direct sum of irreducible 
representations. If we restrict the action to the subspace $\Hom^{+}(\pi_{1}(X),G)$ consisting of reductive representations, the orbit space 
 is Hausdorff.
 The {moduli space of reductive representations} is defined to be the orbit space
 \[\mathcal{R}(X,G):=\Hom^{+}(\pi_{1}(X),G)/G.\] This orbit space coincides with the GIT quotient
 \[\mathcal{R}(X,G):=\Hom(\pi_{1}(X),G)\sslash G,\]
 hence it has a structure of an affine algebraic variety \cite{richardson}. 
Let $\rho:\pi_{1}(X)\to G$ be a representation.
 Let $Z_{{G}}(\rho)\le G$ be the centralizer of $\rho(\pi_{1}(X))$. We say that $\rho$ is {\bf irreducible} if and only if
 it is reductive and $Z_{{G}}(\rho)=Z$ where $Z$ is the centre of $G$.
 
 Theorem \ref{EH1} states that the polystability of $(E,\phi)$
 is equivalent to a reduction $h$ satisfying Hitchin's equation. Assuming that (\ref{hitchin-equation}) indeed holds, a simple computation shows that, if $\nabla_h$ is the Chern--Singer connection of $h$,
$D=\nabla_h+\phi-\sigma_h(\phi)$
is a flat connection on the  $G$-bundle 
$E$ whose holonomy defines a reductive representation of $\pi_{1}(X)$ in  $G$. By a theorem of Donaldson 
\cite{donaldson} and Corlette \cite{corlette}, it  follows that all reductive representations $\rho:\pi_{1}(X)\to G$ arise in this way. 
More concretely, the
following holds.
 
 \begin{theorem}[Non-abelian Hodge correspondence]\label{rep1}
  Let $G$ be a complex semisimple Lie group. The moduli space $\cM(X,G)$ of polystable 
$G$-Higgs bundles and the moduli space $\mathcal{R}(X,G)$
  of reductive representations of $\pi_{1}(X)$ in $G$ are homeomorphic. Under 
this homeomorphism, the set of irreducible representations in  $\mathcal{R}(X,G)$
is in one-to-one correspondence with the set of stable and simple $G$-Higgs bundles.
 \end{theorem}

\subsection{Group actions on the moduli space}\label{section-action}

Let $(E,\phi)$ be a $G$-Higgs bundle. An automorphism $\theta$ of $G$ provides another $G$-Higgs 
bundle $(\theta(E),\theta(\phi))$, as follows: The bundle $\theta(E)$ is the holomorphic principal $G$-bundle with total space $E$ and $G$-action 
$$E\times G\to E;\,(e,g)\mapsto e\theta^{-1}(g),$$
where we have written the action of $G$ on $E$ by adjoining elements of $G$ on the right. Alternatively, this is just the $G$-bundle obtained from $E$ by the extension of structure group induced by $\theta$. If $\phi$ is locally equal to $(e,v)\otimes k$ for some $v\in \lieg$ and local sections $e$ and $k$ of $E$ and $K_X$ respectively, $\theta(\phi)$ is locally equal to $(e,\theta(v))\otimes k$. This is well defined, since 
\begin{align*}
(e g,\theta(\Ad_{g^{-1}}v))&=(e \theta^{-1}(\theta(g)),\Ad_{\theta(g^{-1})}\theta(v))
\\&=(e\cdot\theta(g),\Ad_{\theta(g^{-1})}\theta(v))
\\&=(e,\Ad_{\theta(g)}\Ad_{\theta(g^{-1})}\theta(v))
\\&=(e,\theta(v)),
\end{align*}
where the presence or absence of the dot denotes the $G$-action on $\theta(E)$ or $E$, respectively. 
Note that this defines a left action of the group $\Aut(G)$ of holomorphic automorphisms of $G$ on the set of $G$-Higgs bundles. We sometimes write $\theta(E,\phi)$ instead of $(\theta(E),\theta(\phi))$. 

Let $\Int(G)$ be the group of inner automorphisms of $G$. There is a surjection
$$\Int:G\to\Int(G);\,g\mapsto\Int_g,$$
whose kernel is the centre $Z$ of $G$, and hence $\Int(G)$ is isomorphic to the adjoint group $G/Z$. The group $\Int(G)$ is a normal subgroup of $\Aut(G)$, hence we can define the quotient $\Out(G):=\Aut(G)/\Int(G)$. The natural left action of $\Aut(G)$ induces a right action of $\outg$ on the set of isomorphism classes of $G$-Higgs bundles as follows: given a class $a\in\outg=\Aut(G)/\Int(G)$, we may choose a representative $\theta$ in $\Aut(G)$ and consider the action on the set of $G$-Higgs bundles sending $(E,\phi)$ to $\theta^{-1}(E,\phi)$. The isomorphism class of $\theta^{-1}(E,\phi)$ is independent of the choice of $\theta$, since multiplication by $g$ induces an isomorphism $(E,\phi)\xrightarrow{\sim}\Int_g(E,\phi)$ for every $g\in G$. 


There is also an action of the group of $Z$-bundles over $X$ on the set of isomorphism classes of $G$-Higgs bundles. Recall that this group is isomorphic to the first cohomology group $H^1(X,Z)$, so we will use this notation throughout the paper. In order to define the action, let $\alpha\in H^1(X,Z)$ and a $G$-Higgs bundle $(E,\phi)$. Construct the fibre product $E\times_X\alpha$ with respect to $X$, which fits in the pullback diagramme
\[\begin{tikzcd}
E\times_X\alpha\arrow{r}\arrow{d} &
E\arrow{d}\\
\alpha\arrow{r} & X
\end{tikzcd}.
\]
This is a $G\times Z$-bundle over $X$. Let $E\otimes\alpha$ be the extension of structure group of $E\times_X\alpha$ by the multiplication homomorphism $G\times Z\to G$.
Since the adjoint action of $Z$ on $\lie g$ is trivial, $\ad(E\otimes\alpha)=\ad(E)$, so that $\phi$ may also be regarded as a section of $\ad(E\otimes\alpha)\otimes K_X$.

Finally, there is an action of the group of holomorphic automorphisms $\Aut(X)$ of $X$ given by pullback, and an action of $\C^*$ given by multiplying the Higgs field. The first one is a right action and the second one is right and left, since $\C^*$ is abelian.

The above four actions fit together into a right action of the group $H^1(X,Z)\rtimes(\outg\times\Aut(X))\times\C^*$ on the set of isomorphism classes of $G$-Higgs bundles on $X$, where the left action of $\outg\times\Aut(X)$ on $H^1(X,Z)$ which defines the semidirect product is given by 
\begin{equation*}
    (\outg\times\Aut(X))\times H^1(X,Z)\to H^1(X,Z);\,((a,\eta),\alpha)\mapsto \eta^{*-1}a(\alpha).
\end{equation*}
Note that $a(\alpha)$ is well defined because $\Int(G)$ acts trivially on $Z$.
Explicitly, 
\begin{equation}\label{eq-action}
    (E,\phi)\cdot(\alpha,a,\eta,\mu):=(\eta^*\theta^{-1}(E\otimes\alpha),\mu \eta^*\theta^{-1}\phi),
\end{equation}
where $\theta\in\Aut(G)$ is any element in the coset $a\in\Out(G)$.
Indeed, given elements $\alpha$ and $\alpha'$ in $H^1(X,Z)$, $a$ and $a'$ in $\Out(G)$, $\eta$ and $\eta'\in\Aut(X)$ and $\mu$ and $\mu'$ in $\C^*$, together with a $G$-Higgs bundle $(E,\phi)$ over $X$,
\begin{align*}
    (E\cdot(\alpha,a,\eta,\mu))\cdot(\alpha',a',\eta',\mu')&= (\eta'^*\theta'^{-1}(\eta^*\theta^{-1}(E\otimes\alpha)\otimes\alpha'),\mu'\mu \eta'^*\theta'^{-1}(\eta^*\theta^{-1}(\phi)))\\&=
((\eta\eta')^*(\theta\theta')^{-1}(E\otimes\alpha\otimes \eta^{*-1}a(\alpha')),\mu'\mu(\theta\theta')^{-1}(\phi))\\&=
E\cdot(\alpha\eta^{*-1}a(\alpha'),aa',\eta\eta',\mu\mu'),
\end{align*}
where $\theta$ and $\theta'$ are automorphisms in the cosets $a$ and $a'$, respectively.
Since this action preserves (poly)stability and simplicity, it induces an action on $\cM(X,G)$ which preserves the locus of simple and stable points $\cM_*(X,G)$.

\begin{remark}\label{remark-induced-isomorphisms}
It is important to notice that, given an isomorphism of $G$-Higgs bundles
$$f:(E,\phi)\xrightarrow{\sim} (F,\psi)$$
and an element $(\alpha,\theta,\eta,\mu)\in H^1(X,Z)\rtimes(\Aut(G)\times\Aut(X))\times\C^*$, we
may define an induced isomorphism
$$\eta^*f\otimes \id:\eta^*\theta^{-1}(E\otimes\alpha,\mu\phi)\to \eta^*\theta^{-1}(F\otimes\alpha,\mu\psi)$$
via the natural biholomorphisms between $E$ and $\theta(E)$, and $F$ and $\theta(F)$.
\end{remark}

\section{Background on group theory}\label{section-lie-theory}

In order to state our fixed point description we first need to introduce some group theoretic notions.

\subsection{Lifts and non-abelian cohomology}\label{section-lifts-and-non-abelian-cohomology}

Let $G$ be a complex Lie group with centre $Z$. We follow \cite{PR} to describe the different equivalence classes of  lifts to 
$\Aut(G)$ of a homomorphism 
$a:\Gamma\to \Out(G)$, in terms of {\bf group cohomology}. 

Let $\Gamma$ be a group and  $B$ 
another group acted on by  $\Gamma$ via a homomorphism
$\theta:\Gamma\to \Aut(B)$, that is, 
every  $\gamma\in \Gamma$ defines an automorphism of $B$ that
we denote by $\theta_\gamma$.
A $1$-{\bf cocycle} 
of $\Gamma$ with values in $B$ 
is a map $\beta:\Gamma\to B$ such that
\begin{equation}\label{cocycle}
\beta_1=1\andd \beta_{\gamma\gamma'}=\beta_\gamma\theta_\gamma(\beta_{\gamma'})\;\; \mbox{for every}\;\;
\gamma,\gamma'\in \Gamma,
\end{equation}
where $\beta_{\gamma}$ denotes the image of $\gamma$.
The set of 1-cocycles is denoted $Z^1_\theta(\Gamma,B)$.
Two cocycles $\beta$ and $\beta'\in Z^1_\theta(\Gamma,B)$ are  said to be  {\bf cohomologous}
if there exists $b\in B$ such that
\begin{equation}\label{eq-cohomologous}
\beta'_\gamma=b^{-1}\beta_\gamma \theta_\gamma(b).
\end{equation}
 This is an equivalence relation in $Z^1_\theta(\Gamma,B)$ and the quotient is 
denoted by $H^1_\theta(\Gamma,B)$. This is the {\bf first cohomology set} of
$\Gamma$ with values in $B$. 

Coming back to our problem, let $S_a$ be the set of homomorphisms  $\theta:\Gamma\to \Aut(G)$ lifting $a:\Gamma\to \Out(G)$, i.e. fitting in the commutative diagramme
\begin{equation}\label{eq-lift-a}
    \begin{tikzcd}
\Aut(G)\arrow[r]  & \Out(G)\\
  & \Gamma\arrow[lu,dotted,"\theta"]\arrow[u,"a"]
\end{tikzcd}.
\end{equation}
It is well known that, if $G$ is a connected compact real Lie group, $S_a$ is non-empty --- see, for example, \cite{BGGM}. This is also true when $G$ is a connected reductive complex Lie group, since then it is a complexification of a compact Lie group.

If we fix one element $\theta\in S_a$ then every other lift $\theta'$ is equal to $\beta\theta$ for some map $\beta:\Gamma\to\Int(G)$. For every $\gamma$ and $\gamma'\in\Gamma$ one has
\begin{equation*}
    \beta_{\gamma\gamma'}\theta_{\gamma\gamma'}=
    \theta'_{\gamma\gamma'}=
    \theta'_{\gamma}\theta'_{\gamma'}=
    \beta_{\gamma}\tg\beta_{\gamma'}\theta_{\gamma'}=
    \beta_{\gamma}\tg(\beta_{\gamma'})\theta_{\gamma\gamma'},
\end{equation*}
where by abuse of notation we regard $\theta$ as an automorphism of $\Int(G)$. Comparing the first and last terms we conclude that $\beta\in Z^1_{\theta}(\Gamma,\Int(G))$. Conversely, such a 1-cocycle provides a lift. Summing up, the following holds. 

\begin{lemma}\label{lemma-lifts-vs-non-abelian-cohomology}
    Given a lift $\theta$ of $a$, there is an $\Int(G)$-equivariant bijection
\begin{equation}\label{eq-bijection-Sa-cocycles}
    \{\text{Lifts of $a$}\}\leftrightarrow Z^1_{\theta}(\Gamma,\Int(G));\,\beta\theta\mapsto\beta,
\end{equation}
where the action on the left hand side is given by conjugation and the action on the right hand side is given by (\ref{eq-cohomologous}).
In particular, it induces a bijection
\begin{equation}\label{eq-bijection-Sa-cohomology-pairs}
    \{\text{Lifts of $a$}\}/\Int(G)\leftrightarrow H^1_{\theta}(\Gamma,\Int(G)).
\end{equation}

\end{lemma}

\subsection{Automorphisms of a semisimple complex Lie group}\label{section-Gtheta}
Let $G$ be a connected semisimple complex Lie group with (finite) centre $Z$. Take a finite group $\Gamma$ and a homomorphism
$$\theta:\Gamma\to\autg;\,\gamma\mapsto\tg.$$
Note that the corresponding $\Gamma$-action on $G$ preserves $Z$.

We define
$$\gt:=\{g\in G\suhthat\tg(g)=g\forevery\gamma\in\Gamma\}\le G$$
and
$$\gs:=\{g\in G\suhthat\tg(g)=z(\gamma,g)g\,\forevery\gamma\in\Gamma\;\text{and some}\;z(\gamma,g)\in Z\}\le G.$$
The group $\Gamma$ clearly acts (trivially) on $\gt$ and, since it acts on $Z\le\gs$, it also acts on $\gs$. There is an exact sequence of groups
\begin{equation}\label{eq-exact-seq-groups}
    1\to \gt\to\gs\xrightarrow{\cct} \zt,
\end{equation}
where the last homomorphism sends $g\in\gs$ to the map
$$\Gamma\to Z;\,\gamma\mapsto g^{-1}\tg(g)=\tg(g)g^{-1}=z(\gamma,g)\in Z.$$
The set $\zt$ of 1-cocycles is defined in Section \ref{section-lifts-and-non-abelian-cohomology}, and in this context it has a group structure induced by the group structure on $Z$.
In particular, $\gt$ is a normal subgroup of $\gs$.
To see why $\cct$ is well defined note that, if $\gamma$ and $\gamma'$ are elements of $\Gamma$, $g\in \gs$ and $\cct(g,\gamma)\in Z$ denotes the image of $g$ evaluated at $\gamma$, then
\begin{align*}
    \cct(g,\gamma\gamma')g=\theta_{\gamma\gamma'}(g)=\tg\left(\theta_{\gamma'}(g)\right)=
\tg(\cct(g,\gamma')g)=\tg(\cct(g,\gamma'))\tg(g)&=\tg(\cct(g,\gamma'))\cct(g,\gamma)g\\&=\cct(g,\gamma)\tg(\cct(g,\gamma'))g.
\end{align*}
The exactness of (\ref{eq-exact-seq-groups}) implies that $\cct$ factors through the quotient 
$$\gamt:=\gs/\gt$$
via an injective homomorphism
$\gamt\hookrightarrow\zt.$
In particular, the finiteness of $\zt$ --- which follows from both $\Gamma$ and $Z$ being finite --- implies that $\gs$ is a finite extension of $\gt$, hence the reductiveness of $\gt$ \cite[Section 3, Proposition 3.6]{onishchik3} is inherited by $\gs$.

On the other hand, if $\fun{A}{B}$ denotes the set of maps from a set $A$ to a set $B$, there is a natural group isomorphism
\begin{equation}\label{eq-iso-cohomology}
    H^1(X,\fun{\Gamma}{Z})\xrightarrow{\sim}\fun{\Gamma}{H^1(X,Z)},
\end{equation}
where the group structures on both sides are induced by that of $Z$.
To define this isomorphism we use the isomorphism 
\begin{equation}\label{eq-iso-rep-fundamental-cohomology}
    H^1(X,A)\xrightarrow{\sim}\homm{\pi_1(X)}{A}
\end{equation}
that exists for any finite abelian group $A$, and let (\ref{eq-iso-cohomology}) be the map induced by the natural isomorphism
$$\homm{\pi_1(X)}{\fun{\Gamma}{Z}}\xrightarrow{\sim}\fun{\Gamma}{\homm{\pi_1(X)}{Z}}.$$
One can see that (\ref{eq-iso-cohomology}) restricts to an isomorphism 
\begin{equation}\label{eq-cohomology-Z-iso-Z-cohomology}
H^1(X,\zt)\xrightarrow{\sim} Z^1_{\theta}(\Gamma,H^1(X,Z)).    
\end{equation}

Recall that, given a complex Lie group $H$, the set of equivalence classes of \v{C}ech 1-cocycles on $X$ with values on the sheaf $\underline{H}$ of holomorphic functions to $H$ is in natural bijection with the set of isomorphism classes of holomorphic $H$-bundles. We denote it by $H^1(X,\underline{H})$. It has a distinguished element --- representing the trivial bundle ---, which gives $H^1(X,\underline{H})$ the structure of a pointed set. Using this notation, the homomorphism $\cct$ induces a morphism of pointed sets
$$H^1(X,\underline{\gs})\to H^1(X,\zt)$$
by extension of structure group.
Composing with (\ref{eq-cohomology-Z-iso-Z-cohomology}), we obtain a morphism of pointed sets
$$\ctt:H^1(X,\underline{\gs})\to Z^1_{\theta}(\Gamma,H^1(X,Z)).$$
Since $\ctt$ factors through $\gamt$, there is a commutative triangle
\begin{equation*}
    \begin{tikzcd}
    H^1(X,\underline{\gs})\arrow[r,"\ctt"]\arrow[d]&Z^1_{\theta}(\Gamma,H^1(X,Z))\\
    H^1(X,\gamt)\arrow[ur].
\end{tikzcd}
\end{equation*}
Moreover, the diagonal map is injective, since it is a composition of an isomorphism and the homomorphism induced by the embedding $\gamt\hookrightarrow Z^1_{\theta}(\Gamma,Z)$. This induced homomorphism is also injective because of (\ref{eq-iso-rep-fundamental-cohomology}), which we can apply because $Z^1_{\theta}(\Gamma,H^1(X,Z))$ is abelian.

The group $\gt$ is connected when $G$ is simply connected and the image of the homomorphism $\theta$ is cyclic \cite[Section 8]{steinberg-endomorphisms-of-linear-algebraic-groups}. However, it is not connected in general even when $G$ is simply connected. For this reason, we need to "refine" (\ref{eq-exact-seq-groups}) using the connected component $\gt_0$ of $\gt$. Due to the fact that $\gs$ is a finite extension of $\gt$ by (\ref{eq-exact-seq-groups}), $\gt_0$ is also the connected component of $\gs$. There is thus an extension
\begin{equation}\label{eq-extension-connected-component}
    1\to\gt_0\to\gs\xrightarrow{\pt}\gamtt\to1,
\end{equation}
where $\gamtt:=\gs/\gt_0$ is a finite group because $\gs$ is reductive. Of course there is a natural surjective homomorphism $\gamtt\to\gamt:=\gs/\gt,$
which induces a morphism $H^1(X,\gamtt)\to H^1(X,\gamt)$. We denote by $\qt$ the composition

\begin{equation}\label{eq-def-qt}
\begin{tikzcd}[ar symbol/.style = {draw=none,"#1" description}]
    H^1(X,\gamtt)\arrow[rrr,bend right=15,"\qt", swap]\arrow[r] & H^1(X,\gamt)\ar[r,hook] & H^1(X,\zt)\arrow[r] & Z^1_{\theta}(\Gamma,H^1(X,Z)).
\end{tikzcd}
\end{equation}

\subsection{Finite extensions of a semisimple complex Lie group}\label{section-extension}

Let $G$ be a connected reductive complex Lie group with (maybe infinite) centre $Z$, and let $\Gamma$ be a finite group. Because of the appearance of non-connected reductive structure groups in the study of fixed points, we need a description of the finite extensions of $\Gamma$ by $G$ --- see \cite{BGGM}. These are reductive complex Lie groups $\hat G$ fitting in a short exact sequence
\begin{equation}\label{eq-extensions}
    1\to G\to \hat G\to\Gamma\to 1.
\end{equation}
There is an equivalence relation on the set of extensions of $\Gamma$ by $G$ which relates two extensions given by groups $\hat G$ and $\hat G'$ if and only if there is an isomorphism $\hat G\xrightarrow{\sim} \hat G'$ fitting in a commutative diagramme
$$
\begin{tikzcd}[ar symbol/.style = {draw=none,"#1" description,sloped},
  isomorphic/.style = {ar symbol={\cong}},
  equals/.style = {ar symbol={=}},
  ]
1\arrow[r]  & G \arrow[r]\ar[equal]{d} & \hat G \arrow[r]\ar[d,"\sim" labl] & \Gamma \arrow[r]\ar[equal]{d} & 1\\
1\arrow[r]  & G \arrow[r] & \hat G' \arrow[r] & \Gamma \arrow[r] & 1
\end{tikzcd}.
$$
To each equivalence class of extensions of $\Gamma$ by $G$ we may associate a characteristic homomorphism 
$$a:\Gamma\to\Out(G);\,\gamma\mapsto \ag.$$
Fix such a homomorphism $a$. We want to parametrize the equivalence classes of extensions with characteristic homomorphism $a.$

Since the action of $\Int(G)$ on $Z$ is trivial, $a$ induces a homomorphism $a:\Gamma\to\Aut(Z).$
In this context we define a \textbf{2-cocycle} of $\Gamma$ with values in $Z$ as a map $c:\Gamma\times\Gamma\to Z$
satisfying
\begin{equation}\label{eq-2-cocycle}
    c(1,\gamma)=c(\gamma,1)=1\andd\ag(c(\gamma',\gamma''))c(\gamma,\gamma'\gamma'')=c(\gamma\gamma',\gamma'')c(\gamma,\gamma')
\end{equation}
for every $\gamma$, $\gamma'$ and $\gamma''$ in $\Gamma$. There is an equivalence relation on the set $Z^2_{a}(\Gamma,Z)$ of 2-cocycles, where two 2-cocycles $c$ and $c'$ are related if and only if there is a map
$$\beta:\Gamma\to Z$$
such that
$$c'(\gamma,\gamma')=c(\gamma,\gamma')\ag(\beta(\gamma'))\beta(\gamma)\beta(\gamma\gamma')^{-1}.$$
The set of equivalence classes $H^2_{a}(\Gamma,Z)$ is the \textbf{second cohomology set} of $\Gamma$ with values in $Z$.
If $\theta:\Gamma\to\Aut(G)$ is a homomorphism lifting $a$, we sometimes write $Z^2_{\theta}(\Gamma,Z)$ and $H^2_{\theta}(\Gamma,Z)$ instead of $Z^2_{a}(\Gamma,Z)$ and $H^2_{a}(\Gamma,Z)$, respectively.

\begin{definition}\label{def-twisted-product}
Given a 2-cocycle $c\in Z^2_{a}(\Gamma,Z)$ and a lift $\theta$ of $a$, we define the \textbf{$(\theta,c)$-twisted product of $G$ by $\Gamma$}, written $G\times_{(\theta,c)}\Gamma$, to be the group that is equal to $G\times\Gamma$ as a set and has multiplication
\begin{equation}\label{eq-def-twisted-product}
    (g,\gamma)(g',\gamma')=(g\tg(g')c(\gamma\gamma'),\gamma\gamma')
\end{equation}
for every $g$ and $g'$ in $G$ and every $\gamma$ and $\gamma'$ in $\Gamma$.
\end{definition}

\begin{remark}
     The equivalence class of the extension
$$    1\to G\to G\times_{(\theta,c)}\Gamma\to\Gamma\to 1$$
only depends on $a$ and the class of $c$ in $H^2_{a}(\Gamma,Z)$. More precisely, given a map $s:\Gamma\to G$ such that $\Int_s\in Z^1_{\theta}(\Gamma,\Int(G))$, there exists a 2-cocycle $c_s\in Z^2_a(\Gamma,Z)$ such that $G\times_{(\theta,c)}\Gamma\cong G\times_{(\Int_s\theta,c_sc)}\Gamma$ as extensions --- recall that $\Int_s\theta$ is a homomorphism by Lemma \ref{lemma-lifts-vs-non-abelian-cohomology}. The 2-cocycle $c_s$ is given by
\begin{equation}\label{eq-def-cs}
    c_s(\gamma,\gamma')=s(\gamma)\tg(s(\gamma'))s(\gamma\gamma')^{-1}.
\end{equation}
See \cite[Section 2.4]{BGGM} for more details.
\end{remark}

The main statement in this section is the following. 

\begin{proposition}\label{prop-extensions-isomorphic-twisted-group}
For every extension $\hat G$ of $G$ by $\Gamma$ with characteristic morphism $a$ and every lift $\theta$ of $a$, there exists a 2-cocycle $c\in Z^2_{a}(\Gamma,Z)$ such that the extensions of $G$ given by $\hat G$ and $G\times_{(\theta,c)}\Gamma$ are equivalent. We may further choose the lift $\theta$ so that the image of $c$ is in any given $\Gamma$-invariant maximal compact subgroup of $Z$.
\end{proposition}
\begin{proof}
Take a section $t:\Gamma\to \hat G$ of (\ref{eq-extensions}) whose composition with the natural homomorphism $\hat G\to\Int(\hat G)$ restricts to $\theta$. Every $g\in \hat G$ can be uniquely written in the form $g_0t_{\gamma}$, where $\gamma\in\Gamma$ is the connected component where $g$ lies and $g_0\in G$. This determines a map $\hat G\to G\times \Gamma$. We can then show that this map determines an isomorphism of extensions $\hat G\cong G\times_{(\theta,c)}\Gamma$ for some map $c:\Gamma\times\Gamma\to Z$. Using associativity of the group multiplication of $\hat G$ it can be seen that $c\in Z^2_{a}(\Gamma,Z)$.

The second assertion follows by choosing a maximal compact subgroup $K$ of $\hat G$ and a lift $\theta:\hat G\to \Aut(G)$ preserving its connected component, together with a section $t$ whose image lies in $K$.
\end{proof}

\subsection{Restrictions of the adjoint representation}\label{section-adjoint-representation} 
Let $G$ be a semisimple complex Lie group and $\lieg$ be its Lie algebra. Take a homomorphism $\theta:\Gamma\to\Aut(G)$, where $\Gamma$ is any finite group. We consider the restriction of the adjoint representation
$$G\to\GL(\lieg)$$
to the subgroups $\gt$ and $\gs$, defined in Section \ref{section-Gtheta}. Given a character
$$\mu:\Gamma\to \C^*;\,\gamma\mapsto \mu_{\gamma},$$
we may consider the $\mu$-weight subspace of $\lie g$, given by
$$\liegm:=\{v\in\lieg\suhthat \tg(v)=\mu_{\gamma} v\forevery\gamma\in\Gamma\}.$$
One can see that it is preserved by the adjoint action of $\gs$: for every $g\in\gs$, $\gamma\in\Gamma$ and $v\in\lieg^{\theta}_{\mu}$, one has
$$\tg\Ad_g(v)=\Ad_{\tg(g)}(\tg(v))=\Ad_{\cct(g,\gamma) g}(\mug v)=\mug\Ad_{g}(v),$$
where $\cct(g,\gamma)$ is some element of $Z$.
Abusing notation, we denote by $\Ad:\gs\to\GL(\liegm)$ the restriction of the adjoint representation.

The following lemma is a crucial ingredient to prove Proposition \ref{prop-polystability-extension-structure-group}.

\begin{lemma}\label{lemma-compact-involution}
Given a homomorphism $\theta:\Gamma\to\Aut(G)$, there exists a compact conjugation $\sigma$ of $G$ preserving each connected component of $\gs$ such that $\gs^{\sigma}$ is a maximal compact subgroup of $\gs$ and, for every character $\mu:\Gamma\to\C^*$,
$\sigma(\lieg^{\theta}_{\mu})=\lieg^{\theta}_{\omu}.$ 
\end{lemma}
\begin{proof}
Let $\liegt:=\sum_{\mu\in \Hom(\Gamma,\C^*)}\liegm$. This is a subalgebra of $\lieg$, since $[\liegm,\lieg^{\theta}_{\mu'}]=\lieg^{\theta}_{\mu\mu'}$ for every pair of homomorphisms $\mu$ and $\mu'\in\Hom(\Gamma,\C^*)$. In fact $\liegt$ is the subalgebra $\lieg^{C}$ of fixed points of $\lieg$ under the action of the commutator $C:=[\theta(\Gamma),\theta(\Gamma)]=\langle\{\theta_{\gamma\gamma'\gamma^{-1}\gamma'^{-1}}\}_{\gamma,\gamma'\in\Gamma}\rangle$. Indeed, note that, for each triple $\gamma,\gamma'$ and $\gamma''\in\Gamma$ and an element $v\in\lieg^C$, one has
$$\theta_{\gamma\gamma'\gamma^{-1}\gamma'^{-1}}\theta_{\gamma''}(v)=\theta_{\gamma''}\theta_{\gamma\gamma'\gamma^{-1}\gamma'^{-1}}(v)=\theta_{\gamma''}(v),$$
so that $\theta(\Gamma)$ acts on $\lieg^C$. The elements of the group of automorphisms $\theta(\Gamma)\vert_{\lie g^C}< \Aut(\lie g^C)$ can be simultaneously diagonalizable, since they commute and have finite order --- in particular, they are semisimple ---, thus giving a decomposition as in the definition of $\liegt$. This shows that $\lieg^C\subseteq\liegt$, and the reverse inclusion is clear. 

By \cite[Proposition 3.6, Chapter 3]{onishchik3}, $\lie g^C$ is reductive. The restriction $\theta(\Gamma)\vert_{\lie g^C}$ of $\theta(\Gamma)$ to $\lie g^C$ is a finite abelian group, therefore it has a decomposition $\theta(\Gamma)\vert_{\lie g^C}=\langle\gamma_1\rangle\times\dots\times\langle\gamma_l\rangle$, where each $\gamma_i$ is an element of $\theta(\Gamma)\vert_{\lie g^C}$ and $\langle\gamma_i\rangle$ is the cyclic subgroup generated by it. Let $\Gamma_i:=\langle\gamma_1\rangle\times\dots\times\langle\gamma_i\rangle$, for each $1\le i\le l$ --- in particular, $\lie g^C_l=\liegt$. Again by \cite[Proposition 3.6, chapter 3]{onishchik3}, all the subalgebras $\lie g^C_i:=\lie g^C\cap\lie g^{\Gamma_i}\subset\lie g^C$ are reductive. Set $\lie g^C_0:=\lie g^C$. Note that $\Ad(\gs)$ commutes with $\theta(\Gamma)$, since
\begin{equation*}
    \tg\Ad_g=\Ad_{\tg(g)}\tg=\Ad_{\cct(g,\gamma)g}\tg=\Ad_{g}\tg
\end{equation*}
for every $\gamma\in\Gamma$ and $g\in\gs$ --- recall that $\cct(g,\gamma):=\tg(g)g^{-1}\in Z$. In particular, $\Ad(\gs)$ preserves $\lie g^C_i$ for each $0\le i\le l$.

Choose a Cartan subalgebra $\lie t^{\theta}\subset \lieg^{\theta}$. It is well known that there exists a lift $\tau:\gamtt\to\Aut(\gt_0)$ of the characteristic homomorphism of \ref{eq-extension-connected-component} which preserves $\lie t^{\theta}$. By Proposition \ref{prop-extensions-isomorphic-twisted-group}, there exists a $2$-cocycle $c\in Z^2_{\tau}(\gamtt,Z(\gt_0))$ such that $\gs\cong\gt_0\times_{(\tau,c)}\gamtt$ as extensions. Moreover, we may assume that the image of $c$ is in the maximal compact subgroup 
$K_{Z(\gt_0)}$ of $Z(\gt_0).$ Throughout this proof we denote by $(1,\delta)$ the element of $\gs$ corresponding to $(1,\delta)\in \gt_0\times_{(\tau,c)}\gamtt$ for each $\delta\in\gamtt$, by abuse of notation.

By the proof of \cite[Theorem 4.5(a)]{borel-invariant-cartan-commuting-automorphisms}, $\lie t^{\theta}$ contains a regular element $v_{l-1}$ of $\lie g^C_{l-1}$. Therefore the centralizer of $\lie t^{\theta}$ in $\lie g^C_{l-1}$, which coincides with the centralizer $Z_{\lie g^C_{l-1}}(v_{l-1})$ of $v_{l-1}$, is a Cartan subalgebra of $\lie g^C_{l-1}$, which we denote by $\lie t^C_{l-1}$. Moreover, since $\lie t^{\theta}$ is preserved by $\tau(\gamtt)$ and $\theta_{\gamma_l}$, its centralizer $\lie t^C_{l-1}\subset \lie g^C_{l-1}$ is preserved by both the automorphism $\theta_{\gamma_l}$ and the family of automorphisms $\{\Ad_{(1,\delta)}\}_{\delta\in\gamtt}$. Again by the proof of \cite[Theorem 4.5(a)]{borel-invariant-cartan-commuting-automorphisms}, $\lie t^C_{l-1}$ contains a regular element $v_{l-2}$ of $\lie g^C_{l-2}$, and so its centralizer $\lie t^C_{l-2}\subset \lie g^C_{l-2}$ is a Cartan subalgebra of $\lie g^C_{l-2}$. The automorphisms $\theta_{\gamma_{l-1}}$ and $\theta_{\gamma_l}$ and the family of automorphisms $\{\Ad_{(1,\delta)}\}_{\delta\in\gamtt}$ preserve $\lie t^C_{l-2}$, since they preserve $\lie t^C_{l-1}$. Continuing this argument we construct a Cartan subalgebra $\lie t^C\subset\lie g^C$ invariant by the action of $\Gamma$ and the family of automorphisms $\{\Ad_{(1,\delta)}\}_{\delta\in\gamtt}$. Using $\lie t^C$, we define the compact involution $\sigma$ as follows.

Let $\Lambda$ be the set of roots of $\lie g^C$, given by the adjoint action of $\lie t^C$, and choose a system of positive roots $\Lambda^+$. Consider the root space decomposition
\begin{equation}\label{eq-root-space-decomposition-g^C}
    \lie g^C=\lie t^C\oplus\bigoplus_{\lambda\in\Lambda^+}\lie g_{\lambda}^C\oplus \lie g_{-\lambda}^C.
\end{equation}
Let $G^C_0$ be the connected reductive subgroup of $G$ with lie algebra $\lie g^C$. Recall that there is a family of compact conjugations $\sigma$ of $G^C_0$ associated with $\lie t$. Indeed, we may first define a holomorphic involution as follows: the reductive Lie algebra $\lie g^C$ can be described using $\mathfrak{sl}_2$-triples $(x_{\lambda},t_{\lambda},x_{-\lambda})$, where $\lambda$ is a positive root, $t_{\lambda}\in\lie t$ and $x_{\lambda}\in\lie g^C_{\lambda}$. These triples are unique up to the action of $\C^*$ such that $c\in\C^*$ sends $(x_{\lambda},t_{\lambda},x_{-\lambda})$ to $(cx_{\lambda},t_{\lambda},c^{-1}x_{-\lambda})$. Once these $\mathfrak{sl}_2$-triples are chosen, there is a unique holomorphic involution of $\lie g^C$ which sends $(x_{\lambda},t_{\lambda},x_{-\lambda})$ to $(-x_{-\lambda},-t_{\lambda},-x_{\lambda})$. Composing with the antiholomorphic involution which fixes $x_{\lambda}$ and $t_{\lambda}$ for every root $\lambda$, we obtain a compact conjugation which we denote by $\sigma$. 

We now claim that, for every $\gamma\in\Gamma$ and every positive root $\lambda\in\Lambda^+$, we may choose the $\mathfrak{sl}_2$-triples $(x_{\lambda},t_{\lambda},x_{-\lambda})$ so that $\tg\vert_{\lie g^C}$ and the family of automorphisms $\{\Ad_{(1,\delta)}\vert_{\lie g^C}\}_{\delta\in\gamtt}$ commute with $\sigma$. Once we have proved this, for every $v\in\liegm$ and every $\gamma\in\Gamma$,
\begin{equation*}
\tg\sigma(v)=\sigma\tg(v)=\sigma\mug(v)=\overline{\mug}\sigma(v),
\end{equation*}
which proves that $\sigma(\liegm)=\lieg^{\theta}_{\omu}$. Moreover, the group generated by the elements $(1,\delta)$ is contained in $K_{Z(\gt_0)}\times_{(\tau,c)}\gamtt$, where $K_{Z(\gt_0)}$ is the maximal compact subgroup of $Z(\gt_0)$. Since the family $\{\Ad_{(1,\delta)}\vert_{\lie g^C}\}_{\delta\in\gamtt}$ commutes with $\sigma$, it preserves the maximal compact subgroup $K^C:=(G^C_0)^{\sigma}$ of $G^C_0$. Therefore, the subgroup of $G$ generated by $K^C$ and the elements $(1,\delta)$ for $\delta\in\gamtt$ is equal to $\bigcup_{\delta\in\gamtt}K^C(1,\delta)$, which is compact. By construction, it contains the maximal compact subgroup $\bigcup_{\delta\in\gamtt}K^{\theta}(1,\delta)$ of $\gs$, where $K^{\theta}:=(\gt)^{\sigma}$. Extending this to a maximal compact subgroup of $G$, so that $\sigma$ is extended to a compact involution of $G$, finishes the construction.

In order to prove the claim, consider the compact group $H:=\Gamma\times (K_{Z(\gt_0)}\times_{(\tau,c)}\gamtt)$. The action of $\Gamma$ via $\theta$ and the adjoint action of $K_{Z(\gt_0)}\times_{(\tau,c)}\gamtt\subset \gs$ on $\lie g^C$ commute, therefore they combine into an action of $H$ on $\lie g^C$ preserving $\lie t^C$. This action must commute the roots. Hence, for each root $\lambda$, each element $x_{\lambda}$ of the root space $\lie g^C_{\lambda}$ and each $h\in H$, the automorphism determined by $h$ sends the element $x_{\lambda}$ to $k_{h,\lambda}x_{h(\lambda)}$, where $k_{h,\lambda}\in\C^*$ and $h(\lambda)$ is the image of $\lambda$ under the automorphism induced by $h$. Moreover, since $h(t_{\lambda})=t_{h(\lambda)}$, then
\begin{equation*}
    t_{h(\lambda)}=h([x_{\lambda},x_{-\lambda}])= [k_{h,\lambda}x_{h(\lambda)},k_{h,-\lambda}x_{-h(\lambda)}],
\end{equation*}
and so $k_{h,-\lambda}=k_{h,\lambda}^{-1}$. Since $H$ is compact, whenever $h(\lambda)=\lambda$, $k_{h,\lambda}\in\U(1)$. Therefore, we may choose the triples $(x_{\lambda},t_{\lambda},x_{-\lambda})$ so that $k_{h,\lambda}\in \U(1)$. Indeed, we may choose a representative $\lambda$ in each $H$-orbit and pick $x_{h(\lambda)}$ to be any element in $\U(1)h(x_{\lambda})$ for each $h\in H$. With these choices,
\begin{align*}
    \sigma(h(x_{\lambda}))&=\sigma(k_{h,\lambda}x_{h(\lambda)})
    \\&=\overline{k_{h,\lambda}}\sigma(x_{h(\lambda)})
    \\&=
    -k_{h,\lambda}^{-1}x_{-h(\lambda)}
    \\&=-k_{h,-\lambda}x_{h(-\lambda)}
    \\&=h(\sigma(x_{\lambda})),
\end{align*}
which shows that the restrictions of $H$ and $\sigma$ to $\bigoplus_{\lambda\in\Lambda^+}\lie g_{\lambda}^C\oplus \lie g_{-\lambda}^C$ commute. It is clear that they also commute on $\lie t^C$. Therefore, by (\ref{eq-root-space-decomposition-g^C}), they commute on the whole $\lie g^C$ as required.

\end{proof}

\section{Higgs pairs and twisted equivariant structures}\label{section-twisted-equivariant}

Throughout this section $G$ is a connected reductive complex Lie group with Lie algebra $\lie g$, and $\Gamma$ is a finite group.

\subsection{Twisted \texorpdfstring{$(G,\Gamma)$}{(G,Gamma)}-actions}\label{section-twisted-G-Gamma-action}

Choose a category $\mathcal{C}$, e.g. sets with maps, topological spaces with continuous maps, smooth manifolds with smooth maps or algebraic varieties with algebraic morphisms. Assume that there is notion of $G$-action in this category, for example in the case of sets we may use the forgetful functor that sends an algebraic group into the underlying abstract group. 

Take a homomorphism $\theta:\Gamma\to\Aut(G)$. Let $Z\subset G$ be an abelian subgroup --- for example, the centre --- and let $c\in Z^2_{a}(\Gamma,Z)$ be a 2-cocycle, defined as in Section \ref{section-extension}. Following \cite[Section 2.3]{BGGM}, we give the following definition.

\begin{definition}\label{def-twisted-(G,Gamma)}
    Let $M$ be an object of $\mathcal{C}$. A \textbf{right $(\theta,c)$-twisted $(G,\Gamma)$-action} on $M$ is a right $G$-action on $M$, together with a map
    \begin{equation}\label{eq-twisted-action-M}
        M\times\Gamma\to M;\,(m,\gamma)\mapsto m\bullet\gamma
    \end{equation}
    satisfying \begin{equation}\label{eq-twisted-equivariant-axioms}
(mg)\bullet\gamma=(m\bullet\gamma)\theta_{\gamma^{-1}}(g)\andd(m\bullet\gamma)\bullet\gamma'=(mc(\gamma,\gamma'))\bullet(\gamma\gamma')
\end{equation}
for every $m\in M$, $g\in G$ and $\gamma$ and $\gamma'\in\Gamma$. The object $M$ equipped with the twisted action is called a $(\theta,c)$-twisted $(G,\Gamma)$-equivariant object. If the $G$-action is implicit, the map (\ref{eq-twisted-action-M})
is also called a $(\theta,c)$-twisted $\Gamma$-action on $M$.

A \textbf{left} $(\theta,c)$-twisted $(G,\Gamma)$-action on $M$ is a $G$-action, together with a set of automorphisms of $M$ labeled by $\Gamma$, whose inverses determine a right $(\theta,c)$-twisted $(G,\Gamma)$-action on $M$. Unless otherwise stated, we always work with right twisted actions in this paper.

A \textbf{morphism} between $(\theta,c)$-twisted $(G,\Gamma)$-equivariant objects $M$ and $M'$ is a morphism $M\to M'$ such that the diagramme
  \begin{equation*}
      \begin{tikzcd}
        M\arrow[r,"\bullet\gamma"]\arrow[d] & M\arrow[d]\\
        M'\arrow[r,"\bullet\gamma"] & M'
      \end{tikzcd}
  \end{equation*}
  commutes for each $\gamma\in\Gamma$.
\end{definition}

\subsection{Twisted equivariant \texorpdfstring{$G$}{G}-bundles}
\label{section-twisted-equivariant-bundle} 

This section follows \cite[Section 3.1]{BGGM}.
Let $X$ a compact Riemann surface with canonical bundle $K_X$ and let $\Gamma$ be a finite group equipped with a homomorphism 
\begin{equation*}
    \eta:\Gamma\to\Aut(X);\,\gamma\mapsto\etag.
\end{equation*}
Note that $\eta$ induces a right $\Gamma$-action on $X$, namely
$x\cdot\gamma:=\etag^{-1}(x)$ for each $x\in X$ and $\gamma\in\Gamma$.
Take two homomorphisms $a:\Gamma\to\Out(G)$ and $\theta:\Gamma\to\Aut(G)$ such that $\theta$ is a lift of $a$. Let $Z\subset G$ be an abelian subgroup and let $c\in Z^2_{a}(\Gamma,Z)$ be a 2-cocycle, defined as in Section \ref{section-extension}.

\begin{definition}\label{def-twisted-equivariant-bundles}
A \textbf{$(\theta,c)$-twisted $\Gamma$-equivariant $G$-bundle} over $X$ is a $G$-bundle $E$ equipped with a $(\theta,c)$-twisted $\Gamma$-action compatible with its action on $X$, i.e. fitting in the commutative diagramme
  \begin{equation*}
      \begin{tikzcd}
        E\arrow[r,"\bullet\gamma"]\arrow[d] & E\arrow[d]\\
        X\arrow[r,"\cdot\gamma"] & X
      \end{tikzcd}
  \end{equation*}
  for each $\gamma\in\Gamma$. The twisted action $\bullet$ is called a (right) \textbf{$(\theta,c)$-twisted $\Gamma$-equivariant action} on $E$. An isomorphism of $(\theta,c)$-twisted $\Gamma$-equivariant $G$-bundles is an isomorphism of $G$-bundles which is also an isomorphism of $(\theta,c)$-twisted $(G,\Gamma)$-equivariant objects, according to Section \ref{section-twisted-G-Gamma-action}.

  Given a $(\theta,c)$-twisted $\Gamma$-equivariant $G$-bundle $E$ over $X$ and an element $\gamma$ in the centre $Z(\Gamma)$ of $\Gamma$, there is an induced $(\theta,c)$-twisted $\Gamma$-equivariant structure on $\gamma^*E$ so that each $\gamma'\in\Gamma$ sends $\gamma^*e\in\gamma^*E$ to $\gamma^*(e\bullet\gamma')$. This descends to the same $\Gamma$-action on $X$, since $\gamma$ commutes with $\gamma'$. We define a \textbf{$Z(\Gamma)$-isomorphism} between two twisted equivariant bundles $E$ and $E'$ to be an isomorphism of twisted equivariant $G$-bundles $E\to\gamma^*E'$ for some $\gamma\in Z(\Gamma)$. 
\end{definition}

\begin{remark}\label{remark-theta-theta'-twisted}
    By \cite[Proposition 2.16]{BGGM}, the category of $(\theta,c)$-twisted $\Gamma$-equivariant $G$-bundles and the category of $(\Int_s\theta,c_sc)$-twisted $\Gamma$-equivariant $G$-bundles are equivalent whenever $s:\Gamma\to G$ is a map such that $\Int_s\in Z^1_{\theta}(\Gamma,\Int(G))$. Here $c_s$ is given by (\ref{eq-def-cs}).
\end{remark}

\begin{remark}\label{remark-exact-sequence-aut-no-higgs}
Let $E$ be a $G$-bundle over $X$. Let $\Aut(E)$ be the group of automorphisms of $E$ covering the 
identity of $X$, and let 
$\Aut_{\Gamma,\eta,\theta}(E)$ be the group of biholomorphic maps 
$f:E\to E$, so that $f$ covers the automorphism $\eta^{-1}_\gamma:X\to X$ for some $\gamma\in \Gamma$ and satisfies that $f(eg)=f(e)\tg^{-1}(g)$ for each $e\in E$. Given $f$ and $f'\in \Aut_{\Gamma,\eta,\theta}(E)$, we define their group product by $ff'=f'\circ f$.
There is  an exact
sequence of groups
\begin{equation}\label{exact-aut-no-higgs}
1\to \Aut(E)\to \Aut_{\Gamma,\eta,\theta}(E) \to \Gamma,
\end{equation}
where the group multiplication on $\Gamma$ is transposed, i.e. the product of $\gamma$ and $\gamma'\in\Gamma$ is $\gamma'\gamma$.
A $(\theta,c)$-twisted $\Gamma$-equivariant structure on $E$ is simply
a $c$-twisted lift of (\ref{exact-aut-no-higgs}), i.e. a map $\Gamma\to \Aut_{\Gamma,\eta,\theta}(E)$ satisfying the second equation of (\ref{eq-twisted-equivariant-axioms}) for $M=E$.
\end{remark}

\subsection{Twisted equivariant Higgs pairs}\label{section-twisted-equivariant-higgs-pairs}

This Section follows \cite[Section 3.2]{BGGM} and \cite{oscar-suratno}. We keep the notation of Section \ref{section-twisted-equivariant-bundle}. Let $V$ be a complex vector space equipped with a complex representation
$$\rho:G\to\GL(V).$$
Given a $G$-bundle $E$ over $X$, there is an associated vector bundle $E(V):=E\times_{\rho}V$. Let $K_X$ be the canonical bundle of $X$.

\begin{definition}\label{def-higgs-pair}
A \textbf{$(G,V)$-Higgs pair over $X$} is a pair $(E,\phi)$, where $E$ is a holomorphic principal $G$-bundle and $\phi$ is a holomorphic section of $E(V)\otimes K_X$. Given two $(G,V)$-Higgs pairs $(E,\phi)$ and $(F,\psi)$ over $X$, a morphism
$(E,\phi)\to(F,\psi)$
is a holomorphic homomorphism of $G$-bundles from $E$ to $F$ whose induced morphism
$E(V)\otimes K_X\to F(V)\otimes K_X$
sends $\phi$ to $\psi$.
\end{definition}

Consider a finite group $\Gamma$ equipped with a homomorphism $\eta:\Gamma\to \Aut(X)$. Take a homomorphism 
$$\theta:\Gamma\to\Aut(G);\,\gamma\mapsto\tg$$
and a 2-cocycle $c\in Z^2_{\theta}(\Gamma,Z)$, defined as in Section \ref{section-extension}. Fix a map $\rhog:\Gamma\to\GL(V)$. We call $\rhog$ a \textbf{$(\theta,c)$-twisted representation} if the pair $(\rho,\rhog)$ determines a left $(\theta,c)$-twisted $(G,\Gamma)$-action on $V$, according to Definition \ref{def-twisted-(G,Gamma)}.
If $\theta$ or $c$ is trivial, we may omit it from the notation. We sometimes denote the induced right $(\theta,c)$-twisted $\Gamma$-action on $V$ by
\begin{equation*}
    V\times\Gamma\to V;\,(v,\gamma)\mapsto v\bullet\rhog(\gamma):=\rhog(\gamma)^{-1}v.
\end{equation*}

\begin{remark}
    Note that $\rhog$ is not an actual group representation in general, but it does induce a linear representation $(\rho,\rhog):G\times_{\theta,c}\Gamma\to\GL(V)$.
\end{remark}

Let $E$ be a $(\theta,c)$-twisted $\Gamma$-equivariant $G$-bundle over $X$. The vector bundle $E(V)\otimes K_X$ is then $\Gamma$-equivariant with $\Gamma$-action defined for each $\gamma\in\Gamma$ by
\begin{equation}\label{eq-def-action-on-higgs-field}
    E(V)\otimes K_X\ni(e,v)\otimes k\mapsto (e\bullet\gamma,v\cdotv\gamma)\otimes(\etag^*k).
\end{equation}

\begin{remark}\label{remark-associated-equivariant}
    In general, given a $(\theta,c)$-twisted $\Gamma$-equivariant $G$-bundle $E$ over $X$ and a manifold $M$ equipped with a left $(\theta,c)$-twisted $(G,\Gamma)$-action, the associated fibre bundle $E(M)\to X$ inherits a genuine $\Gamma$-equivariant action, given by $(e,m)\cdot\gamma:=(e\bullet\gamma,\rhog(\gamma)^{-1}m)$, where $e\in E$, $m\in M$, $\gamma\in\Gamma$ and $\rhog$ denotes the twisted $\Gamma$-action on $M$ --- see \cite[Proposition 3.16]{BGGM}. 
\end{remark}

\begin{definition}\label{def-twisted-equivariant-higgs-pairs}
A \textbf{$(\theta,c,\rhog)$-twisted $\Gamma$-equivariant $(G,V)$-Higgs pair over $X$} is a $(G,V)$-Higgs pair $(E,\phi)$ over $X$ equipped with a $(\theta,c)$-twisted $\Gamma$-equivariant action on $E$ such that $\phi$ is invariant with respect to the $\Gamma$-action (\ref{eq-def-action-on-higgs-field}). We sometimes write $(E,\bullet,\phi)$ to emphasize the twisted action $\bullet$.
If $V=\lie g$, $\rho=\Ad$ and $\rhog=\theta$, we say that $(E,\bullet,\phi)$ is a \textbf{$(\theta,c)$-twisted $\Gamma$-equivariant $G$-Higgs bundle}, thus omitting $\rhog$ from the notation.

There is a notion of $Z(\Gamma)$-isomorphism as in Section \ref{section-twisted-equivariant-bundle}, namely a $Z(\Gamma)$-isomorphism of twisted equivariant bundles sending the first Higgs field to the second one. 
\end{definition}

\begin{remark}\label{remark-exact-sequence-aut}
Remark \ref{remark-exact-sequence-aut} is also valid for Higgs pairs if we impose that all the morphisms involved preserve the Higgs field. Indeed, let $(E,\phi)$ be a $(G,V)$-Higgs pair.  Let $\Aut(E,\phi)$ be the group of automorphisms of $(E,\phi)$ covering the 
identity of $X$, and let 
$\Aut_{\Gamma,\eta,\theta}(E,\phi)$ be the group of biholomorphic maps 
$f:E\to E$ preserving $\phi$ --- according to some twisted representation $\rhog$ of $\Gamma$ on $V$ ---, so that $f$ covers the automorphism $\eta^{-1}_\gamma:X\to X$ for some $\gamma\in \Gamma$ and satisfies that $f(eg)=f(e)\tg^{-1}(g)$ for each $e\in E$. If we equip $\Aut_{\Gamma,\eta,\theta}(E,\phi)$ with the same group structure as in Remark \ref{remark-exact-sequence-aut}, then there is  an exact
sequence of groups
\begin{equation}\label{exact-aut}
1\to \Aut(E,\phi)\to \Aut_{\Gamma,\eta,\theta}(E,\phi) \to \Gamma,
\end{equation}
where the group multiplication on $\Gamma$ is transposed.
A $(\theta,c,\rhog)$-twisted $\Gamma$-equivariant structure on $E$ may be regarded as
a $c$-twisted lift of (\ref{exact-aut}), i.e. a map $\Gamma\to \Aut_{\Gamma,\eta,\theta}(E,\phi)$ satisfying the second equation of (\ref{eq-twisted-equivariant-axioms}) for $M=E$.
\end{remark}

By Remark \ref{remark-theta-theta'-twisted} we may assume that the lift $\theta$ is chosen so that it preserves a maximal compact subgroup $K$ of $G$ with Lie algebra $\lie k$. We may define notions of (poly, semi)stability by changing slightly Definition \ref{def-stability-higgs-bundle}: first note that, if $s$ is in the $\Gamma$-invariant part $i\lie k^{\Gamma}$ of $i\lie k$, then the corresponding parabolic subgroup $P_s$ is $\Gamma$-invariant and so is its Levi subgroup $L_s$, since
\begin{equation}\label{eq-def-levi}
    L_s=\{g\in G\suhthat \lim_{t\to\infty} e^{ts}ge^{-ts}=g\}.
\end{equation}
Therefore, $\theta$ induces $(\theta,c)$-twisted $(G,\Gamma)$-actions on $G/P_s$ and $P_s/L_s$, where $G$ acts by left multiplication and $\Gamma$ acts via $\theta$ --- note however that the respective actions of the centre $Z$ are both trivial, since $Z$ is contained in $L_s\subset P_s$, and so this is actually a twisted action with trivial 2-cocycle. If $E$ is a $(\theta,c)$-twisted $\Gamma$-equivariant bundle, then there is a $\Gamma$-equivariant action on the fibre bundle $E(G/P_s)$ \cite[Section 3.2]{BGGM}. This lets us introduce the corresponding space of $\Gamma$-invariant reductions of structure group $H^0(X,E(G/P_s))^{\Gamma}$. Given such a reduction $\sigma$, the corresponding $P_s$-bundle $E_{\sigma}$ inherits a $(\theta,c)$-twisted $\Gamma$-equivariant action, since $Z$ is contained in $P_s$. Moreover, since $Z$ is contained in $L_s$, there is also a $\Gamma$-equivariant associated bundle $E_{\sigma}(P_s/L_s)$ with a corresponding space of $\Gamma$-invariant reductions $H^0(X,E_{\sigma}(P_s/L_s))^{\Gamma}$.

For each $s\in i\lie k$ we may define
$$V_s:=\{v\in V\suhthat \rho(e^{ts})v\,\text{remains bounded as}\;t\to\infty\}$$
and
\begin{equation*}
    V^0_s=\{v\in V\suhthat \lim_{t\to\infty} \rho(e^{ts})v=v\}.
\end{equation*}
Given a reduction $\sigma\in H^0(X,E(G/P_s))$, we may define a subbundle $E(V)_{\sigma,s}:=E_{\sigma}\times_{P_s}V_s\subseteq E(V)$. Given a further reduction $\sigma'\in H^0(X,E_{\sigma}(P_s/L_s))$, there is also an associated subbundle $E(V)^0_{\sigma',s}:=E_{\sigma'}\times_{L_s}V^0_s\subseteq E(V)_{\sigma,s}$.

\begin{definition}\label{definition-stable-twisted-equivariant}
A $(\theta,c,\rhog)$-twisted $\Gamma$-equivariant $(G,V)$-Higgs pair $(E,\phi)$ over $X$ is:
\begin{itemize}
    \item \textbf{Equivariantly semistable} if $\deg E(\sigma,s)\ge 0$ for any $s\in i\lie k^{\Gamma}$ and any reduction of structure group $\sigma\in H^0(X,E(G/P_s))^{\Gamma}$ such that $\phi\in H^0(X,E(V)_{\sigma,s}\otimes K_X)$.
    
    \item \textbf{Equivariantly stable} if $\deg E(\sigma,s)> 0$ for any $s\in i\lie k^{\Gamma}$ and any reduction of structure group $\sigma\in H^0(X,E(G/P_s))^{\Gamma}$ such that $\phi\in H^0(X,E(V)_{\sigma,s}\otimes K_X)$.
    
    \item \textbf{Equivariantly polystable} if it is equivariantly semistable and, provided that $\deg E(\sigma,s)=0$ for some $s\in i\lie k^{\Gamma}$ and a reduction $\sigma\in H^0(X,E(G/P_s))^{\Gamma}$ with $\phi\in H^0(X,E(V)_{\sigma,s}\otimes K_X)$, there is a further reduction of structure group $\sigma'\in H^0(X,E_{\sigma}(P_s/L_s))^{\Gamma}$ with $\phi\in H^0(X,E(V)^0_{\sigma',s}\otimes K_X)$.
\end{itemize}
\end{definition}

\begin{remark}
    In general, the stability notions for twisted equivariant Higgs pairs depend on a parameter $\zeta\in i\zk^{\Gamma}$, where $\zk$ is the centre of $\lie k$. However, in this paper we are only concerned with the case $\zeta=0$.
\end{remark}

\begin{remark}
    When $V=\lie g$, $\rho$ is the adjoint representation of $G$ on $\lie g$ and $\rhog=\theta$, then $E(V)_{\sigma,s}=E_{\sigma}(\lie p_s)$ and $E(V)^0_{\sigma',s}=E_{\sigma'}(\lie l_s)$ in Definition \ref{definition-stable-twisted-equivariant}. 
\end{remark}

We consider the moduli space $\cM(X,G,\Gamma,\theta,c,V,\rhog)$ classifying isomorphism classes of equivariantly polystable $(\theta,c,\rhog)$-twisted $\Gamma$-equivariant $(G,V)$-Higgs pairs. 
If $\Gamma$ is trivial, this coincides with the moduli space of $(G,V)$-Higgs pairs over $X$, which we denote by $\cM(X,G,V)$. If $V=\lie g$, $\rho$ is the adjoint representation of $G$ on $\lie g$ and $\rhog=\theta$, we obtain the moduli space of \textbf{$(\theta,c)$-twisted $\Gamma$-equivariant $G$-Higgs bundles}, which we write $\cM(X,G,\Gamma,\theta,c)$. 
Given a character 
\begin{equation*}
    \mu:\Gamma\to\C^*;\,\gamma\mapsto\mug
\end{equation*}
of $\Gamma$, there is an alternative $(\theta,c)$-twisted action of $\Gamma$ on $\lie g$ given by $\rhomu:=\omu_{\gamma}\theta$. It is clear that the following categories are equivalent:
\begin{enumerate}
    \item The category of \textbf{$(\theta,c,\rhomu)$-twisted $\Gamma$-equivariant $G$-Higgs bundles}.
    \item The category of \textbf{$(\theta,c,\mu)$-twisted $\Gamma$-equivariant $G$-Higgs bundles}, whose objects are pairs $(E,\phi)$ consisting of a $G$-bundle $E$ equipped with a $(\theta,c)$-twisted $\Gamma$-equivariant action $\bullet$ and a Higgs field $\phi\in H^0(X,E(\lie g)\otimes K_X)$ fitting in the following diagramme:
    $$
    \begin{matrix}\label{higa}
    E(\mathfrak{g})\otimes K_X & \stackrel{(\bullet,\tg^{-1})\otimes\etag^{*}}{\longrightarrow} & 
    E(\mathfrak{g})\otimes K_X\\
    ~\Big\uparrow\varphi && ~\,\text{  }~\,\text{ }\Big\uparrow \mug\varphi\\
    X & \stackrel{\etag^{-1}}{\longrightarrow} & X
    \end{matrix},
    $$
    and whose morphisms are $\Gamma$-equivariant homomorphisms of $G$-bundles sending one Higgs field to the other. Here, by abuse of notation, $(\bullet,\tg^{-1})$ denotes the equivariant group action of $\Gamma$ on $E(\lie g)$.
\end{enumerate}
Via this equivalence of categories, we may construct a moduli space $\cM(X,G,\Gamma,\theta,c,\mu)$ of $(\theta,c,\mu)$-twisted $\Gamma$-equivariant $G$-Higgs bundles.

The notion of (poly, semi)stability for $(\theta,c,\rhog)$-twisted $\Gamma$-equivariant $(G,V)$-Higgs pairs restricts to a corresponding notion for $(\theta,c)$-twisted $\Gamma$-equivariant $G$-bundles, which may be regarded as Higgs bundles with zero Higgs field.
There is a moduli space classifying isomorphism classes of equivariantly polystable $(\theta,c)$-twisted $\Gamma$-equivariant $G$-bundles over $X$, which we denote by $M(X,G,\Gamma,\theta,c)$.

\subsection{Twisted equivariant Higgs pairs and Hitchin equations}\label{section-twisted-equivariant-higgs-pairs-and-hitchin-equations}

With notation as in Section \ref{section-twisted-equivariant-higgs-pairs} let $K\subset G$ be a $\Gamma$-invariant maximal compact subgroup
of $G$ with Lie algebra $\lie k$, where $\Gamma$ acts on $G$ via $\theta$. {\color{red}Choose an embedding $G\hookrightarrow\GL(N,\C)$ such that $K$ is embedded in $U(N)$, and take the induced $G$-invariant hermitian metric on $\lie g$, which restricts to an isomorphism $\lie k\cong\lie k^*$.} Let $(E,\phi)$ be a $(\theta,c,\rhog)$-twisted $\Gamma$-equivariant $(G,V)$-Higgs pair on $X$. Let $h\in\Omega^0(X,E(G/K))^{\Gamma}$ be a $\Gamma$-equivariant smooth reduction of structure group of $E$ from $G$
to $K$, where $\Gamma$ acts equivariantly on $E(G/K)$ by Remark \ref{remark-associated-equivariant}. Let $F_h$ be the curvature of  the corresponding Chern--Singer connection. Choose a hermitian metric $h_V$ on $V$ such that $\rho(K)$ is contained in the group $U(V)$ of unitary automorphisms of $V$.
 Let $\sigma_{h,h_V}:\Omega^{1,0}(X,E(V))\to \Omega^{0,1}(X,E(V))$
be the $\C$-antilinear map defined by the metrics $h$ and $h_V$, and the conjugation 
between $(1,0)$ and $(0,1)$-forms on $X$. Let $\rho^*:E_h(\lie u(V))^*\to E_h(\lie k)^*\cong E_h(\lie k)$ be the dual of the infinitesimal action of $K$ on $V$ given by $\rho$, where $\lie u(V)$ is the Lie algebra of $U(V)$. Let $\Lambda:\Omega^2(X)\to\Omega^0(X)$ be the adjoint of wedging with the K\"ahler of volume 1 on $X$.
Consider the equation
\begin{equation}\label{hitchin-equation-pairs}
\Lambda F_h-\rho^*(\frac{i}{2}\phi\otimes\sigma_{h,h_V}(\phi))=0,
\end{equation}
where we identify $\frac{i}{2}\phi\otimes\sigma_{h,h_V}(\phi)$ as a skew symmetric section of $\End(E(V)\otimes K)^*=\End(E(V))^*$, hence a section of $E_h(\lie u(V))^*$. 
The following Hitchin--Kobayashi correspondence holds --- see \cite{hitchin1987,simpson,PBI}). 
\begin{theorem}\label{EH1-equivariant}
Let $(E,\bullet,\phi)$ be a $(\theta,c,\rhog)$-twisted $\Gamma$-equivariant $(G,V)$-Higgs pair on $X$.
Then $(E,\phi)$ is equivariantly polystable if and only $E$ 
admits a $\Gamma$-invariant reduction of structure group $h$ to $K$ satisfying (\ref{hitchin-equation-pairs}). 
\end{theorem}

\begin{remark}\label{remark-hitchin-eqs-equivalent}
    If $V=\lie g$, $\rho$ is the adjoint representation of $G$ and $h_V=h$, then (\ref{hitchin-equation-pairs}) is equivalent to (\ref{hitchin-equation}).
\end{remark}

\subsection{Local structure at isotropy points}\label{section-isotropy}

Following previous sections, take homomorphisms $\eta:\Gamma\to \Aut(X)$ and $\theta:\Gamma\to\Aut(G)$, and a 2-cocycle $c\in Z^1_{\theta}(\Gamma,Z)$.
Let $x\in X$, and \[\Gamma_{x}:=\{\gamma\in \Gamma \mid \eta_{\gamma}(x)=x\}\] be the corresponding isotropy subgroup. 
Let $\PPP:=\{x\in X \mid \Gamma_{x}\neq \{1\}\}$.

It is well-known that, when $\Gamma$ acts on $X$ faithfuly and properly 
discontinuously,
$\PPP$ consists of a finite number of points $\{x_{1},\cdots,x_{r}\}$ 
and, for each $x_{i}\in \PPP$, 
the isotropy subgroup $\Gamma_{x_i}\subset \Gamma$ is cyclic
\cite[Section III, Proposition 3.1]{miranda}. Each $x_i\in\PPP$ is called an \textbf{isotropy point} of $X$.

Choose a complex representation $\rho:G\to\GL(V)$ and a $(\theta,c)$-twisted representation $\rhog:\Gamma\to\GL(V)$ --- defined in Section \ref{section-twisted-equivariant-higgs-pairs}. Let $(E,\bullet,\phi)$ be a $(\theta,c,\rhog)$-twisted $\Gamma$-equivariant $(G,V)$-Higgs pair with bundle projection $p_E:E\to X$. Following \cite{oscar-suratno}, we study the local twisted $\Gamma$-action on the fibre of $E$ over each isotropy point. A study of these local actions for trivial 2-cocycle $c$ can be found in \cite{damiolini-local types}.

The underlying $(\theta,c)$-twisted $\Gamma$-equivariant
structure on $E$ determines the following. For each $x\in \PPP$ and $e\in E$ such that $p_E(e)=x$, there is a map 
\[\sigma_{e}:\Gamma_{x}\to G\]
defined by \[e\bullet \gamma=e\tg^{-1}(\sigma_{e}(\gamma))^{-1}.\]

Following \cite{oscar-suratno}
\begin{proposition}\label{local-data}
The following equations hold:

 \begin{enumerate}
  \item $\sigma_{e}(\gamma\gamma')=c(\gamma,\gamma')\sigma_{e}(\gamma)\theta_{\gamma}(\sigma_{e}(\gamma'))$\forevery $\gamma,\gamma'\in \Gamma_x$.
  \item $\sigma_{eg}(\gamma)=g^{-1}\sigma_{e}(\gamma)\theta_{\gamma}(g)$ 
  for all $\gamma\in \Gamma_x$ and $g\in G$.
 \end{enumerate}
\end{proposition} 
 \begin{proof}
  For each $e\in E$ the equation $e\bullet(\gamma\gamma')=e\theta_{\gamma\gamma'}^{-1}(\sigma_e(\gamma\gamma'))^{-1}$ holds by definition of $\sigma_e$, whereas 
   \begin{equation*}
\begin{split}
(e\bullet\gamma)\bullet\gamma' & = (e\theta_{\gamma}^{-1}(\sigma_e(\gamma))^{-1})\bullet\gamma'\\
 & = (e\bullet\gamma')\theta_{\gamma'}^{-1}(\theta_{\gamma}^{-1}(\sigma_e(\gamma)))^{-1} \\
 & = e\theta_{\gamma'}^{-1}(\sigma_e(\gamma'))^{-1}\theta_{\gamma\gamma'}^{-1}(\sigma_e(\gamma))^{-1}\\
 & = e\theta_{\gamma\gamma'}^{-1}(\sigma_e(\gamma)\theta_{\gamma}(\sigma_e(\gamma')))^{-1}   \\
\end{split}
\end{equation*}
By (\ref{eq-twisted-equivariant-axioms}) $(e\bullet(\gamma\gamma'))\theta_{\gamma\gamma'}^{-1}(c(\gamma,\gamma'))=(e\bullet\gamma)\bullet\gamma'$,
therefore
\[e\theta_{\gamma\gamma'}^{-1}(\sigma_e(\gamma\gamma'))^{-1}\theta_{\gamma\gamma'}^{-1}(c(\gamma,\gamma'))
=e\theta_{\gamma\gamma'}^{-1}(\sigma_e(\gamma)\theta_{\gamma}(\sigma_e(\gamma')))^{-1} \;\; \mbox{for every}\;\; e\in E.\]
Hence, 
$\sigma_e(\gamma\gamma')=c(\gamma,\gamma')\sigma_e(\gamma)\theta_{\gamma}(\sigma_e(\gamma'))$.

To check $(2)$, note that 
$$ eg\tg^{-1}(\sigma_{eg}(\gamma))^{-1}=(eg)\bullet\gamma=(e\bullet\gamma)\tg^{-1}(g)=
e\tg^{-1}(\sigma_e(\gamma))^{-1}\theta_{\gamma}^{-1}(g),
$$
hence $g\tg^{-1}(\sigma_{eg}(\gamma))^{-1}=\tg^{-1}(\sigma_e(\gamma))^{-1}\theta_{\gamma}^{-1}(g)$ and the equation follows after taking inverses, multiplying by $g$ on the right and applying $\tg$.
  \end{proof}

For each isotropy point $x\in X$ , let us denote by $Z^1_{\theta,c_x}(\Gamma_x,G)$ the set of all maps $\beta:\Gamma_{x}\to G$ satisfying 
$\beta(\gamma_1\gamma_2)=c_x(\gamma_1,\gamma_2)\beta(\gamma_1)\theta_{\gamma_1}(\beta(\gamma_2))$ for all $\gamma_1$ and $\gamma_2\in\Gamma$, 
where $c_x$ is the $2$-cocycle induced by the restriction of $c$ to $\Gamma_x$. 
We say that such a map $\beta$ is a a $c_x$-twisted $1$-cocycle for the action of $\Gamma_x$ on $G$ given by  $\theta$. 
Two twisted $1$-cocycles $\beta$ and $\beta'$ are related if there exists $g\in G$ 
such that $\beta'=g^{-1}\beta\theta_{\gamma}(g)$. 
We denote the set of all $c_x$-twisted $1$-cocycles modulo the above defined relation by $H^1_{\theta,c_x}(\Gamma_x,G)$.



Let $(E,\phi)$ be a $(\theta,c,\rhog)$-twisted $\Gamma$-equivariant $(G,V)$-Higgs pair. 
By Proposition \ref{local-data}, for each  $x\in \PPP$ there is an associated equivalence class $\sigma_x$ of $c_x$-twisted $1$-cocycles, which only depends on the $\Gamma$-equivariant isomorphism class of 
$E$.

We fix a class $\sigma_{x_i}\in H^1_{\theta,c_{x_i}}(\Gamma_{x_i},G)$ for each  $x_{i}\in \PPP$ and write $\sigma:=\{\sigma_{x_i}\}_i$. We say that a 
$(\theta,c,\rhog)$-twisted $\Gamma$-equivariant 
$(G,V)$-Higgs pair $(E,\phi)$ has 
{\bf local type} $\sigma_{x_i}\in H^1_{\theta,c_x}(\Gamma_x,G)$ at an isotropy point $x_{i}\in X$ 
if the class of twisted $1$-cocycles induced by the 
$(\theta,c)$-twisted $\Gamma$-equivariant structure on $E$ 
is equal to the class $\sigma_{x_i}
$.

We define  $\cM(X,G,\Gamma,\theta,c,V,\rhog,\sigma)$ as 
the subvariety of the moduli space 
of twisted equivariant Higgs pairs $\cM(X,G,\Gamma,\theta,c,V,\rhog)$
with fixed local types $\sigma_{x_i}$, $i=1,\cdots, r$. When $V=\lie g$, $\rho$ is the adjoint representation and $\rhog=\theta$ (i.e., when we consider twisted equivariant Higgs bundles) we omit $V$ and $\rhog$. If we are considering the moduli space of $(\theta,c,\mu)$-twisted $\Gamma$-equivariant $G$-Higgs bundles --- defined in the last paragraphs of Section \ref{section-twisted-equivariant-higgs-pairs} ---, we write $\cM(X,G,\Gamma,\theta,c,\mu,\sigma)$.

Let $c,c'\in Z^2_\theta(\Gamma,Z)$ be two cohomologous $2$-cocycles, that is, there exists a map $s:\Gamma\to Z$ such that 
\begin{equation}\label{com}
c'(\gamma,\gamma')=c(\gamma,\gamma')s_{\gamma}\theta_{\gamma}(s_{\gamma'})s_{\gamma\gamma'}^{-1}\quad\text{for every}\;\gamma\;\text{and}\;\gamma'\in \Gamma. \end{equation}
Let $\sigma_{x_i}\in H^1_{\theta,c_{x_i}}(\Gamma_{x_i},G)$. Then $\sigma^s_{x_i}:=s^{-1}\sigma_{x_i}$ is an element of $ H^1_{\theta,c'_{x_i}}(\Gamma_{x_i},G)$, where we are also calling $s$ to its restriction to $\Gamma_{x_i}$ by abuse of notation. Indeed,
\begin{align*}
    s^{-1}_{\gamma\gamma'}\sigma_{x_i}(\gamma\gamma')
    &=s_{\gamma}\tg(s_{\gamma'})s^{-1}_{\gamma\gamma'}c(\gamma,\gamma')s_{\gamma}^{-1}\sigma_{x_i}(\gamma)\tg(s_{\gamma'})^{-1}\tg(\sigma_{x_i}(\gamma'))\\&
    =c'(\gamma,\gamma')\sigma_{x_i}^s(\gamma)\tg(\sigma_{x_i}^s(\gamma')),
\end{align*}
where we denote by $\sigma_{x_i}$ a representative of the comohology class by abuse of notation.

\begin{proposition}\label{cohom}
 If $c$ and $c'$ are
cohomologous cocycles in $Z^2_\theta(\Gamma,Z)$ with $c'(\gamma,\gamma')=c(\gamma,\gamma')s_{\gamma\gamma'}^{-1}s_{\gamma}\tg(s_{\gamma'})$, $\rhog^{s}:=\rho(s)\rhog$ and $\sigma^s=s^{-1}\sigma$,
$$
{\cM}(X,G,\Gamma,\theta,c,V,\rhog,\sigma)
\cong {\cM}(X,G,\Gamma,\theta,c',V,\rhog^{s},\sigma^s).
$$
\begin{proof}
 Let $(E,\bullet,\phi)\in {\cM}(X,G,\Gamma,\theta,c,V,\rhog,\sigma)$.
We can show that $(E,\phi)$ admits a $(\theta,c',\rhog^s)$-twisted $\Gamma$-equivariant structure 
with local 
types $\{\sigma^s_{x_i}\}$ namely define
$e*\gamma:=(es_{\gamma})\bullet\gamma$
 as in the proof of \cite[Proposition 2.16]{BGGM}. By the same proposition this provides a $(\theta,c',\rhog^{s})$-twisted $\Gamma$-equivariant action on $(E,\phi)$. Moreover,
\begin{align*}
    e*\gamma
    =(e\bullet\gamma)\tg^{-1}(s_{\gamma})
    =e\tg^{-1}(\sg^{-1}\sigma_{e}(\gamma))^{-1}
    =e\tg^{-1}(\sigma^s_{e}(\gamma))^{-1},
\end{align*}
hence the local type of $(E,*,\phi)$ is $\{\sigma^s_{x_i}\}$.
Similarly, we can construct an inverse.
\end{proof}
\end{proposition}

\subsection{Non-connected groups and twisted equivariant structures}\label{section-equivalence-of-categories}
This section and the next one follow \cite[Section 4]{BGGM}. There is a very explicit relation between principal bundles with non-connected structure group and twisted equivariant principal bundles with connected structure group which is crucial in the Prym--Narasimhan--Ramanan construction and we explain next. Let $G$ be the connected component of the identity of a reductive group $\hat{G}$, $Z$ the centre of $G$ and $\Gamma:=\hat{G}/G$ the group of connected components of $\hat{G}$. There is a short exact sequence
\begin{equation}\label{eq-general-extension}
    1\to G\to \hat{G}\to\Gamma\to 1.
\end{equation}
Let $a:\Gamma\to \Out(G)$ be the characteristic homomorphism of (\ref{eq-general-extension}) and pick a lift $\theta:\Gamma\to \Aut(G)$. 

By Proposition \ref{prop-extensions-isomorphic-twisted-group}, there exists a 2-cocycle $c\in Z^2_{a}(\Gamma,Z)$ such that the extensions of $G$ given by $\hat{G}$ and $G\times_{(\theta,c)}\Gamma$ are equivalent. Abusing notation slightly, we consider the groups $\hat{G}$ and $G\times_{(\theta,c)}\Gamma$ to be equal. Let $E$ be a connected $\hat{G}$-bundle over $X$ and set $Y:=E/G$. This is a connected principal $\Gamma$-bundle over $X$, and hence it may be regarded as a connected étale cover
$\pa:\xa\to X$ with Galois group $\Gamma$. Assume that $E$ is connected, so that $Y$ is also connected.

Let $\rho:\hat{G}\to \GL(V)$ be a complex representation. By restriction we obtain a representation $\rho:G\to\GL(V)$. There is also a natural $(\theta,c)$-twisted representation $\rhog:=\rho\circ t:\Gamma\to\GL(V)$ --- see Section \ref{section-twisted-equivariant-higgs-pairs} for the definition of twisted representation ---, where the map $t:\Gamma\to\hat G$ sends $\gamma$ to $(1,\gamma)$.

\begin{proposition}[Higgs pair version of Proposition 4.5 in \cite{BGGM}]\label{prop-twisted-equivariant-bundles-one-to-one}
Let $Y\to X$ be a connected étale cover with Galois group $\Gamma$. The category of $(\hat{G},V)$-Higgs pairs $(E,\phi)$ over $X$ satisfying
$E/G\cong Y$
and the category of $(\theta,c,\rhog)$-twisted $\Gamma$-equivariant $(G,V)$-Higgs pairs over $\xa$ equipped with $Z(\Gamma)$-isomorphisms are equivalent --- see Definition \ref{def-twisted-equivariant-higgs-pairs}.
\end{proposition}
\begin{proof}[Sketch of the proof]
    Given a $\hat{G}$-bundle $E$ over $X$ such that $E/G\cong Y$, the morphism $E\to E/G\cong \xa$ is a $G$-bundle projection. Moreover, using the map $\Gamma\to\hat G$ sending $\gamma$ to $(1,\gamma)$, each element of $\Gamma$ is associated to a complex automorphism of the total space of $E$ which is clearly compatible with the $\Gamma$-bundle action on $Y$. This can be seen to be a $(\theta,c)$-twisted $\Gamma$-equivariant action.
\end{proof}

\begin{remark}
Note that the equivalence of categories is not true if we replace the category of twisted equivariant Higgs pairs with $Z(\Gamma)$-isomorphisms with the subcategory of twisted equivariant Higgs pairs with usual (fibre-preserving) isomorphisms, since an automorphism of a $\hat{G}$-bundle $E$ on $X$ may not induce the identity on $E/G$.
\end{remark}

Now assume that we are given an embedding $\hat{G}\hookrightarrow\GL(n,\C)$ such that $\rho:\hat{G}\to\GL(V)$ is a subrepresentation of the restriction of the adjoint representation $\Ad:\GL(n,\C)\to\GL(\gl(n,\C))$ to $\hat G$. Then, a $(\hat{G},V)$-Higgs pair $(E,\phi)$ over $X$ has an associated vector Higgs bundle $(E(\C^n),\phi)$, where $\phi$ is regarded as a section of $\ad E(\C^n)\otimes K_X$ via the embedding $E(V)\subset\End E(\C^n)$. Using this setup, Proposition \ref{prop-twisted-equivariant-bundles-one-to-one} may be stated in terms of associated Higgs (vector) bundles, as follows.

\begin{proposition}\label{prop-associated-bundle-equivariant}
Let $Y\to X$ be a connected étale cover with Galois group $\Gamma$.
Take a $(\hat{G},V)$-Higgs pair $(E,\phi)$ over $X$ such that $E/G\cong Y$. Let $(F,\bullet,\psi)$ be the corresponding $(\theta,c)$-twisted $\Gamma$-equivariant $G$-bundle over $\xa$ via Proposition \ref{prop-twisted-equivariant-bundles-one-to-one}, where $Y$ is regarded as an étale cover of $X$. Then there is an equivariant action of $\Gamma$ on the associated vector bundle $F(\C^n)$ preserving the corresponding Higgs field $\psi$, such that $(e,v)\bullet\gamma=(e\bullet\gamma,(1,\gamma)^{-1}v)$ for each $(e,v)\in F(V)$ and $\gamma\in\Gamma$. Moreover, 
$$(F(\C^n),\psi)/\Gamma\cong (E(\C^n),\phi).$$

Conversely, given a $\Gamma$-equivariant action on a Higgs vector bundle $(W,\psi)$ over $\xa$, there is an induced $\Gamma$-equivariant action $*$ on its bundle of frames $P$. For each $\gamma\in\Gamma$ there is an automorphism of $P$ given by
\begin{equation}\label{eq-dot-vs-*}
    p\bullet\gamma:=p*\gamma\bullet (1,\gamma)
\end{equation}
for each $p\in P$. The map
$$P\times\Gamma\to P;\,(p,\gamma)\mapsto p\bullet\gamma$$
is an $(\Int_{t},c)$-twisted $\Gamma$-equivariant action on $P$, where $t:\Gamma\to \hat{G}\subset \GL(n,\C)$ is the map sending $\gamma$ to $(1,\gamma)$. If $(W,\psi)=(F(\C^n),\psi)$ for some $(G,V)$-Higgs pair $(F,\psi)$, and $F\bullet\gamma=F\subset P$ for every $\gamma\in\Gamma$, the restriction of $\bullet$ to $(F,\psi)$ is a $(\theta,c)$-twisted $\Gamma$-equivariant action inducing --- via the construction of the previous paragraph --- the original equivariant action on $(W,\psi).$
\end{proposition}
\begin{proof}
Same as the proof of \cite[Proposition 2.9]{prym-narasimhan-ramanan}, taking account of the Higgs fields.
\end{proof}

\subsection{Twisted equivariant bundles and monodromy}\label{section-monodromy}

We keep the notation of the previous section but we do not assume that $Y$ or $E$ are connected. Let $E$ be a (not necessarily connected) $\hat G$-bundle over $X$. Let
$\pa:\xa\to X$ be a connected component of $E/G$. All the other connected components are isomorphic to $\xa$ via the $\Gamma$-action. Alternatively, $\xa$ is a subbundle of $E/G$ with minimal structure group, which we denote by $\gam$. Indeed, any two different connected reductions of structure group are in different components of $E/G$ because $\Gamma$ is finite. Hence, the image of $\pi_1(X)$ under any monodromy representation $\pi_1(X)\to\Gamma$ of $Y$ is a subgroup of $\Gamma$ conjugate to $\gam$.

There is a subgroup 
$\ga:=G\times_{(\theta,c)}\gam$ of $\hat{G}$, where we have also called $\theta$ and $c$ to their restrictions to $\gam$ and $\gam\times\gam$ respectively. There is a subbundle $E'\subset E$ with structure group $G_Y$, since $E/\ga = (E/G)/\gam$ and so the section of $H^0(X,(E/G)/\gam)$ corresponding to $Y$ may be regarded as a section of $H^0(X,E/\ga)$. Note that $E'/G$ is isomorphic to $\xa$. 

There is also a subvariety $\mdl_{Y}(X,\hat{G})$ of the moduli space $\mdl(X,\hat{G})$ consisting of $(\hat{G},V)$-Higgs pairs $(E,\phi)$ such that $E/G$ is isomorphic to the extension of structure group of $Y$ to $\Gamma$.

\begin{proposition}\label{prop-polystability-extension-of-structure-group-non-connected}
    Let $(E,\phi)$ be a $(\ga,V)$-Higgs pair over $X$ with extension of structure group $(E_{\hat{G}},\phi_{\hat{G}})$ to ${\hat{G}}$. Then $(E,\phi)$ is polystable if and only if $(E_{\hat{G}},\phi_{\hat{G}})$ is polystable. In particular, there exists an extension of structure group morphism $\mdl(X,\ga)\to\mdl(X,{\hat{G}})$.
 \end{proposition}
 \begin{proof}
     Assume that $(E,\phi)$ is polystable. First we prove that $(E_{\hat{G}},\phi_{\hat{G}})$ is semistable: consider a reduction of structure group $\sigma_{\hat{G}}\in H^0(X,E_{\hat{G}}({\hat{G}}/P_s))$, where $s\in i\lie k^{\Gamma}$. Since $s$ is $\Gamma$-invariant the parabolic subgroup $P_s\le {\hat{G}}$ intersects every connected component of ${\hat{G}}$, hence the fibre of the total space of the reduction $(E_{\hat{G}})_{\sigma_{\hat{G}}}$ over $x\in X$ intersects every connected component of $E_{\hat{G}}\vert_x$ and each intersection is isomorphic to $P_s\cap G$. Hence, since $E\vert_x$ is a union of connected components of $E_{\hat{G}}\vert_x$, its intersection with $P_s$ must be isomorphic to $\ga\cap P_s=P'_s$, where $P'_s:=\{g\in \ga\suhthat e^{ts}ge^{-ts}\,\text{remains bounded as}\;t\to\infty\}$ is the parabolic subgroup of $G_Y$ defined by (\ref{eq-def-Ps}). Therefore, the intersection of $(E_{\hat{G}})_{\sigma_{\hat{G}}}$ with $E$ is a reduction of structure group of $E$ to $P'_s\le P_s$, which we denote by $\sigma\in H^0(X,E(\ga/P'_s))$. Moreover,
     \begin{equation*}
         \deg E_{\hat{G}}(\sigma_{\hat{G}},s)=\deg E(\sigma,s)\ge 0,
     \end{equation*}
     where the inequality follows from semistability of $E$ and the equation follow from (\ref{eq-def-deg}) and the fact that a connection on $E$ induces a connection on $E_{\hat{G}}$ --- so that the corresponding curvatures are equal. 

     Now assume that $\phi_{\hat{G}}\in H^0(X,(E_{\hat{G}})_{
    \sigma_{\hat{G}}}(\lie p_s)\otimes K_X)$ and $\deg E_{\hat{G}}(\sigma_{\hat{G}},s)=0$ or, equivalently, $\phi\in H^0(X,E_{
    \sigma}(\lie p'_s)\otimes K_X)$ and $\deg E(\sigma,s)= 0$. By polystability of $E$ there is a further holomorphic reduction of structure group $\sigma'\in H^0(X,E_{\sigma}(P'_s/L'_s))$ such that $\phi\in H^0(X,E_{
    \sigma'}(\lie l'_s)\otimes K_X)$, where $L'_s$ is the Levi subgroup of $P'_s$. Hence the extension of structure group of $E_{\sigma'}$ to $L_s$ yields a reduction $\sigma'_{\hat{G}}\in H^0(X,(E_{{\hat{G}}})_{\sigma_{\hat{G}}}(P_s/L_s))$ such that $\phi_{\hat{G}}\in H^0(X,(E_{\hat{G}})_{
    \sigma'_{\hat{G}}}(\lie l_s)\otimes K_X)$, as required.

     The converse can be shown using the reverse argument: assume that $(E_{\hat{G}},\phi_{\hat{G}})$ is polystable. For semistability consider a reduction of structure group $\sigma\in H^0(X,E(\ga/P'_s))$ such that $\phi\in H^0(X,E_{
    \sigma}(\lie p'_s)\otimes K_X)$, where $s\in i\lie k^{\Gamma}$. Then, extending structure group, we obtain a reduction $\sigma_{\hat{G}}\in H^0(X,E_{\hat{G}}({\hat{G}}/P_s))$ such that $\phi_{\hat{G}}\in H^0(X,(E_{\hat{G}})_{
    \sigma_{\hat{G}}}(\lie p_s)\otimes K_X)$ and so we conclude that 
     \begin{equation*}
         \deg E(\sigma,s)=\deg E_{\hat{G}}(\sigma_{\hat{G}},s)\ge 0.
     \end{equation*}
     If equality holds, we may find a further reduction $\sigma'_{\hat{G}}\in H^0(X,(E_{\hat{G}})_{\sigma_{\hat{G}}}(P_s/L_s))$ such that $\phi_{\hat{G}}\in H^0(X,(E_{\hat{G}})_{
    \sigma'_{\hat{G}}}(\lie l_s)\otimes K_X)$, and the intersection of $(E_{\hat{G}})_{\sigma_{\hat{G}}}$ with $E$ yields a reduction $\sigma'\in H^0(X,E_{\sigma}(P'_s/L'_s))$ such  that $\phi\in H^0(X,E_{
    \sigma'}(\lie l'_s)\otimes K_X)$.
 \end{proof}

\begin{proposition}\label{prop-action-centralizer}
    Let $Z_{\Gamma}(\gam)$ be the centralizer of $\gam$ in $\Gamma$. Let $\isoc(\xa,G,\gam,\theta,c,V,\rhog)$ be the set of isomorphism classes of $(\theta,c,\rhog)$-twisted $\gam$-equivariant $(G,V)$-Higgs pairs over $\xa$.
    There is a natural left action of $Z_{\Gamma}(\gam)$ on $\isoc(\xa,G,\gam,\theta,c,V,\rhog)$ given as follows: take a $(\theta,c,\rhog)$-twisted $\gam$-equivariant $(G,V)$-Higgs pair $(E,\bullet,\phi)$ over $\xa$ and an element $z\in Z_{\Gamma}(\gam)$. 
    \begin{itemize}
        \item There is another $(G,V)$-Higgs pair $(\theta_z(E),\theta_z(\phi))$ given by extension of structure group.
        \item There is a natural $(\theta,c,\rhog)$-twisted $\gam$-action on $\theta_z(E)$ given by 
        \begin{equation}\label{eq-action-centralizer-on-twisted-equivariant}
            e*\gamma:=[ec(z^{-1},z)^{-1}c(z^{-1},\gamma)c(z^{-1}\gamma,z)]\bullet\gamma,
        \end{equation}
        leaving $\theta_z(\phi)$ invariant.
    \end{itemize}
    Via Proposition \ref{prop-twisted-equivariant-bundles-one-to-one}, this induces an action of $\cent$ on the set of isomorphism classes $H^1_{\xa}(X,\ga,V)$ of $(\ga,V)$-Higgs pairs $(E_{\ga},\phi_{\ga})$ over $X$ such that $E_{\ga}/G\cong \xa$, given by extension of structure group by $\Int_{(1,z)}$ for each $z\in \cent$.
\end{proposition}
\begin{proof}
By \cite[Proposition 2.10]{prym-narasimhan-ramanan}, the twisted equivariant bundle $(\theta_z(E),*)$ corresponds, via Proposition \ref{prop-twisted-equivariant-bundles-one-to-one}, to the $\ga:=G\times_{(\theta,c)}\gam$-bundle $\Int_{(1,z)}(E_{\ga})$ over $X$, where $E_{\ga}$ is the $G_Y$-bundle associated to $(E,\bullet)$. The induced Higgs field on $E_{\ga}$ is $\Int_{(1,z)}(\phi_{\ga})$, which corresponds to $\theta_z(\phi)$ on $E$. This proves the last statement of the proposition, and the rest follows from the fact that extension of structure group defines a left action of $\cent$ on the set of isomorphism classes of $(\ga,V)$-Higgs pairs.
\end{proof}

\begin{proposition}\label{prop-prym-narasimhan-ramanan}
Let $\cC_1$ be the category of $(\theta,c,\rhog)$-twisted $\gam$-equivariant $(G,V)$-Higgs pairs over $\xa$. Denote by $\cC_2$ the category of $G_Y$-bundles over $X$ whose quotient by $G$ is isomorphic to $Y$, and let $\cC_3$ be the category of ${\hat{G}}$-bundles over $X$. The composition of the equivalence of categories $\cC_1\to\cC_2$ of Proposition \ref{prop-twisted-equivariant-bundles-one-to-one} with the extension of structure group functor $\cC_2\to\cC_3$ provides a functor $\cC_1\to\cC_3$. With the notation of Proposition \ref{prop-action-centralizer}, this induces a bijection
\begin{equation}\label{eq-narasimhan-ramanan-iso}
    \isoc(\xa,{G},\gam,\theta,c,V,\rhog)/Z_{\Gamma}(\gam)\xrightarrow{\sim} H^1_{Y}(X,{\hat{G}},V),
\end{equation}
where  $H^1_{Y}(X,{\hat{G}},V)$ is the set of isomorphism classes of $({\hat{G}},V)$-Higgs pairs $(E,\phi)$ over $X$ such that $E/G$ is isomorphic to the extension of structure group of $Y$ to $\Gamma$. This restricts to a bijection
\begin{equation}\label{eq-narasimhan-ramanan}
    \mdl(\xa,{G},\gam,\theta,c,V,\rhog)/Z_{\Gamma}(\gam)\xrightarrow{\sim} \mdl_{Y}(X,{\hat{G}},V).
\end{equation}
\end{proposition}

\begin{proof}
Same as the proof of \cite[Proposition 2.11]{prym-narasimhan-ramanan} after taking account of the Higgs fields.
\end{proof}

\subsection{Non-abelian Hodge correspondence for twisted equivariant Higgs bundles}\label{section-non-abelian-hodge-twisted-equivariant}

In this section we assume that $\eta:\Gamma\to\Aut(X)$ is injective. An equivariant base point is a $\Gamma$-equivariant map $x:\Gamma \to X$, 
where $\Gamma$ acts on itself by multiplication, and on $X$ via $\eta$.
Suppose that $(X,x)$ has a universal $\Gamma$-equivariant covering 
$(\widehat{X},\widehat{x})\to (X,x)$. By the universal property of such 
covering, the group of automorphisms of the equivariant covering $\widehat{X}$
over $X$ is uniquely determined by $X$, up to unique isomorphism. This group is
called  the {\bf equivariant fundamental group} of $X$ with respect to 
the action of $\Gamma$ and is denoted by $\pi_1(X,\Gamma, x)$ or simply
$\pi_1(X,\Gamma)$  \cite[Definition 3.1]{Huis}. 

Let $1\in \Gamma$ be the identity, set $x_1:=x(1)$ and let
 $\pi_1(X,x_1)$ be the fundamental group
of $X$ with base point $x_1$.
By \cite[Proposition 3.2]{Huis}, $\pi_1(X,\Gamma,x)$ fits into an exact 
sequence
 \begin{equation}\label{fundu} 
1\to \pi_1(X,x_1)\to \pi_1(X,\Gamma,x)\to \Gamma\to 1.
\end{equation}
 Note that if $\Gamma$ acts trivially on $X$ --- i.e. if $\eta$ is trivial ---, 
then $\pi_1(X,\Gamma,x)=\pi_1(X,x)$, and if $\Gamma$ acts
 freely on $X$ then $\pi_1(X,\Gamma,x)=\pi_1(X/\Gamma,\bar{x})$, where 
$\bar{x}$ is 
the  composite map of $x$ and the quotient $X\to X/\Gamma$  \cite[Proposition 3.4]{Huis}. 

Choose $x\in X$ not fixed by any  $\gamma\in \Gamma$ with $\gamma\neq 1$.
 By the description of the equivariant fundamental group in terms of equivariant loops \cite[Section 6]{Huis}, we can 
 identify $\pi_1(X,\Gamma,x)$ with the set of all homotopy classes of maps $\sigma:[0,1]\to X$ such that $\sigma(0)=x$ and 
 $\sigma(1)\in \{\eta_{\gamma}(x)\mid \gamma\in \Gamma\}$. 
Under this identification, the surjective map 
 $p: \pi_1(X,\Gamma,x)\to \Gamma$ can be identified with 
$p(\sigma)=\gamma$ if 
$\sigma(1)=\eta_{\gamma}(x)$.
 
Let $G\times_{(\theta,c)}\Gamma$ be the twisted product given by Definition \ref{def-twisted-product}.
A representation $\widehat{\rho}$ of $\pi_1(X,\Gamma,x)$ on $ G\times_{(\theta,c)}\Gamma$ 
is said to be  {\bf $(\theta,c)$-twisted $\Gamma$-equivariant} 
if it is an extension of a representation 
$\rho\,:\, \pi_1(X,\, x_1)\,\longrightarrow\, G$ 
fitting  in a commutative diagramme
of homomorphisms

\begin{equation}\label{compatible-rep}
\begin{tikzcd}[ar symbol/.style = {
  equals/.style = {ar symbol={=}}}
  ]
0\arrow[r] &\pi_1(X,\,x_1)\arrow[r]\arrow[d,"\rho"] &\pi_1(X,\Gamma,x)\arrow[r]\arrow[d,"\widehat{\rho}"] &\Gamma\arrow[r]\ar[equals]{d}&0\\
0\arrow[r] &G\arrow[r] &G\times_{(\theta,c)}\Gamma\arrow[r] &\Gamma\arrow[r]&0.
\end{tikzcd}
\end{equation}
Denote by $\Hom(\pi_1(X,\Gamma,x), G\times_{(\theta,c)}\Gamma)$ the set of   
$(\theta,c)$-twisted $\Gamma$-equivariant representations 
$\widehat{\rho}\,:\,
\pi_1(X,\Gamma, x)\,\longrightarrow\, G\times_{(\theta,c)}\Gamma$. 

Let $\cM(X,G,\Gamma,\theta,c)$ be  the moduli space
of $(\theta,c)$-twisted $\Gamma$-equivariant $G$-Higgs bundles considered in Section 
\ref{section-twisted-equivariant-higgs-pairs}. 
Let $\calR(X,G,\Gamma,\theta,c)$ be the moduli space  
consisting of  
 $G$-conjugacy classes of elements of   
$\Hom(\pi_1(X,\Gamma,x), G\times_{(\theta,c)}\Gamma)$
whose restriction to $\pi_1(X,\,x_1)$ is reductive.
We then have the following Theorem \cite{PBI}.

\begin{theorem}[Non-abelian Hodge correspondence]\label{equivariant-nahc}
There is a homeomorphism between $\cM(X,G,\Gamma,\theta,c)$ 
and  $\calR(X,G,\Gamma,\theta,c)$.
\end{theorem}

\section{Fixed points in the absence of tensorization}\label{section-alpha-trivial}
Throughout this Section $G$ is a connected semisimple complex Lie group with centre $Z$ and Lie algebra $\lie g$, $X$ is a compact Riemann surface and $\Gamma$ is a finite subgroup of $\Aut(X)\times\Out(G)\times\C^*$. By Section \ref{section-action} there is an action of $\Gamma$ on $\cM(X,G)$. Let $\eta:\Gamma\to\Aut(X), a:\Gamma\to\Out(G)$ and $\mu:\Gamma\to\C^*$ be the homomorphisms induced by projection on each of the factors. Fix a homomorphism $\theta:\Gamma\to\Aut(G)$ lifting $a$. 

Let us denote by $\cM(X,G)^{\Gamma}$ the fixed-point subvariety of $\cM(X,G)$ for the action of $\Gamma$. The results of \cite{oscar-suratno} have been incorporated in this section in order to give a description of $\cM(X,G)^{\Gamma}$. The main result is Theorem \ref{main}, which contains our final description for the case where $\eta$ is injective and the $\Gamma$-action does not involve tensorization by elements of $H^1(X,Z)$. When $\eta$ is not injective, Theorem \ref{th-prym-narasimhan-ramanan-alpha-trivial-higgs} will provide a refinement of Theorem \ref{main} using the ideas developed in Section \ref{section-trivial-eta}.

\subsection{The forgetful morphism}
\label{section-forgetful-morphism}
With notation as above, fix a 2-cocycle $c\in Z^2_{a}(\Gamma,Z)$. Recall the notion of $(\theta,c,\mu)$-twisted $\Gamma$-equivariant $G$-Higgs bundle, given in Section \ref{section-twisted-equivariant-higgs-pairs}.

\begin{proposition}[Proposition 4.1 in \cite{oscar-suratno}]\label{underlyingsemi}
Let $(E,\bullet,\phi)$ be a $(\theta,c,\mu)$-twisted $\Gamma$-equivariant $G$-Higgs bundle.
\begin{enumerate}
  \item  If the underlying Higgs bundle $(E,\phi)$ is (semi)stable then $(E,\phi)$ is equivariantly (semi)stable.
  \item If $(E,\phi)$ is equivariantly semistable then $(E,\phi)$ is semistable.
  \item If  $(E,\phi)$ is equivariantly polystable 
        then $(E,\phi)$ is polystable.
 \end{enumerate}
\end{proposition}
\begin{proof}
The statement (1) follows directly from Definitions \ref{def-stability-higgs-bundle} and \ref{definition-stable-twisted-equivariant}. In order to prove (3), let $(E,\bullet,\phi)$ be an equivariantly polystable twisted equivariant $G$-Higgs bundle. Then, by Theorems \ref{EH1-equivariant} and \ref{EH1}, the underlying $G$-Higgs bundle is polystable.

Now we prove (2). Suppose $(E,\bullet,\phi)$ is equivariantly semistable
 but $(E,\phi)$ is not semistable. Note that, in characteristic $0$, $(E,\phi)$ is semistable if and only if 
 $(\ad(E),\ad(\phi))$ is semistable \cite[Lemma 2.10]{AB} --- hence $(\ad(E),\ad(\phi))$ is not semistable.
 Then there is a unique filtration of $(\ad(E),\ad(\phi))$ 
 by $\ad(\phi)$ invariant subbundles
 \[0=F_0\subset F_1\subset \cdots F_{n-1}\subset F_{n}=\ad(E)\] such that each $(F_{i}/F_{i-1},\ad(\phi)|_{F_i/F_{i-1}})$ is semistable and
 $\mu(F_{i}/F_{i-1})<\mu(F_{i-1}/F_{i-2})$, for all $i=1,2,\cdots,n$, called the Harder-Narasimhan filtration. Here $n$ is odd and $F_{n+1/2}$ is a parabolic subalgebra
 bundle of $\ad(E)$ --- see the proofs of \cite[Lemma 2.5]{bhp} and \cite[Lemma 2.11]{AB}.
 Moreover, as $(\ad(E),\ad(\phi))$ is not semistable, $\mbox{deg}(F_{n+1/2})> 0$. Let $\Ad(E):=E\times_G G$ be the group scheme associated 
 to $E$ for the adjoint action of $G$ on itself.
 By \cite[Lemma 4]{AAB} there exists a parabolic 
 subgroup scheme $\underline{P}\subset \Ad(E)$ such that the associated Lie algebra bundle is $F_{n+1/2}$. By Remark \ref{remark-associated-equivariant}, the twisted equivariant structure $\bullet$ on $(E,\phi)$ induces a $\Gamma$-equivariant structure on $(\ad(E),\ad(\phi))$, which we denote by `$\cdot$'. By uniqueness of the Harder-Narashimahn 
 filtration, $F_{n+1/2}\cdot\gamma=F_{n+1/2}$ for all $\gamma\in \Gamma$.
 One can also see that there exists a parabolic subgroup $P\subset G$ and a reduction of structure group $E_{\sigma}\subset E$ to $P$ such that 
 $\Ad(E_{\sigma})=\underline{P}$ --- see the proof \cite[Lemma 2.11]{AB} ---, hence in particular $F_{n+1/2}=E_{\sigma}(\mathfrak{p})$. But, since $F_{n+1/2}$ is $\Gamma$-invariant, so are $\sigma$ and $P$, and this would contradict our assumption that $(E,\bullet,\phi)$ is equivariantly semistable. Therefore, $(\ad(E),\mbox{ad}(\phi))$ is 
 semistable, and so is $(E,\phi)$.
 

\end{proof}

Let $\sigma_{x_i}\in Z^1_{c_{x_i}}(\Gamma_{x_i},G)$ for each isotropy point $x_i\in X$, with notation as in Section \ref{section-isotropy}. By Proposition \ref{underlyingsemi} here exists a forgetful morphism
\begin{equation}\label{eq-forgetful-morphism}
    \cM(X,G,\Gamma,\theta,c,\mu,\sigma)\to \cM(X,G).
\end{equation}
We denote the image of (\ref{eq-forgetful-morphism}) by  $\widetilde{\cM}(X,G,\Gamma,\theta,c,\mu,\sigma)$.
This consists of those isomorphism classes of polystable $G$-Higgs bundles which admit a 
$(\theta,c,\mu)$-twisted $\Gamma$-equivariant structure.
Now, if a $G$-Higgs bundle $(E,\phi)$ admits a $(\theta,c,\mu)$-twisted $\Gamma$-equivariant structure then, by definition of twisted equivariant structure, 
\[(E,\phi)\cong (\eta_{\gamma}^*\theta^{-1}_{\gamma}(E),\mug\eta_{\gamma}^*\theta^{-1}_{\gamma}(\phi)),\] where $\cong$ denotes isomorphism of $G$-Higgs bundles.
As points of $\mathcal{M}(X,G)$ consist of isomorphism classes of polystable $G$-bundles, we immediately have the following.
\begin{proposition}\label{obvious-fix}
 $\widetilde{\cM}(X,G,\Gamma,\theta,c,\mu,\sigma)\subset \cM(X,G)^{\Gamma}$.
\end{proposition}

\subsection{Fixed points and simplicity}\label{section-fixed-points-and-simplicity-alpha-trivial}

Recall that a $G$-Higgs bundle $(E,\phi)$ is said to be {\bf simple} if $\mbox{Aut}(E,\phi)=Z$.


\begin{proposition} \label{simple}
Let $\theta:\Gamma\to \Aut(G)$ be a lift of $a:\Gamma\to \Out(G)$, and 
let $(E,\phi)$ be a simple $G$-Higgs bundle over $X$ such that 
 $(E,\phi)\cong (\eta_{\gamma}^*\theta^{-1}_{\gamma}(E),\mu_{\gamma}\eta_{\gamma}^*\theta^{-1}_{\gamma}(\phi))$. 
Then $(E,\phi)$ admits a $(\theta,c,\mu)$-twisted $\Gamma$-equivariant structure for some 
$c\in Z^2_a(\Gamma,Z)$.
\end{proposition} 
 \begin{proof}
 Let $(E,\phi)$ be a  $G$-Higgs bundle over $X$ such that 
 \begin{equation}\label{sim1} 
 (E,\phi)\cong (\eta_{\gamma}^*\theta^{-1}_{\gamma}(E),\mu_{\gamma}\eta_{\gamma}^*\tg^{-1}(\phi))
 \end{equation}
for all $\gamma \in \Gamma$.  Let $\Aut(E,\phi)$ be the group of automorphisms covering the identity of $X$, and $\Aut_{\Gamma,\eta,\theta}(E,\phi)$ be the group
defined in Remark \ref{remark-exact-sequence-aut} for $\rhog=\omu\theta$. 
The simplicity of $(E,\varphi)$ implies that 
$\Aut(E,\varphi)\cong Z$,  and hence (\ref{exact-aut})
gives an extension
$$
1\to Z\to \Aut_{\Gamma,\eta,\theta}(E,\varphi) \to \Gamma\to 1.
$$
This extension defines a 2-cocycle $c\in Z^2_a(\Gamma,Z)$,  and a
$c$-twisted homomorphism 
$$
\Gamma\to \Aut_{\Gamma,\eta,\theta}(E,\varphi)
$$ 
with cocycle $c$, that is, a $(\theta,c,\mu)$-twisted $\Gamma$-equivariant
structure on $(E,\varphi)$. 
\end{proof}

We are now ready to state the main result of this section. In order to make the notation more consistent with future sections, we rewrite $\widetilde{\cM}(X,G,\Gamma,\theta,c,\mu)$ as $\widetilde{\cM}(X,G,\Gamma,\theta,c,\omu\theta)$, taking advantage of the equivalence of categories between $(\theta,c,\mu)$-twisted $\Gamma$-equivariant Higgs bundles and $(\theta,c,\omu\theta)$-twisted $\Gamma$-equivariant Higgs bundles outlined at the end of Section \ref{section-twisted-equivariant-higgs-pairs}.

\begin{theorem}\label{main}
 Let ${\cM}_*(X,G)\subset \cM(X,G)$ be the open subvariety of $\cM(X,G)$ 
consisting of those $G$-Higgs bundles
 which are stable and simple. Then 
$$ \bigcup_{[c]\in H^2_a(\Gamma,Z)} 
\widetilde{\cM}(X,G,\Gamma,\theta,c,\omu\theta)
\subset
\cM(X,G)^{\Gamma}
$$
and 
 $$
{\cM}_*(X,G)^{\Gamma}\subset \bigcup_{[c]\in H^2_a(\Gamma,Z)} 
\widetilde{\cM}(X,G,\Gamma,\theta,c,\omu\theta).
$$
Here $\theta:\Gamma\to \Aut(G)$ is 
any  lift of $a:\Gamma\to \Out(G)$, and $H^2_a(\Gamma,Z)$ is the second cohomology group of $\Gamma$ with values in $Z$ with $\Gamma$-action induced by $a$. Moreover, the subvarieties
\begin{equation}\label{eq-intersections-twisted-equivariant-smooth}
    {\cM}_*(X,G)\cap\widetilde{\cM}(X,G,\Gamma,\theta,c,\omu\theta)={\cM}_*(X,G)^{\Gamma}\cap\widetilde{\cM}(X,G,\Gamma,\theta,c,\omu\theta)
\end{equation}
are all disjoint for different cohomology classes $[c]\in H^2_a(\Gamma,Z)$.

\end{theorem}
\begin{proof}
 Let $(E,\phi)\in {\cM}_*(G)^{\Gamma}$. 
 Then, by Proposition \ref{simple}, $(E,\phi)$ admits a $(\theta,c,\mu)$-twisted $\Gamma$-equivariant structure, where $c\in Z^2_\theta(\Gamma,Z)$, the set of all $2$-cocycles where 
 $\Gamma$ acts on $Z$ via $\theta$.
  Hence $(E,\phi)\in \widetilde{\cM}(X,G,\Gamma,\theta,c,\omu\theta)$. 
It follows from Proposition  \ref{cohom} that the
union should run over $[c]\in H^2_a(\Gamma,Z)$. The fact that the intersections \ref{eq-intersections-twisted-equivariant-smooth} are disjoint follows from the fact that a twisted $\Gamma$-equivariant action on a simple Higgs bundle is unique up to multiplication of the action of each $\gamma\in\Gamma$ by an element of $Z$, which preserves the cohomology class of the corresponding 2-cocycle.
\end{proof}

\begin{remark}\label{remark-isotropy}
    Theorem \ref{main} may be refined by taking account of the twisted $\Gamma$-action at isotropy points according to Proposition \ref{local-data}. More precisely, let $\{x_i\}_{i}$ be the set of isotropy points of the $\Gamma$-action on $X$, and let $\Gamma_{x_i}$ be the isotropy group of $\Gamma$ at $x_i$. Fix $c\in Z^2_a(\Gamma,Z)$ and let $c_{x_i}$ be its restriction to $\Gamma_{x_i}$. Denote by ${\cM}(X,G,\Gamma,\theta,c,\mu,\sigma)$ the moduli space of $(\theta,c,\mu)$-twisted $\Gamma$-equivariant $G$-Higgs bundles with local types $\{\sigma_{x_i}\}_{i}$, where $\sigma:=\{\sigma_{x_i}\}$ is a set of equivalence classes $\sigma_{x_i}\in H^1_{c_{x_i}}(\Gamma_{x_i},G)$ --- see Section \ref{section-isotropy}. Then each moduli space ${\cM}(X,G,\Gamma,\theta,c,\omu\theta)$ appearing in Theorem \ref{main} has a decomposition
\begin{equation*}
    {\cM}(X,G,\Gamma,\theta,c,\omu\theta)=\bigsqcup_{ 
[\sigma_{x_i}]\in H^1_{c_{x_i}}(\Gamma_{x_i},G)} 
{\cM}(X,G,\Gamma,\theta,c,\omu\theta,\sigma).
\end{equation*}
\end{remark}

\subsection{Fixed points in the character variety}\label{section-character-variety-alpha-trivial}

We now study the action of $\Gamma$ on the character variety $\calR(X,G)$, and 
describe the fixed points in terms of twisted equivariant representations. Recall that we are given  homomorphisms  $\eta:\Gamma\to \Aut(X)$, 
$a:\Gamma\to \Out(G)$ and $\theta:\Gamma\to \Aut(G)$, where $\theta$ is a lift of $a$. For most of this section the character $\mu$ is trivial.

Fix a point $x\,\in\, X$. For each $\gamma\in\Gamma$, the automorphism  $\eta_\gamma$ of $X$ produces a
homomorphism
$$
{\eta_\gamma}_*\, :\, \pi_1(X, x)\to \pi_1(X, \eta_\gamma(x))\, .
$$
This induces an automorphism of $\calR(X,G)$, since 
the quotient $\Hom(\pi_1(X,x),G)/G$ is independent of the base point of $X$.
Now, given the automorphism $\theta_\gamma$  of $G$ and  
$\rho\in \Hom(\pi_1(X,x),G)$, there is another representation of $\pi_1(X,x)$ in $G$ given by $\theta_\gamma\circ \rho$. This defines a left action of $\Gamma$ on 
$\calR(X,G)$ that clearly depends only on $a:\Gamma\to \Out(G)$.
So, for every $\gamma\in \Gamma$ and  $\rho\in \Hom(\pi_1(X),G)$, we define 
$\rho\cdot\gamma \in \Hom(\pi_1(X),G)$ by 
\begin{equation}\label{eq-action-alpha-trivial-character-variety}
\rho\cdot\gamma=\theta_\gamma^{-1}\circ \rho\circ {\eta_\gamma}_*.
\end{equation}
It is straightforward to show that the action of $\Gamma$ on  $\calR(X,G)$ 
given by this coincides with the action of $\Gamma$ on $\cM(X,G)$ defined in 
Section \ref{section-action} via the non-abelian Hodge correspondence ---see 
\cite{biswas-calvo-García-Prada,PR,ow} 
for a similar computation. 





Hence, combining Theorem \ref{rep1} with Theorems \ref{equivariant-nahc} and \ref{main}, we conclude the following.

\begin{theorem}\label{main-rep}
Let ${\calR}_{\irr}(X,G)\subset \calR(X,G)$ be the subvariety of $\calR(X,G)$ 
consisting of irreducible representations, and let 
$\widetilde{\calR}(X,G,\Gamma,\theta,c)$ 
be the image of $\calR(X,G,\Gamma,\theta,c)$ in $\calR(X,G)$ under the 
natural map defined by Diagramme  (\ref{compatible-rep}). Let $\calR(X,G)^{\Gamma}$ be
the fixed-point subvariety for the action of $\Gamma$ defined by the homomorphism $(\eta,a): \Gamma \to \Aut(X)\times \Out(G)$.
Then

$$
\bigcup_{[c]\in H^2_a(\Gamma,Z)}\widetilde{\calR}(X,G,\Gamma,\theta,c)\subset \calR(X,G)^{\Gamma}  
$$

and

$$
{\calR}_{\irr}(X,G)^{\Gamma}\subset \bigcup_{[c]\in H^2_a(\Gamma,Z)} 
\widetilde{\calR}(X,G,\Gamma,\theta,c).
$$
Here $\theta:\Gamma\to \Aut(G)$ is 
any  lift of $a:\Gamma\to \Out(G)$.
\end{theorem}

\begin{remark}
We could also have considered a non-trivial character $\mu:\Gamma\to \C^*$ 
with image equal to the subgroup $\{\pm 1\}<\C^*$. With a minor
modification in the definition of the group $G\times_{(\theta,c)}\Gamma$ 
one would obtain similar results --- see  
\cite{biswas-calvo-García-Prada,PR,ow} 
for an analogous situation.
\end{remark}

\subsection{Involutions of hyperelliptic curves and \texorpdfstring{$G=\SL(2,\C)$}{G=SL(2,C)}}\label{section-alfa-trivial-example1}

 Let $G=\SL(2,\C)$ and $(X,\sigma)$ be a hyperelliptic curve together with the 
hyperelliptic involution $\sigma$.
 In this case let $\Gamma=\Z/2\Z$  and 
consider  the homomorphism $\eta:\Z/2\Z\to \Aut(X)$ defined by sending 
$-1\mapsto \sigma$.
In this case $\Out(G)=1$ and $Z=\Z/2\Z$,  hence $\Aut(G)=\Int(G)$, 
and therefore
 $\Aut(G)$ acts trivially on the centre $Z$. So, in this case, $H^2(\Z/2\Z,Z)=\Z/2\Z$. 
Also,
there are only two characters $\mu^\pm$, defined by  $\mu^\pm(-1)=\pm 1$. 
We can then define actions on the moduli space of $\SL(2,\C)$-Higgs bundles
defined by $\eta$ and $\mu^\pm$. The case in which $\eta$ is the trivial homomorphism from $\Gamma$ to $\Aut(X)$ and $\mu=\mu^-$ is studied in
\cite{hitchin1987,García-Prada,García-Prada-ramanan,PR}.

\subsection{Involutions of hyperelliptic curves and \texorpdfstring{$G=\SL(n,\C)$}{G=SL(n,C)}}\label{section-alfa-trivial-example2}

Let $G=\SL(n,\C)$, with $n>2$ and $X$ a hyperelliptic curve together with the 
hyperelliptic involution $\sigma$, as above. Let $\Gamma=\Z/2\Z$ and 
$\eta:\Z/2\Z\to \Aut(X)$ be the 
homomorphism defined by sending $-1\mapsto \sigma$. 
In this case $\Out(G)\cong \Z/2\Z$ and $Z\cong \Z/n\Z$.
Let us denote the 
trivial homomorphism from $\Gamma\to \Out(G)$ by $a^+$, and the homomorphism which sends $-1$ to $b$, where
$\Out(G)=\langle b \rangle$, by $a^-$. 
In the first case $\Gamma$ acts trivially on the 
centre $Z$ via $a^+$. To compute the second group cohomology of $\Z/2\Z$ with coefficients in $\Z/n\Z$ we use the following fact:
Let $\Lambda$ be a cyclic group of order $r$ generated by $t$ and $A$ be a finite abelian group with a $\Lambda$-action. Let 
$N=1+t+\cdots+t^{r-1}\in \Z[\Gamma]$ then obviously $Na$, $a\in A$, is fixed by all $\lambda\in \Lambda$. With these notations
$H^p(\Lambda, A)=\frac{A^{\Gamma}}{NA}$, $p=2,4,6...$.
Hence, in this case, $H^2(\Gamma,Z)=0$ when $n$ 
is odd and $H^2(\Gamma,Z)=(\Z/2\Z)^r$ when $n$ is even,
where $r$ is the minimal number of generators of the $2$-Sylow subgroups of 
$\Z/n\Z$. 
On the other 
hand the action of the generator of $\Z/2\Z$ on $Z$ induced by $a^-$ 
sends $x\in Z$ to $x^{-1}$. In this case $H^2_{a^-}(\Gamma, Z)=Z^{\Z/2\Z}$.
Therefore the action is trivial when $n$ is odd and 
hence $H^2_{a^-}(\Gamma, Z)=0$, and if $n$ is even then $H^2_{a^-}(\Gamma,Z)$ consists of  
all order $2$ elements of $\Z/n\Z$.
As in the previous example, one has $\mu^\pm$ as possible characters.

The cases in which $\eta$ is the trivial homomorphism from $\Gamma$ to $\Aut(X)$ and $\mu=\mu^\pm$ is studied in
\cite{García-Prada,PR} --- see also \cite{heller-schaposnik,schaffhauser} 
for related work.

\subsection{The case \texorpdfstring{$G=\Spin(8,\C)$}{G=Spin(8,C)}}\label{section-alfa-trivial-example3}

Let $G=\Spin(8,\C)$. Then $Z=\Z/2\Z\times \Z/2\Z$ and $\Out(G)\cong S_3$. 
In \cite{PR} the authors consider various 
actions of cyclic subgroups  $\Gamma$ of $\Out(G)$, with 
$\Gamma$ acting trivially on  $X$, and identify the fixed-point subvarieties.  

Now in our situation the following three cases are relevant.

Case (I): Let $X$ be a compact Riemann surface of genus $g>2$ and $\Gamma:=S_3$.
Let $\eta:\Gamma \to \Aut(X)$ be an injective homomorphism, so that the action of $\Gamma$ on $X$ is faithful.
 Let $\sigma$ and $\tau$ generate the group $\Out(\Spin(8,\C))\cong S_3$. 
Let $a: \Gamma\to \Out(\Spin(8,\C))$ be the isomorphism defined by sending an order $2$ generator to $\sigma$
and an order $3$ generator to $\tau$. Let $\mu:\Gamma\to \mathbb{C}^*$ be a  
character of $S_3$. We know that $S_3$ has three non-equivalent conjugacy classes. Let 
$\mu_i$, $i=1,2,3$, be the corresponding characters. 
We define homomorphisms 
$F_i=(\eta,a,\mu_i):\Gamma \to \Aut(X)\times \Out(G)\times \C^*$, $i=1,2,3$. Then each $F_i$  determines an action on 
the moduli space of $G$-Higgs bundles. Let $H=\langle\tau\rangle$ be the 
normal subgroup of $G$ generated by $\tau$. 
Then by \cite[Lemma 2.2.4]{oxr} $H^2(\Gamma,Z)=H^2(\Gamma/H,Z^{H})$. As $\tau$ is an element of order $3$ either $Z^{H}=(e)$ or $Z^{H}=Z$.
So either $H^2(\Gamma,Z)=0$ or $H^2(\Gamma,Z)=\Z/2\Z$.

Case (II): Let $X$ be a hyperelliptic curve and $\Gamma:=S_3$. We 
define a homomorphism $\eta:S_3\to \Aut(X)$ by sending $\sigma$ to $\text{Id}$ and $\tau$ to an order $2$ hyperelliptic 
involution. Let $b_i\in \Out(G)$ be the class of an order two automorphism of $G$ and $\theta_i:\Gamma \to \Out(G)$
be the homomorphism defined by $\tau \mapsto 1$ and $\sigma \mapsto b_i$. 
We define $F_i:=(\eta,\theta_i,\mu_i):\Gamma\to \text{Aut}(X)\times \text{Out}(G)$, $i=1,2,3$. Then the respective actions of $\Gamma$ on the moduli space 
of $G$-Higgs bundles are determined by $F_i$. In this case the subgroup $H$ 
acts on $Z$ trivially, therefore $Z^{H}=Z$, and hence 
$H^2(\Gamma,Z)=\Z/2\Z$.

Case (III): $X$ is a cyclic trigonal curve.  In other words we assume that the subgroup $\langle \tau\rangle$ acts trivially on $X$ and
$f$ is an order $3$ automorphism such that $X/\langle f\rangle\cong \mathbb{P}^1$.
This case is related to the work of Oxbury and Ramanan \cite{ow}, and to
what they refer as Galois $\Spin(8,\mathbb{C})$-bundles.
We 
define a homomorphism $\eta:S_3\to \Aut(X)$ by sending $\sigma$ to $\Id$ and 
$\tau$ to the order $3$ automorphism $f$.
Let $b$ be the class of unique order $3$ automorphism of $X$, and let $\theta:S_3\to \Out(G)$ be the homomorphism sending $\sigma$ to $I$ and $\tau$ to $b$.  As in the previous case we define $F_i:=(\eta,\theta_i,\mu_i):\Gamma\to \text{Aut}(X)\times \text{Out}(G)$, $i=1,2,3$. 
Then the action of $\Gamma$ on the moduli space 
of $G$-Higgs bundles is determined by $F_i$. Since $\Z/3\Z$ and 
$Z=\Z/2\Z\times \Z/2\Z$ have coprime 
order, by \cite[Lemma 2.2.4]{ow} $H^2(\Z/3\Z,Z)=0$.

\subsection{The case \texorpdfstring{$G=E_6$}{G=E6}}\label{section-alfa-trivial-example4}

Let $G$ be a group of type $E_6$. In this case $\text{Out}(G)=\Z/2\Z$.
Let $X$ be a hyperelliptic curve together with a hyperelliptic involution $\sigma$ and $\Gamma=\Z/2\Z$.
We define a homomorphism $\eta:\Gamma\to \text{Aut}(X)$ by sending $-1\mapsto \sigma$. As in the case 
of example $1$ one has two homomorphisms $a^\pm:\Gamma\to \text{Out}(G)$ and two characters $\mu^{\pm}$.

\section{Fixed points for trivial action on \texorpdfstring{$X$}{X}}
\label{section-trivial-eta}

Let $X$ be a compact Riemann surface and let $G$ be a connected semisimple complex Lie group with centre $Z$ and Lie algebra $\lieg$. Let $\Gamma$ be a finite group equipped with a homomorphism
$$\Gamma\to H^1(X,Z)\rtimes\outg\times\C^*.$$
Projections on the first, second and third factors provide a 1-cocycle $\alpha\in Z^1_a(\Gamma,H^1(X,Z))$ and homomorphisms $a:\Gamma\to \Out(G)$, $\mu:\Gamma\to\C^*$, respectively --- see Section \ref{section-lifts-and-non-abelian-cohomology}.
We consider the action of $\Gamma$ on $\mdl(X,G)$ defined by (\ref{eq-action}), for trivial $\eta$.

\begin{remark}\label{remark-finiteness-zt}
Sections \ref{section-simple-and-H}, \ref{section-fixed-isomorphism-classes} and \ref{section-trivial-eta-fixed-moduli} could be rewritten for $G$ an arbitrary connected reductive complex Lie group, hence Theorem \ref{th-fixed-points-oscar-ramanan-higgs} also holds in this general setting. Section \ref{section-prym-narasimhan-ramanan} may also be rewritten if $G$ is reductive, under the assumption that the order of $\alpha\in Z^1_a(\Gamma,H^1(X,Z))$ is finite.
\end{remark}

\subsection{Simple fixed points and elements of \texorpdfstring{$H^1_{\theta}(\Gamma,G/Z)$}{H1(Gamma,G/Z)}}\label{section-simple-and-H}
Fix a lift $\theta:\Gamma\to\Aut(G)$ of $a$. Abusing notation we also call $\theta$ to the induced homomorphism $\Gamma\to\Aut(G/Z)$. 
In this section we construct a map from the set $S$ of isomorphism classes of simple $G$-Higgs bundles which are fixed under the action of $\Gamma$ to $H^1_{\theta}(\Gamma,G/Z)$, as defined in Section \ref{section-lifts-and-non-abelian-cohomology}.

Let $(E,\phi)$ be a simple $G$-Higgs bundle over $X$ such that $(E,\phi)\cong(\tg^{-1}(E\otimes\alg),\mug\tg^{-1}(\phi))$ for every $\gamma$ in $\Gamma$. In other words, for each $\gamma\in\Gamma$ there is an isomorphism
$$\hg:E\xrightarrow{\sim}\tg^{-1}(E\otimes\alg)$$
such that
$\hg(\phi)=\mu_{\gamma}\tg^{-1}(\phi).$
This in turn induces an isomorphism
$\ohg:E/Z\to\tg^{-1}(E\otimes\alg)/Z.$
Since $\tg(Z)=Z$, there are natural isomorphisms
$\tg^{-1}(E\otimes\alg)/Z\cong\tg^{-1}(E)/Z\cong\tg^{-1}(E/Z).$
According to Section \ref{section-action}, $\tg^{-1}(E/Z)$ and $E/Z$ are biholomorphic. Therefore, we may regard the isomorphisms $\ohg$ as biholomorphisms
$$\ohg:E/Z\to E/Z$$
satisfying
$\ohg(eg)=\ohg(e)\tg(g).$

Define a holomorphic map
$$f:E/Z\to \fun{\Gamma}{G/Z}$$
in such a way that $\ohg(e)=e \fg(e)$ for each $e\in E/Z$. A straightforward calculation shows that 
\begin{equation}\label{eq-conjugacy-f}
    f(e g)=g^{-1}f(e)\theta(g)
\end{equation}
for every $g\in G/Z$ and $e\in E/Z$, where we are identifying elements in $G/Z$ with constant functions from $\Gamma$ to $G/Z$. 

\begin{lemma}\label{lemma-f-Z}
For every $e\in E/Z$, $f(e)\in Z^1_{\theta}(\Gamma,G/Z)$.
\end{lemma}
\begin{proof}
Recall from Remark \ref{remark-induced-isomorphisms} that, if $\gamma$ and $\gamma'$ are elements of $\Gamma$, the isomorphism $\hg$ induces an isomorphism 
$$(E,\phi)\cdot\gamma'\to (E,\phi)\cdot\gamma\gamma'$$
which we also call $\hg$. Since $h_{\gamma}h_{\gamma'}h_{\gamma\gamma'}^{-1}$ is in the gauge group $Z$ of $(E,\phi)$, we obtain $\ohg \oh_{\gamma'}=\oh_{\gamma\gamma'}$ and so
$$ef_{\gamma\gamma'}(e)=\oh_{\gamma\gamma'}(e)=\oh_{\gamma}\oh_{\gamma'}(e)=\oh_{\gamma}(ef_{\gamma'}(e))=\oh_{\gamma'}(e) \theta_{\gamma}(f_{\gamma'}(e))=ef_{\gamma}(e) \theta_{\gamma}(f_{\gamma'}(e))$$
for each $e\in E/Z$, as required.
\end{proof}

Note that the fact that $G/Z$ is an affine algebraic variety implies that $\fung$ is affine. Since $Z^1_{\theta}(\Gamma,G/Z)$ is a closed subvariety of $\fung$, it is itself affine. Consider the action of $G/Z$ on $Z^1_{\theta}(\Gamma,G/Z)$, such that $g\in G/Z$ sends $\beta\in Z^1_{\theta}(\Gamma,G/Z)$ to $g\beta\theta(g)^{-1}$. Because of Lemma \ref{lemma-f-Z} and (\ref{eq-conjugacy-f}), there is a morphism
\begin{equation*}\label{eq-def-morphism-f}
    X\to Z^1_{\theta}(\Gamma,G/Z)\sslash G/Z,
\end{equation*}
where the right hand side is the corresponding GIT quotient \cite{mumford}. Since $Z^1_{\theta}(\Gamma,G/Z)$ is affine, so is $Z^1_{\theta}(\Gamma,G/Z)\sslash G/Z$. Since this is an algebraic morphism from a projective variety to an affine variety, it must be constant. 

Denote by $O(e)$ the orbit of $f(e)$ under the action of $G/Z$. Given another element $e'\in E$ in the same fibre of $e$, by the previous paragraph the images of $f(e)$ and $f(e')$ in the GIT quotient $Z^1_{\theta}(\Gamma,G/Z)\sslash G/Z$ are equal and so the closures $\overline{O(e)}$ and $\overline{O(e')}$ in $Z^1_{\theta}(\Gamma,G/Z)$ intersect. We claim that $O(e)$ is closed in $Z^1_{\theta}(\Gamma,G/Z)$, so that $O(e)=O(e')$ for every two elements $e$ and $e'$ in the same fibre.

In fact, the orbit of any element of $Z^1_{\theta}(\Gamma,G/Z)$ by the $G/Z$-action is closed. Indeed, there is a closed embedding
$$Z^1_{\theta}(\Gamma,G)\hookrightarrow\Hom(\Gamma,G\rtimes_{\theta}\Gamma);\,\beta\mapsto (\gamma\mapsto (\beta(\gamma),\gamma)).$$
This is $G$-equivariant for the conjugation action of $G$ on $\Hom(\Gamma,G\rtimes_{\theta}\Gamma)$, since
\begin{align*}
g(\beta(\gamma),\gamma)g^{-1}=(g\beta(\gamma)\tg(g)^{-1},\gamma)
\end{align*}
for each $g\in G$, $\gamma\in\Gamma$ and $\beta\in Z^1_{\theta}(\Gamma,G)$.
Recall that a homomorphism $\Gamma\to G\rtimes_{\theta}\Gamma$ is \textbf{reductive} if its composition with the adjoint representation is completely reducible. On the other hand, $G\rtimes_{\theta}\Gamma$ is a reductive group --- it is a finite extension of a reductive group --- and so reductive representations have closed orbits \cite{representations-finitely-generated-groups}. But $\Gamma$ is finite, hence all its representations are reductive.

We thus obtain a map
\begin{equation}\label{eq-def-f}
    \wf:S\to H^1_{\theta}(\Gamma,G/Z)=Z^1_{\theta}(\Gamma,G/Z)/ (G/Z)
\end{equation}
sending $(E,\phi)$ to the class of $f(e)$ for any $e\in E/Z$, where $S$ is the set of isomorphism classes of simple $G$-Higgs bundles whose isomorphism class is preserved by $\Gamma$. Moreover, the following lemma holds.

\begin{lemma}\label{lemma-orbit-fibre}
An element $\beta\in Z^1_{\theta}(\Gamma,G/Z)$ is in the class $\wf(E,\phi)$ if and only if, for every --- or for some --- $x\in X$, there exists $e$ in the fibre of $E/Z$ over $x$ such that $f(e)=\beta$.
\end{lemma}
\begin{proof}
The ``if'' direction follows immediately from the definition of $\wf$. For the ``only if'' direction, fix $x\in X$ and let $\beta\in \wf(E,\phi)$. Then there exist $e'$ in the fibre of $x$ and $g\in G/Z$ such that $f(e')=g^{-1}\beta\theta(g).$ By (\ref{eq-conjugacy-f}) we obtain
$$f(e'g^{-1})=gf(e')\theta(g^{-1})=\beta,$$
so we set $e:=e'g^{-1}$.
\end{proof}

\begin{remark}\label{remark-G/Z-vs-Int(G)}
    Using the natural isomorphism $G/Z\cong\Int(G)$, we sometimes regard the image of $\wf$ as an element of  $H^1_{\theta}(\Gamma,\Int(G))$, where the action of $\Gamma$ on $\Int(G)$ is conjugation by $\theta$.
\end{remark}

\begin{remark}\label{remark-f-theta-theta'}
    Given another lift $\theta'=\beta\theta$ of $a$, where $\beta\in Z^1_{\theta}(\Gamma,\Int(G))$, the bijections from $Z^1_{\theta}(\Gamma,\Int(G))$ and $Z^1_{\theta'}(\Gamma,\Int(G))$ to the set of lifts of $a$ induce the bijection
    \begin{equation}\label{eq-iso-Htheta-Htheta'}
        Z^1_{\theta}(\Gamma,G/Z)\xrightarrow{\beta^{-1}}Z^1_{\theta'}(\Gamma,G/Z)
    \end{equation}
    given by multiplication by $\beta^{-1}$ on the right. If we denote by $\wf'$ the map (\ref{eq-def-f}) defined in this section after replacing $\theta$ with $\theta'$, then $\wf'$ is equal to the composition of $\wf$ with (\ref{eq-iso-Htheta-Htheta'}), i.e. $\wf'=\beta^{-1}\wf$.
\end{remark}

\subsection{Simple \texorpdfstring{$G$}{G}-Higgs bundles and fixed points}\label{section-fixed-isomorphism-classes}
We describe the isomorphism classes of $G$-Higgs bundles which are fixed under the action of $\Gamma$ using arguments similar to the ones in \cite{PR}. Let $K_X$ be the canonical bundle of $X$. In what follows we use the following notation: in general, given a $G$-Higgs bundle $(E,\phi)$ --- see Definition \ref{def-higgs-pair} ---, a subgroup $G'\le G$ and a reduction of structure group $F$ of $E$ to $G'$, there is an induced Higgs field $\psi\in F(\lieg)\otimes K_X$, where the representation of $G'$ in $\lieg$ is the restriction of the adjoint representation. This is due to the fact that, if $\phi$ is locally defined by $(e,v)\otimes k$ for some $e\in E$, $v\in \lieg$ and $k\in K_X$, there exists $g\in G$ such that $(eg,\Ad_{g^{-1}}v)\otimes k$ is also a local expression for $\phi$ and $eg\in F$. In fact, for this reason $E(\lieg)$ and $F(\lieg)$ are isomorphic vector bundles. we denote by $(F,\psi)$ a \textbf{reduction of structure group} of $(E,\phi)$. Conversely, $(E,\phi)$ is an \textbf{extension of structure group} of $(F,\psi)$.

\begin{proposition}\label{prop-reduction}
Let $(E,\phi)$ be a $G$-Higgs bundle and $\theta\in\homgh$ a lift of $a$. With notation as in Sections \ref{section-Gtheta} and \ref{section-adjoint-representation}, assume that there is a reduction of structure group of $(E,\phi)$ to a $(\gs,\liegm)$-Higgs pair $(F,\psi)$ such that 
\begin{equation}\label{eq-c-theta}
    \ctt(F)=\alpha.
\end{equation}
Then $(E,\phi)$ is isomorphic to $(E,\phi)\cdot\gamma:=\tg^{-1}(E\otimes\alg,\mug\phi)$ for every $\gamma\in\Gamma$. If $\wf$ is the map (\ref{eq-def-f}), then $\wf(E,\phi)=1\in H^1_{\theta}(\Gamma,G/Z)$.
\end{proposition}

\begin{proof}
For each $\gamma$ in $\Gamma$ we want to obtain an isomorphism
$$\hg:F\xrightarrow{\sim}\tg^{-1}(F\otimes\alg).$$
This will, in turn, induce an isomorphism from the extension $E$ to the extension $\tg^{-1}(E\otimes \alg)$.
Fix $\gamma\in\Gamma$ and choose an open cover $\{U_i\}_{i\in I}$ of $X$ which trivializes both $F$ and $\alg$. Let $e_i$ and $z_i$ define local sections of $F$ and $\alg$ on $U_i$, respectively. We obtain transition functions
$$g_{ij}:U_i\cap U_j\to G,\,z_{ij}:U_i\cap U_j\to Z$$
satisfying $e_j=e_ig_{ij}$ and $z_j=z_iz_{ij}$. A set of local trivializations for $\tg^{-1}(F\otimes\alg)$ is then $\{U_i,e_i\otimes z_i\}$ --- with each $e_i\otimes z_i$ regarded as a local section of $\tg^{-1}(F\otimes \alg)$ as in Section \ref{section-action} ---, and the corresponding transition functions are
$$\tg^{-1}(g_{ij}z_{ij}):U_i\cap U_j\to G.$$
But, according to (\ref{eq-c-theta}), we may assume that $z_{ij}=g_{ij}^{-1}\tg(g_{ij})$ and so
$$\tg^{-1}(g_{ij}z_{ij})=\tg^{-1}(g_{ij}g_{ij}^{-1}\tg(g_{ij}))=g_{ij}.$$
Therefore, we may set $\hg(e_i):=e_i\otimes z_i$ and extend the isomorphism to the whole $F$ by imposing that it respects the $\gs$-actions.

We need to check that 
$$\hg(\psi)=\mu_{\gamma}\tg^{-1}(\psi).$$
Note that, if $\psi$ is locally of the form $(e_i,v)\otimes k$, where $e_i$ is the local section defined above, then the local form of $\hg(\psi)$ is also $(e_i,v)\otimes k$ with $e_i$ considered as a local section of $\tg^{-1}(E)$. But $\mug\tg^{-1}(\psi)$ is locally equal to
$$\mug(e_i,\tg^{-1}(v))\otimes k=\mug\mu_{\gamma}^{-1}(e_i,v)\otimes k=(e_i,v)\otimes k,$$
as required.

Finally, by the construction of $\hg$, the induced isomorphism $\hg:E/Z\to E/Z$ features an element $e\in E/Z$ --- namely, the image of $e_i$ in $E/Z$ for any $i$ --- such that $\hg(e)=e$, therefore the last statement of the proposition follows by the definition of $\wf$.
\end{proof}

\begin{proposition}\label{prop-simple-fixed-points}
Let $(E,\phi)$ be a simple $G$-Higgs bundle over $X$ isomorphic to $(E,\phi)\cdot\gamma$ for every $\gamma$ in $\Gamma$. Then there exists a homomorphism $\theta\in\homgh$ lifting $a$ and a $(\gs,\liegm)$-Higgs pair $(F,\psi)$ providing a reduction of structure group of $(E,\phi)$ satisfying (\ref{eq-c-theta}), i.e.
$
    \ctt(F)=\alpha.
$
Moreover, $\theta$ is unique up to conjugation by $\Int(G)$ and such a reduction exists and is unique, up to the action of $Z$, for every element of its conjugacy class. Given an automorphism $\theta'$ conjugate to $\theta$ and two reductions $(F,\psi)$ and $(F',\psi')$ to $(\gs,\liegm)$ and $(G_{\theta'},\lie g^{\theta'}_{\mu})$ respectively, there exists $g\in G$, such that $\theta'=\Int_g\theta\Int_g^{-1}$ and $F=F'g$.
\end{proposition}

\begin{proof}
Fix a lift $\theta$ of $a$. With notation as in Section \ref{section-simple-and-H} choose an element $\beta\in \wf(E,\phi)\subseteq Z^1_{\ot}(\Gamma,G/Z)\cong Z^1_{\ot}(\Gamma,\Int(G))$. consider the map
$\beta\theta:\Gamma\to\Aut(G).$
By Lemma \ref{lemma-lifts-vs-non-abelian-cohomology}, this is a homomorphism.

Take a map $s\in\Fun(\Gamma,G)$ such that $\beta=\Int_s$. Replace $\theta$ by $\beta\theta$ and the isomorphisms $\hg:(E,\phi)\to(\tg^{-1}(E\otimes\alg),\mug\tg^{-1}(\phi))$ by the compositions
\begin{align*}
    (E,\phi)\xrightarrow{\hg}(\tg^{-1}(E\otimes\alg),\mug\tg^{-1}(\phi))\xrightarrow{s_{\gamma}^{-1}}&(\tg^{-1}\Int_{s_{\gamma}^{-1}}(E\otimes\alg),\mug\tg^{-1}\Int_{s_{\gamma}^{-1}}(\phi))
    \\&=((\Int_{s_{\gamma}}\tg)^{-1}(E\otimes\alg),\mug(\Int_{s_{\gamma}}\tg)^{-1}(\phi)),
\end{align*}
where the second map is just multiplication by $s_{\gamma}^{-1}$ according to the $G$-action on $E$ --- or, alternatively, multiplication by $\tg^{-1}(s)$ according to the action on $\tg^{-1}(E\otimes\alg)$. Then we may assume that $1$ is in the image of the map $f:E/Z\to Z^1_{\theta}(\Gamma,G/Z)$ defined in Section \ref{section-simple-and-H}. By Lemma \ref{lemma-orbit-fibre}, $f^{-1}(1)$ has non-empty intersection with every fibre of $E/Z$. 

Define
\begin{equation}\label{eq-definition-reduction-from-iso}
    F:=\{e\in E:\hg(e)=e\otimes\zg(e),\,\zg(e)\in\alg\},
\end{equation}
where the elements $e\otimes\zg(e)$ are regarded as elements of $\tg^{-1}(E\otimes\alg)$ according to Section \ref{section-action}. This is the preimage of $f^{-1}(1)$ under the natural projection
$E\to E/Z.$ 
By the previous paragraph we may assume that the intersection of $F$ with every fibre is non-empty. Moreover, given $e\in F$ and $g\in G$, the element $eg$ is also in $F$ if and only if
$$e\tg(g)\otimes\zg'(e)=\hg(eg)=eg\otimes\zg(eg)$$
for every $\gamma\in\Gamma$ and some other elements $\zg'(e)\in\alg$ or, equivalently, $\tg(g)g^{-1}\in Z$. This shows that $F$ is a reduction of structure group of $E$ to $\gs$, and it inherits isomorphisms
$$F\to\tg^{-1}(F\otimes\alg);\,e\mapsto e\otimes\zg(e)$$
for every $\gamma\in\Gamma$. To see why (\ref{eq-c-theta}) is true, we fix $\gamma\in\Gamma$ and use an open cover $\{U_i\}_{i\in I}$ of $X$ trivializing $F$ together with local sections $e_i$ on each $U_i$, so that $\{z_i:=\zg(e_i)\}$ is a set of local sections for $\alg$. Setting $z_j=z_iz_{ij}$ --- according to the $Z$-action on $\alg$ --- and $e_j=e_ig_{ij}$, and denoting the action of $G$ on $\tg^{-1}(F\otimes\alg)$ with a dot,
$$(e_i\otimes z_i)g_{ij}z_{ij}=(e_i\otimes z_i)\cdot \tg^{-1}(g_{ij}z_{ij})=e_j\otimes z_j=\hg(e_j)=\hg(e_ig_{ij})=(e_i\otimes z_i)\cdot g_{ij}.$$
Therefore,
$z_{ij}=\tg(g_{ij})g_{ij}^{-1}$
as required.

On the other hand, if $\phi$ is locally of the form $(e,v)\otimes k\in E(\lie g)\otimes K_X$, the equation
\begin{equation}\label{eq-Phi-induces-local}
    (\hg(e),v)\otimes k=\mu_{\gamma}(e,\tg^{-1}(v))\otimes k
\end{equation}
holds for every $\gamma\in \Gamma$, where we are using the natural identification between $\tg^{-1}(E\otimes\alg)(\lieg)$ and $\tg^{-1}(E)(\lieg)$. The Higgs field $\phi$ is also equal to $(e g,\Ad_{g^{-1}}v)\otimes k$ for some $g\in G$ such that $e g\in F$, so we may assume that $e$ lies in $F$. In this case, (\ref{eq-Phi-induces-local}) translates into
$$\tg^{-1}(E)(\lieg)\otimes K_X\ni(e,v)\otimes k=\mu_{\gamma}(e,\tg^{-1}(v))\otimes k,$$
that is to say,
$v=\mu_{\gamma}\tg^{-1}(v)$
whenever $k$ does not vanish, and so $\phi$ is induced by a section of $F(\lieg^{\theta}_{\mu})\otimes K_X$ as required.

Note that another choice of $\beta$ in $\wf(E,\phi)$ would provide an element $g^{-1}\beta\ot(g)\in G/Z$ for some $g\in G/Z$, so that $\beta\theta$ is replaced by $\intt{g^{-1}}\beta\theta\intt g$ and so every element in the conjugacy class of $\beta\theta$ satisfies the requirements in the statement of the proposition. The uniqueness of the reduction up to multiplication by $Z$ follows from Proposition \ref{prop-reduction}, the simplicity of $(E,\phi)$ and the fact that the resulting isomorphisms completely determine the reduction by (\ref{eq-definition-reduction-from-iso}). 

Finally, we prove the uniqueness of $\theta$ up to conjugation by $\Int(G)$. Assume that there are two lifts $\theta$ and $\theta'$ of $a$ providing reductions of structure group $(F,\psi)$ and $(F',\psi')$ of $(E,\phi)$ to $\gs$ and $G_{\theta'}$, respectively, and satisfying (\ref{eq-c-theta}). Proposition \ref{prop-reduction} provides isomorphisms
$$\hg:(E,\phi)\xrightarrow{\sim}(\tg^{-1}(E)\otimes\alg,\tg^{-1}(\phi))$$
and
$$\hg':(E,\phi)\xrightarrow{\sim}(\tg'^{-1}(E)\otimes\alg,\tg'^{-1}(\phi)).$$
Moreover, the images of $(E,\phi)$ by the functions $\wf:S\to H^1_{\theta}(\Gamma,G/Z)$ and $\wf':S\to H^1_{\theta'}(\Gamma,G/Z)$ defined in Section \ref{section-simple-and-H} using $\theta$ and $\theta'$, respectively, are the respective trivial cohomology classes. Let
$\beta\in Z^1_{\theta}(\Gamma,\Int(G))\cong Z^1_{\theta}(\Gamma,G/Z)$ be such that $\theta'=\beta\theta$. By Remark \ref{remark-f-theta-theta'} $\wf'(E,\phi)\beta=\wf(E,\phi)$, which implies that there exists $g\in G$ such that $\Int_{g\tg(g)^{-1}}=\beta_{\gamma}$ for each $\gamma\in\Gamma$. Therefore $\theta'=\beta\theta=\Int_g\theta\Int_g^{-1}$. Moreover, $F' g$ is a reduction of structure group to $g^{-1}G_{\theta'}g=G_{\Int_g^{-1}\theta'\Int_g}=\gs$. To see why this equation is true, note that $u\in G$ is contained in $\Int_{g}(\gs)$ if and only if
$$\Int_g\theta\Int_{g^{-1}}(u)=c_{\theta'}(u,\gamma)u$$
for each $\gamma\in\Gamma$ and some $c_{\theta'}(u,\gamma)\in Z$, which yields $\Int_{g}(\gs)=G_{\theta'}$.
The reduction of structure group $F' g$ also satisfies (\ref{eq-c-theta}), since $F'g$ is the extension of structure group of $F'$ by the homomorphism $\Int_g^{-1}:G_{\theta'}\to \gs$ and, for every $s\in G_{\theta'}$ and any $\gamma\in\Gamma$,
\begin{align*}
    Z\ni c_{\theta'}(s)&=\tg'(s)s^{-1}
    \\&=\Int_g^{-1}(\tg'(s)s^{-1})
    \\&=\Int_g^{-1}\tg'\Int_g(\Int_g^{-1}(s))\Int_g^{-1}(s)^{-1}
    \\&=\tg(\Int_g^{-1}(s))\Int_g^{-1}(s)^{-1}
    \\&=c_{\theta}(\Int_g^{-1}(s)).
\end{align*}
Hence, by the previous paragraph, $F'g$ is equal to $F$ up to an element of $Z$. This proves the last statement of the proposition.

\end{proof}

\begin{remark}\label{remark-paramatrize-theta-cohomology}
In particular, Proposition \ref{prop-simple-fixed-points} shows that, when describing simple fixed points in the moduli space, we only need to consider lifts $\theta\in\homgh$ of $a$ up to conjugation by elements in $\Int(G)$. By Lemma \ref{lemma-lifts-vs-non-abelian-cohomology}, these are in bijection with $H^1_{\theta}(\Gamma,\Int(G))$.

\end{remark}

\subsection{Fixed points in the moduli space}\label{section-trivial-eta-fixed-moduli}

So far we have studied fixed points in the set of isomorphism classes of $G$-Higgs bundles. Next we will express these results in the context of moduli spaces.

\begin{proposition}\label{prop-polystability-extension-structure-group}
Let $\theta\in\homgh$, and let $\gs$ and $\liegm$ be defined as in Section \ref{section-Gtheta}. Then the following statements hold.
\begin{enumerate}
    \item If a $(\gs,\liegm)$-Higgs pair $(F,\psi)$ is polystable, the $G$-Higgs bundle $(E,\phi)$ obtained by extension of structure group is also polystable.
    \item If $(E,\phi)$ is a (semi,poly)stable $G$-Higgs bundle with a reduction of structure group to a $(\gs,\liegm)$-pair $(F,\psi)$, then $(F,\psi)$ is (semi,poly)stable.
    \item Given $g\in G$ and $\theta':=\Int_g\theta\Int_{g^{-1}}$, there is a canonical isomorphism between $\cM(X,\gs,\liegm)$ and $\cM(X,G_{\theta'},\lieg^{\theta'}_{\mu})$ making the following diagramme commute:
    \[\begin{tikzcd}
\cM(X,\gs,\liegm)\arrow{r}\arrow{d} &
\cM(X,G)\\
\cM(X,G_{\theta'},\lieg^{\theta'}_{\mu})\arrow{ur}
\end{tikzcd},
\]
where the morphisms to $\cM(X,G)$ are given by extension of structure group. For each $Y\in H^1(X,\gamtt)$, it restricts to a commutative diagramme
\[\begin{tikzcd}
\cM_{Y}(X,\gs,\liegm)\arrow{r}\arrow{d} &
\cM(X,G)\\
\cM_{Y}(X,G_{\theta'},\lieg^{\theta'}_{\mu})\arrow{ur}
\end{tikzcd},
\]
where $\cM_{Y}(X,\gs,\liegm)$ is the moduli space of $(\gs,\liegm)$-Higgs pairs $(F,\psi)$ such that $F/\gt_0\cong Y$.
\end{enumerate}
\end{proposition}
\begin{proof}
The proof of (2) is precisely the same as the proof of \cite[Proposition 5.7 (2)]{PR}. We give here the proof of (1), which also follows \cite{PR}: fix a maximal compact subgroup $K_{\theta}$ of $\gs$ and consider a maximal compact subgroup $K$ of $G$ containing it, so that $K_{\theta}=K\cap\gs$. By Lemma \ref{lemma-compact-involution} we may assume that $K$ is the fixed locus of an antiholomorphic involution $\sigma$ of $G$ satisfying
\begin{equation}\label{eq-sigma}
    \sigma(\liegm)=\lieg^{\theta}_{\omu}
\end{equation}
for every character $\mu:\Gamma\to\C^*$. Then, given a polystable $(\gs,\liegm)$-Higgs pair $(F,\psi)$, by Theorem \ref{EH1-equivariant} and Proposition \ref{prop-twisted-equivariant-bundles-one-to-one} there exists a reduction of structure group $h\in \Omega^0(F/K_{\theta})$, which satisfies (\ref{hitchin-equation-pairs}) for $V=\liegm$ and $h_V$ determined by $\sigma$ and the Killing form on $\lie g$. Let $(E,\phi)$ be the extension of structure group of $(F,\psi)$ to $G$. Using the inclusion $F/K_{\theta}\subset E/K$, we obtain an extension of $h$ to a reduction of structure group in $\Omega^0(E/K)$, which we also call $h$. The corresponding Chern--Singer connection is induced by that of $(F,\psi)$, hence the corresponding curvatures are equal. By Remark (\ref{remark-hitchin-eqs-equivalent}) $\Ad^*(\frac{i}{2}\phi\otimes\sigma_{h}(\phi))=\Lambda[\phi,\sigma_h(\phi)]$, where $\Lambda:\Omega^2(X)\to\Omega^0(X)$ is the adjoint of wedging with the K\"ahler form of volume 1 on $X$. Therefore, if $\Ad_{\theta,\mu}:\gs\to \GL(\liegm)$ is the restriction of the adjoint representation --- this is well-defined by Section \ref{section-adjoint-representation} ---, then
\begin{equation*}
    \Ad_{\theta,\mu}^*(\frac{i}{2}\psi\otimes\sigma_{h}(\psi))=\Lambda p_{\theta}[\psi,\sigma_h(\psi)]=\Lambda[\psi,\sigma_h(\psi)],
\end{equation*}
where $p_{\theta}:\lie g\to\lie g^{\theta}$ is the orthogonal projection and the last equation follows from (\ref{eq-sigma}). Hence $\Ad_{\theta,\mu}^*(\frac{i}{2}\psi\otimes\sigma_{h}(\psi))=\Lambda[\psi,\sigma_h(\psi)]=\Lambda[\phi,\sigma_h(\phi)]$. Since (\ref{hitchin-equation-pairs}) holds, so does (\ref{hitchin-equation}) for $(E,\phi)$. Therefore, by Theorem \ref{EH1}, $(E,\phi)$ is polystable.

The isomorphism in (3) is given as follows: consider a $G$-Higgs bundle $(E,\phi)$, a homomorphism $\theta\in\homgh$ and a reduction of structure group $F$ to $\gs$ such that $\phi$ is induced by a section $\psi$ of $F(\liegm)\otimes K_X$. Let $g\in G$ and $\theta':=\Int_{g}\theta \Int_{g^{-1}}$. We show that there is also a reduction $F'$ to $G_{\theta'}$ such that $\phi$ is induced by a section $\psi'$ of $F'(\lieg^{\theta'}_{\mu})\otimes K_X$. We set 
$$F':=Fg^{-1},$$
so that $F'$ is a reduction to $\Int_{g}(\gs)$. It is straightforward to see that $\Int_{g}(\gs)=G_{\theta'}$ and $\Int_{g}(\gt)=G^{\theta'}$ --- see the proof of Proposition \ref{prop-reduction} ---, therefore $g^{-1}$ induces an isomorphism from $F/\gt$ to $F'/G^{\theta'}$. Moreover, if $\psi$ is locally equal to $(e,v)\otimes k$ for some $e\in F$, $v\in\liegm$ and $k\in K_X$, we may define $\psi'$ locally as $(eg^{-1},\Ad_gv)\otimes k$. It is straightforward to see that this is a section of $F'(\lieg^{\theta'}_{\mu})\otimes K_X$.
\end{proof}

As above, let $\Gamma$ be a finite group equipped with a homomorphism
$\Gamma\to H^1(X,Z)\rtimes\outg\times\C^*$, and denote by $a:\Gamma\to \Out(G)$, $\mu:\Gamma\to\C^*$ and $\alpha\in Z^1_a(\Gamma,H^1(X,Z))$ the projections on the second, third and first factor, respectively.
Given $\theta\in\homgh$ lifting $a$, let $\cM_{\alpha}(X,\gs,\liegm)$ be the moduli space of $(\gs,\liegm)$-higgs pairs $(F,\psi)$ such that $\ctt(F)\cong\alpha$, where $\ctt$ is defined in Section \ref{section-Gtheta}. We denote by $\widetilde{\cM}_{\alpha}(X,\gs,\liegm)$ the image of $\cM_{\alpha}(X,\gs,\liegm)$ in $\cM(X,G)$. By Proposition \ref{prop-polystability-extension-structure-group}, if $\theta'=\Int_g\theta\Int_{g^{-1}}$ for some $g\in G$, then
$\widetilde{\cM}_{\alpha}(X,\gs,\liegm)=\widetilde{\cM}_{\alpha}(X,G_{\theta'},\lieg^{\theta'}_{\mu}).$

Let $\cM(X,G)^{\Gamma}$ be the fixed-point locus of the $\Gamma$-action on $\cM(X,G)$, and let $\cM_*(X,G)^{\Gamma}$ be the intersection with the stable and simple locus. Combining Propositions \ref{prop-reduction} and \ref{prop-simple-fixed-points}, we obtain the main result of this section.

\begin{theorem}\label{th-fixed-points-oscar-ramanan-higgs}
Fix $\theta\in\homgh$ lifting $a$. The following relations between moduli spaces hold:
\begin{enumerate}
    \item $$\bigcup_{[\beta]\in H^1_{\theta}(\Gamma,\Int(G))}\widetilde{\cM}_{\alpha}(X,G_{\beta\theta},\lieg^{\beta\theta}_{\mu})\subset\cM(X,G)^{\Gamma}. $$
    
    \item $$\cM_*(X,G)^{\Gamma}\subset\bigcup_{[\beta]\in H^1_{\theta}(\Gamma,\Int(G))}\widetilde{\cM}_{\alpha}(X,G_{\beta\theta},\lieg^{\beta\theta}_{\mu}).$$
    
\end{enumerate}

Moreover, the intersections 
$$\cM_*(X,G)\cap\widetilde{\cM}_{\alpha}(X,G_{\beta\theta},\lieg^{\beta\theta}_{\mu})=\cM_*(X,G)^{\Gamma}\cap\widetilde{\cM}_{\alpha}(X,G_{\beta\theta},\lieg^{\beta\theta}_{\mu})$$
are disjoint for different $[\beta]\in H^1_{\theta}(\Gamma,\Int(G))$.
\end{theorem}
\begin{remark}
If $\mu$ is trivial then the moduli spaces $\widetilde{\cM}_{\alpha}(X,G_{\beta\theta},\lieg^{\beta\theta}_{\mu})$ involved in the statement of Theorem \ref{th-fixed-points-oscar-ramanan-higgs} are moduli spaces of $\gs$-Higgs bundles.
\end{remark}

\subsection{The Prym-Narasimhan-Ramanan construction of fixed points}\label{section-prym-narasimhan-ramanan}

The Prym--Narasimhan--Ramanan construction can be now given as a combination of Theorem \ref{th-fixed-points-oscar-ramanan-higgs} and a corollary of Proposition \ref{prop-prym-narasimhan-ramanan}.


We keep the notation of Sections \ref{section-Gtheta} and \ref{section-trivial-eta-fixed-moduli}. Let $Z( \gt_0)$ be the centre of $ \gt_0$ and $\hat Y\in H^1(X,\gamtt)$, a $\gamtt$-bundle over $X$. Any connected component $Y$ of $\hat Y$ may be regarded as a connected étale cover of $X$ with Galois group $\Gamma_Y\le\gamtt$. Consider the subgroup $G_{Y}\le\gs$. This is the preimage of $\Gamma_Y$ under the quotient $\gs\to\gamtt$. Choose a lift
$\taut:\gamtt\to\Aut( \gt_0)$
of the characteristic homomorphism of the extension (\ref{eq-extension-connected-component}).
By Proposition \ref{prop-extensions-isomorphic-twisted-group}, there exists a 2-cocycle $\ct\in Z^2_{\taut}(\gamtt,Z( \gt_0))$ such that
$\gs\cong  \gt_0\times_{(\taut,\ct)}\gamtt.$
On the other hand, a $\gs$-bundle $E$ over $X$ satisfies $\ctt(E)\cong{\alpha}$ if and only if $E/ \gt_0\cong\hat Y$ for some element $\hat Y\in H^1(X,\gamtt)$ such that $\qqt(Y)\cong\alpha$, where $\qt$ is given by (\ref{eq-def-qt}).

Let $\rho_{Y}:\Gamma_Y\to \GL(\liegm)$ be the $(\taut,\ct)$-twisted representation given by the composition of the map $\Gamma_Y\to\gs$ sending $\gamma$ to $(1,\gamma)$, and the adjoint representation of $G$ on $\lie g$.

\begin{theorem}\label{th-prym-narasimhan-ramanan-oscar-ramanan-higgs}
For each homomorphism $\theta:\Gamma\to\Aut(G)$ lifting $a$ there is an isomorphism
\begin{equation}\label{eq-prym-narasimhan-ramanan}
    \bigsqcup_{\qqt(Y)\cong \alpha}\mdl(Y,\gt_0,\Gamma_Y,\tau^{\theta},c^{\theta}, \lieg^{\theta}_{\mu},\rho_Y)/Z_{\gamtt}(\Gamma_Y)\cong \mdl_{\alpha}(X,\gs,\liegm),
\end{equation}
where $Z_{\gamtt}(\Gamma_Y)$ is the centralizer of $\Gamma_Y$ in $\gamtt$, which acts on $\mdl(Y,\gt_0,\Gamma_Y,\tau^{\theta},c^{\theta}, \lieg^{\theta}_{\mu},\rho_Y)$ as explained in Proposition \ref{prop-action-centralizer}.


    
    


\end{theorem}
\begin{proof}
    Follows directly from Proposition \ref{prop-prym-narasimhan-ramanan}.
\end{proof}

\subsection{The Prym--Narasimhan--Ramanan construction for character varieties}\label{section-prym-narasimhan-ramanan-character-varieties}

We keep the notation of Section \ref{section-prym-narasimhan-ramanan}. If we suppose that $\mu$ is trivial, Theorems \ref{th-fixed-points-oscar-ramanan-higgs} and \ref{th-prym-narasimhan-ramanan-oscar-ramanan-higgs} yield a description of the fixed points of a $\Gamma$-action on the character variety $\calR(X,G)$.
The action of $\Gamma$ on $\calR(X,G)$ induced by its action on $\cM(X,G)$ via Theorem \ref{rep1} is given as follows. For each $\gamma\in\Gamma$, the corresponding $Z$-bundle $\alg$ is flat, since $Z$ is finite, and so it is given by a representation $\rho_{\gamma}:\pi_1(X)\to Z$. Given a representation $\rho:\pi_1(X)\to G$, we may multiply it by $\rho_{\gamma}$ to obtain a new representation $\rho\otimes\alpha_{\gamma}:=\rho\cdot\rho_{\gamma}$. Now, given the automorphism $\theta_\gamma$  of $G$ and  
$\rho\in \Hom(\pi_1(X,x),G)$, there is another representation of $\pi_1(X,x)$ in $G$ given by $\theta_\gamma\circ \rho$. This defines a left action of $\Aut(G)$ on 
$\calR(X,G)$ that factors through an action of $\Out(G)$.
So for every $\gamma\in \Gamma$ and  $\rho\in \Hom(\pi_1(X),G)$, we define
$\rho\cdot\gamma \in \Hom(\pi_1(X),G)$ by
\begin{equation}\label{eq-action-character-no-eta}
    \rho\cdot\gamma=\theta_\gamma^{-1}\circ (\rho\otimes\alg).
\end{equation}
It is straightforward to show --- see 
\cite{biswas-calvo-García-Prada,PR,ow} 
for a similar computation --- that the right action of $\Gamma$ on  $\calR(X,G)$ 
given by (\ref{eq-action-character-no-eta}) coincides with the action of $\Gamma$ on $\cM(X,G)$ defined in 
Section \ref{section-action} via the non-abelian Hodge correspondence. Recall
that here we are taking the character $\mu:\Gamma\to \C^*$ to be trivial.

For each lift $\theta:\Gamma\to\Aut(G)$ of $a$, let $\taut:\gamtt\to\Aut(\gt_0)$ and $\ct\in Z^1_{\taut}(\gamtt,Z(\gt_0))$ fitting in $\gs\cong\gt_0\times_{(\taut,\ct)}\gamtt$. Let $Y$ be the connected component of a $\gamtt$-bundle, with Galois group $\Gamma_Y\le\gamtt$. Then, by Proposition \ref{prop-polystability-extension-structure-group}, Theorem \ref{th-prym-narasimhan-ramanan-oscar-ramanan-higgs} and the non-abelian Hodge correspondence \ref{equivariant-nahc}, there is a morphism $\calR(Y,\gt_0,\Gamma_Y,\taut,\ct)/Z_{\gamtt}(\Gamma_Y)\to\calR(X,G)$. Here $\calR(Y,\gt_0,\Gamma_Y,\taut,\ct)$ is the moduli space  
consisting of  
 $G$-conjugacy classes of elements of   
$\Hom(\pi_1(X,\Gamma,x), G\times_{(\theta,c)}\Gamma)$
whose restriction to $\pi_1(X,\,x_1)$ is reductive, defined in Section \ref{section-non-abelian-hodge-twisted-equivariant}. Let $\wcalR(Y,\gt_0,\Gamma_Y,\taut,\ct)$ be its image. By Theorem \ref{th-prym-narasimhan-ramanan-oscar-ramanan-higgs}, the following theorem holds.

\begin{theorem}\label{th-prym-narasimhan-ramanan-character-varieties}
Let $\mu$ be trivial and fix $\theta\in\homgh$ lifting $a$. The following relations between character varieties hold:
\begin{enumerate}
    \item $$\bigcup_{[\beta]\in H^1_{\theta}(\Gamma,\Int(G)),\, q_{\beta\theta}(Y)=\alpha}\wcalR(Y,G^{\beta\theta}_0,\Gamma_Y,\tau^{\beta\theta},c^{\beta\theta})\subset\calR(X,G)^{\Gamma}. $$
    
    \item $$\calR_{\irr}(X,G)^{\Gamma}\subset\bigcup_{[\beta]\in H^1_{\theta}(\Gamma,\Int(G)),\, q_{\beta\theta}(Y)=\alpha}\wcalR(Y,G^{\beta\theta}_0,\Gamma_Y,\tau^{\beta\theta},c^{\beta\theta}),$$
    where $\calR_{\irr}(X,G)\subset\calR(X,G) $ is the subvariety of irreducible representations.
    
\end{enumerate}
\end{theorem}

\subsection{Action of a cyclic subgroup of the Jacobian on \texorpdfstring{$\mdl(X,\SL(n,\C))$}{M(X,SL(n,C))}}\label{section-example-generalize-narasimhan}
In this section we apply Theorem \ref{th-prym-narasimhan-ramanan-oscar-ramanan-higgs} to the case $G=\SL(n,\C)$ and $a$ trivial. In particular this shows that Theorem \ref{th-prym-narasimhan-ramanan-oscar-ramanan-higgs} gives the construction in \cite{narasimhan-ramanan} when applied to $\SL(n,\C)$-bundles, or vector bundles with trivial determinant. 

Let $L$ be a line bundle over $X$ with order $r$, which may be seen as an element of $H^1(X,\Z/r\Z)$, and assume that $n=rm$ is divisible by $r$ --- otherwise there are no fixed points, since $E\otimes L\cong E$ implies $(\det E)\otimes L^{n}\cong\det E$. Let $\Gamma<H^1(X,\Z/r\Z)$ be the subgroup generated by $L$ and consider a homomorphism
$$\mu:\Gamma\to\C^*;\,\gamma\mapsto\mug.$$
We want to identify the fixed points of the $\Gamma$-action on $\cM(X,\SL(n,\C))$ such that $L$ sends $(E,\phi)$ to $(E\otimes L,\mu_L\phi)$. 

The fixed points of the tensorization action of $\Gamma$ on $M(X,\GL(n,\C))$ are given in \cite[Section 4.1]{prym-narasimhan-ramanan}. The calculation for $\cM(X,\SL(n,\C))$ is similar apart from what concerns the Higgs fields, so we don't go into much detail. First of all, every class in $H^1(\Z/r\Z,\Int(\SL(n,\C)))$ is represented by an automorphism of the form $\theta:=\Int_D$, where
$$D:=\begin{pmatrix}
    I_{p_1}&0&\ldots&0\\
    0&\zeta I_{p_2}&\ldots&0\\
    \vdots&\vdots&\ddots&\vdots\\
    0&0&\ldots&\zeta^{r-1}I_{p_r}
    \end{pmatrix},$$
$\zeta$ is a primitive $r$-th root of unity and $p_1+\ldots+p_r=n$. The subgroup of fixed points $\SL(n,\C)^{\theta}$ is equal to 
$\mathrm{S}(\GL(p_1,\C)\times\GL(p_2,\C)\times\ldots\times\GL(p_d,\C))$, and $\SL(n,\C)_{\theta}$ consists of matrices $A\in\SL(n,\C)$ sending the $\zeta^i$-eigenspace of $D$ to the $\zeta^{i+k}$-eigenspace, for every $i$ and some $k$ depending on $A$. By Theorem \ref{th-fixed-points-oscar-ramanan-higgs} the smooth fixed point locus $\cM_*(X,\SL(n,\C))^{\Gamma}$ is empty unless the homomorphism
$$\cct:\SL(n,\C)_{\theta}\to\Z/r\Z$$
is surjective. This is true if and only if all the eigenspaces have the same dimension, i.e. $p_1=\dots=p_d=m$. We assume this from now on, so that the corresponding class $[\theta]$, where $\theta=\Int_D$, is the only one appearing in the decomposition of Theorem \ref{th-fixed-points-oscar-ramanan-higgs}.

Hence, $\SL(n,\C)_{\theta}$ is a $\Z/r\Z$-extension of $\SL(n,\C)^{\theta}$ generated by $\SL(n,\C)^{\theta}$ and the permutation matrix
$$S:=\xi\begin{pmatrix}
    0&I_m&0&\ldots&0\\
    0&0&I_m&\ldots&0\\
    \vdots&\vdots&\ddots&\ddots&\vdots\\
    0&0&0&\ldots&I_m\\
    I_m&0&0&\ldots&0\\
    \end{pmatrix},$$
where $\xi$ is a $r$-th root of $(-1)^{r-1}$ ensuring that $\det S=1$. Since $S^r=(-1)^{r-1}$, the 2-cocycle involved is a map
$$c:\Z/r\Z\times\Z/r\Z\to Z(\GL(n,\C)^{\theta})\cong (\C^*)^{r};(a,b)\mapsto
\begin{cases}
1\quad\text{if}\;a+b<r\\
(-1)^{r-1}\quad\text{if}\;a+b\ge r,
\end{cases}$$
where we are identifying the elements of $\Z/r\Z$ with integers between $0$ and $r-1$.
This provides an isomorphism
$$\nslm\cong \slm \times_{(\Int_S,c)}\Z/r\Z,$$
where we denote by $\Int_S$ the homomorphism $\Z/r\Z\to\Aut(\slm)$ such that the image of a generator $\zeta\in\Z/r\Z$ is equal to $\Int_S$, by abuse of notation.

Let 
$$p_L:X_L\to X$$ 
be the étale $r$-cover of $X$ corresponding to $L$ and let $\zeta$ be a generator of its Galois group $\Gamma_L$, isomorphic to $\Z/r\Z$. According to Theorem \ref{th-fixed-points-oscar-ramanan-higgs}, elements in $\mdl_*(X,\SL(n,\C))^{\Gamma}$ correspond to $(\Int_S,c)$-twisted $\Z/r\Z$-equivariant $(\slm,\sll(n,\C)^{\theta}_{\mu})$-Higgs pairs over $X_L$. Let $(F,\bullet)$ be a $(\Int_S,c)$-twisted $\Z/r\Z$-equivariant $\slm$-bundle. The corresponding vector bundle $F(\C^n)$ is a direct sum of vector bundles of rank $m$. 
The equivariant $\Z/r\Z$-action given by Proposition \ref{prop-associated-bundle-equivariant} using the natural embedding of $\SL(n,\C)_{\theta}$ in $\GL(n,\C)$ is such that the generator $\zeta\in\Z/r\Z$ sends $(e,v)\in F(\C^n)$ to $(e\bullet\gamma,S^{-1}v)$. Since $S$ exchanges the different copies of $\GL(m,\C)$ in $\SL(n,\C)^{\theta}$, the $\Z/r\Z$-equivariant action on $F(\C^n)$ exchanges its summands. Therefore, there is a decomposition $F(\C^n)\cong V\oplus \zeta^*V\oplus\dots\oplus\zeta^{*r-1}V$. If $E$ is the $\nslm$-bundle over $X$ corresponding to $(F,\bullet)$ via Theorem \ref{th-prym-narasimhan-ramanan-oscar-ramanan-higgs} then, by Proposition \ref{prop-associated-bundle-equivariant},
\begin{equation*}
    E(\C^n)\cong F(\C^n)/(\Z/r\Z)\cong p_{L*}V.
\end{equation*}
Conversely, using \ref{prop-associated-bundle-equivariant} one may also show that, for any vector bundle $W\to X$ with trivial determinant associated to a $\nslm$-bundle whose quotient by $\SL(n,\C)^{\theta}$ is isomorphic to $L$, $W$ is isomorphic to $p_{L*}V$ for some vector bundle $V\to X_L$ of rank $m$.

Now assume that $(F,\bullet)$ is equipped with a $\Z/r\Z$-invariant Higgs field 
$$\psi\in H^0(X_L,F(\sll(n,\C)^{\theta}_{\mu})\otimes K_X).$$ It is straightforward to see that $\sll(n,\C)_{\mu}^{\theta}$ is the vector space of traceless matrices $A$ sending the $\zeta^k$-eigenspace of $D$ to the $\zeta^k\mu(L)$-eigenspace, for each $k\in\Z$. The homomorphism $\pi_1(X)\to\C^*$ determined by $L$ induces a canonical isomorphism $\Gamma_L\cong \Gamma^*:=\Hom(\Gamma,\C^*)$. Regarding $\mu$ as an element of $\Gamma_L$ and using the $\Gamma_L$-equivariance, we find that the Higgs field on $V$ must be induced by homomorphisms $\zeta^{k*}V\to\mu^{*}\zeta^{k*}V\otimes K_L$ obtained by pulling back a homomorphism 
$$\psi:V\to\mu^{*}V\otimes K_L$$
by the elements of $\gal(X_L/X)$, where $K_L$ is the canonical bundle of $X_L$. 

In summary, the following proposition holds.
\begin{proposition}
A simple $\SL(n,\C)$-Higgs bundle $(E,\phi)$ with trivial determinant over $X$ is isomorphic to $(E\otimes L,\mu(L)\phi)$ if and only if its associated vector bundle is the pushforward of a vector bundle $V\to X_L$ of rank $m$ and $\phi$ is induced by a homomorphism
$$\psi:V\to\mu^{*}V\otimes K_X.$$
\end{proposition}

\begin{remark}
Given a vector bundle $V\to X_L$, the condition that the determinant of $p_{L*}V$ is trivial is equivalent to the norm of $\det(V)$ being isomorphic to $L^{(r-1)/2}$ \cite{narasimhan-ramanan} --- recall that the norm of $\det(V)$ is the quotient of $\det(V)\otimes \zeta^*\det(V)\otimes\dots\otimes\zeta^{*r-1}\det(V)$ by the equivariant permutation $\Z/r\Z$-action.
\end{remark}

\begin{corollary}\label{cor-U(m,m)}
Let $L$ be a line bundle of order 2 over $X$ with corresponding \'etale cover $X_L\to X$. Consider the involution of $\cM(X,\SL(2m,\C))$ sending each Higgs bundle $(E,\phi)$ to $(E\otimes L,-\phi)$. Then, a smooth point is fixed under this action if and only if its associated vector bundle is the pushforward of a vector bundle $V\to X_L$ of rank $m$ equipped with a homomorphism $\psi:V\to\zeta^*V\otimes K_L$, where $\zeta$ is the generator of $\gal(X_L/X)$ and $K_L$ is the canonical bundle of $X_L$.
\end{corollary}

\begin{remark}
With data $(V,\psi)$ as in Corollary \ref{cor-U(m,m)}, the higgs bundle $(V\oplus\zeta^*V,\psi\oplus\zeta^*\psi)$ is a $U(m,m)$-Higgs bundle on $X_L$.
\end{remark}

\subsection{Action of an arbitrary finite subgroup of the Jacobian on \texorpdfstring{$\mdl(X,\SL(n,\C))$}{M(X,SL(n,C))}}\label{section-jacobian}
Let $\Gamma\le J(X)\cong H^1(X,\C^*)$ be a finite subgroup and set $a$ and $\mu$ to be trivial. There is a $\Gamma$-action on $\cM(X,\SL(n,\C))$ given by tensorization. In this section we regard $\cM(X,\SL(n,\C))$ as the moduli space of Higgs bundles of rank $n$ and trivial determinant.

An \textbf{antisymmetric pairing} on $\llambda$ is a homomorphism $l:\llambda\to\llambda^*$, where $\llambda^*:=\Hom(\llambda,\C^*)$ and the associated pairing satisfies $\langle\ambda,\ambda\rangle=1$ for every $\ambda\in\llambda$. Given such an antisymmetric pairing $l$, a subgroup $\Delta\le\llambda$ is called \textbf{isotropic} if the pairing of $\Delta$ with itself is trivial. For each choice of antisymmetric pairing choose a maximal isotropic subgroup $\Delta\le\llambda$. Note that the embedding of $\Delta$ in $J(X)$ is an element of $\Hom(\Delta, H^1(X,\C^*))\cong H^1(X,\Delta^*)$, Hence it defines a connected \'etale cover $\pd:\xd\to X$. This has Galois group $\Delta^*$ by \cite[Proposition 5.6]{prym-narasimhan-ramanan}.

Consider the moduli space $\cM(\xd,\GL(n/\vert\Delta\vert,\C))$ of $\GL(n/\vert\Delta\vert,\C)$-Higgs bundles over $\xd$ --- of course, this is empty if $n/\vert\Delta\vert$ is not an integer. For each antisymmetric pairing $l$, denote by 
$\cM(\xd,\GL(n/\vert\Delta\vert,\C))^{l(\llambda)}$ the subvariety of $\cM(\xd,\GL(n/\vert\Delta\vert,\C))$ consisting of isomorphism classes of polystable Higgs bundles $(E,\phi)$ such that $(E,\phi)\cong (l(\ambda)^*E\otimes \pd^*\ambda,l(\ambda)^*\phi)$ for every $\ambda\in\llambda$, where $l(\ambda)$ is regarded as an element of $\Delta^*$ by restriction. There is a natural pushforward morphism
\begin{equation*}
p_{\Delta*}:\cM(\xd,\GL(n/\vert\Delta\vert,\C))^{l(\Gamma)}\to \cM(X,\GL(n,\C)).
\end{equation*}

For each antisymmetric pairing $l$ of $\llambda$, fix a maximal isotropic subgroup $\Delta\le \llambda$. Consider the 1-dimensional representation $\Delta^*\to\GL(\bigwedge^{\text{top}} \C[\Delta^*])$, induced by the permutation representation of $\Delta^*$. This is a character of $\Delta^*$, hence it determines an element $\delta_l\in\Delta$. Moreover, $\delta_l$ has order at most two. Given a $\GL(n/\vert\Delta\vert,\C)$-Higgs bundle $(E,\phi)$ over $\xd$, by an argument similar to the proof of \cite[Proposition 3.23]{nasser}, we conclude
\begin{equation*}
    \det p_{\Delta*}E=\Nmd(\det E)\otimes \delta_l,
\end{equation*}
where $\Nmd:J(\xd)\to J(X)$ is the norm map. Let
\begin{equation*}
\vartheta_{l}:=\Nmd\circ\det:\cM(\xd,\GL(n/\vert\Delta\vert))\to J(X)
\end{equation*}
and consider the subvariety $\vartheta_{l}^{-1}(\delta_l)^{l(\llambda)}\subset \cM(\xd,\GL(n/\vert\Delta\vert))^{l(\llambda)}$.

Our main result in this setting is the following.
\begin{theorem}\label{th-finite-subgroup-jacobian}
For each antisymmetric pairing $l$, choose a corresponding maximal isotropic subgroup $\Delta\le\Gamma$. Then the inclusions
$$\bigcup_{l}p_{\Delta*}\vartheta_{l}^{-1}(\delta_l)^{l(\llambda)}\subset\cM(X,\SL(n,\C))^{\llambda}$$
and
$$\cM_*(X,\SL(n,\C))^{\llambda}\subset \bigcup_{l}p_{\Delta*}\vartheta_{l}^{-1}(\delta_l)^{l(\llambda)}$$
hold.
Here the parameter $l$ runs over all antisymmetric pairings on $\llambda$ such that the order of $\Delta$ divides $n$.
\end{theorem}
\begin{proof}
Same as the proof of \cite[Theorem 5.15]{prym-narasimhan-ramanan}, taking account of the Higgs field.
\end{proof}



\subsection{Action of a line bundle of order 2 on \texorpdfstring{$\mdl(X,\Sp(2n,\C))$}{M(X,Sp(2n,C))}}\label{section-example-sp-finite}
Let $L\in H^1(X,\Z/2\Z)$ be a line bundle of order 2 and let $X_L$ be the corresponding étale cover of degree two over $X$. Consider the involution of $\cM(X,\Sp(2n,\C))$ which sends $(E,\phi)$ to $(E\otimes L,\pm\phi)$. In other words, the generator of $\Gamma$ is sent to $(L,1,\pm 1)\in H^1(X,\Z/2\Z)\rtimes\Out(\Sp(2n,\C))\times\C^*$. Note that $\Out(\Sp(2n,\C))$ is trivial and the centre of $\Sp(2n,\C)$ is $\Z/2\Z$, so this covers all the possible subgroups of order 2 of $H^1(X,\Z/2\Z)\rtimes\Out(\Sp(2n,\C))\times\C^*$.

Consider the embedding of $\Sp(2n,\C)$ in $\GL(2n,\C)$ associated to the standard symplectic form, given by the matrix
\begin{equation}\label{eq-S-symplectic-form}
    S:=
\begin{pmatrix}
0&I_m\\-I_m&0
\end{pmatrix}.
\end{equation}
According to \cite[Chapter X, Section 2.3]{helgason}, the conjugacy classes of involutions in 
$$\Aut(\Sp(2n,\C))/\Int(\Sp(2n,\C))$$ 
are represented by $\Int_S$ and $\Int_{K_{p,q}}$, where 
$$K_{p,q}:=i\begin{pmatrix}
    -I_p & 0 & 0 & 0\\
    0 & I_q & 0 & 0\\
    0 & 0 & -I_p & 0\\
    0 & 0 & 0 & I_q 
    \end{pmatrix}$$
    and $p+q=n$.
The matrix $S$ is conjugate to
$$
S'=\begin{pmatrix}
    i & 0\\
    0 & -i
    \end{pmatrix}.$$
Set $\theta:=\Int_{S'}$ and $\tau_{p,q}:=\Int_{K_{p,q}}$. By Theorem \ref{th-fixed-points-oscar-ramanan-higgs},
\begin{align*}
    \cM_*(X,\Sp(2n,\C))^{\Gamma}=\\\cM_*(X,\Sp(2n,\C))\cap\left(\widetilde{\cM}_{L}(X,\Sp(2n,\C)_{\theta},\lieg^{\theta}_{\pm1})\cup\bigcup_{p+q=n}\widetilde{\cM}_{L}(X,\Sp(2n,\C)_{\tau_{p,q}},\lieg^{\tau_{p,q}}_{\pm1})\right),
\end{align*}
where the subscript $L$ indicates that the quotient by $\Sp(2n,\C)^{\theta}$ and $\Sp(2n,\C)^{\tau_{p,q}}$, respectively, is a $\Z/2\Z$-bundle isomorphic to $L$. Since the group $\Sp(2n,\C)_{\tau_{p,q}}$ is equal to $\Sp(2n,\C)^{\tau_{p,q}}$ unless $p=q$, we actually have 
\begin{align}\label{eq-sp-fixed-points}
    \cM_*(X,\Sp(2n,\C))^{\Gamma}=\\
    \cM_*(X,\Sp(2n,\C))\cap\left(\widetilde{\cM}_{L}(X,\Sp(2n,\C)_{\theta},\lieg^{\theta}_{\pm1})\cup\widetilde{\cM}_{L}(X,\Sp(2n,\C)_{\tau},\lieg^{\tau}_{\pm1})\right),\nonumber
\end{align}
where $\tau:=\tau_{n/2,n/2}$ and the second component is only present if $n$ is even. From now on, whenever $n/2$ appears it will be implicit that $n$ is even.

By Theorem \ref{th-prym-narasimhan-ramanan-oscar-ramanan-higgs}, each of the components on the right hand side of (\ref{eq-sp-fixed-points}) is the image of a moduli space of twisted equivariant bundles with structure group $\Sp(2n,\C)^{\theta}$ and $\Sp(2n,\C)^{\tau}$ respectively. From the definitions of $S'$ and $K_{n/2,n/2}$, which are diagonal with two eigenvalues, we know that these fixed points are the symplectic matrices preserving two subspaces of $\C^{2n}$ of dimension $n$ --- of course, the subspaces are different for $\theta$ and $\tau$. The $-1$-weight spaces in $\lie{sp}(2n,\C)$ for the action of $\theta$ and $\tau$ consist of matrices which permute the corresponding subspaces. According to \cite[Section 4.2]{prym-narasimhan-ramanan}, the underlying twisted equivariant principal bundles are in correspondence with the following objects, corresponding to $\theta$ and $\tau$ respectively.
\begin{enumerate}
    \item Vector bundles $V\to X_L$ of rank $n$ equipped with an isomorphism $f:V\xrightarrow{\sim}\lambda^*V^*$ satisfying $\lambda^*f^*=-f$, where $\lambda\in\Gamma^*=\gal(X_L/X)$ is the generator. The corresponding vector bundle on $X$ is the pushforward $p_{L*}V$, and the symplectic form is the one induced by the standard one on $V\oplus V^*$ via $f$.
    \item Symplectic vector bundles $E$ of rank $n$ on $X_L$. Again, the corresponding vector bundles on $X$ are the pushforwards.
\end{enumerate}
Let $K_L$ is the canonical bundle of $X_L$. The Higgs field has values in the $\pm 1$-weight space of the corresponding involution. In the $+1$ case this means that it preserves $E$, in the sense that it is given by a homomorphism $E\to E\times K_L$, whereas in the $-1$ case it is determined by a $\lambda^*$-invariant homomorphism $E\to\lambda^*E\otimes K_L$, i.e. the Higgs field exchanges $E$ and $\lambda^*E$. Summing up, we obtain the following result.

\begin{proposition}\label{prop-Sp-example}
Let $L$ be a line bundle of order 2 over $X$ with corresponding \'etale cover $p_L:X_L\to X$ and Galois group $\Gamma_L$ generated by $\lambda$. Let $K_L$ be the canonical bundle of $X_L$. Consider the automorphism of $\cM(X,\Sp(2n,\C))$ sending $(E,\phi)$ to $(E\otimes L,\mu\phi)$, where $\mu=\pm 1$. Then the smooth fixed points in $\cM(X,\Sp(2n,\C))$ are pushforwards by $p_L$ of the following two types of Higgs bundles with extra structure.
\begin{enumerate}
    \item A vector bundle $V$ of rank $n$ over $X_L$ equipped with a homomorphism $\psi:V\to\mu^*V\otimes K_L$ and an isomorphism $f:(V,\psi)\xrightarrow{\sim}\lambda^*(V^*,\psi^*)$ compatible with $\psi$ such that $\lambda^*f^*=-f$, where $\mu$ is regarded as an element of $\Gamma_L\cong\Z/2\Z$.
    \item (Assuming that $n$ is even,) a symplectic vector bundle $V$ of rank $n/2$ over $X$, equipped with a homomorphism $\psi:V\to\mu^*V\otimes K_L$ compatible with the symplectic form.
\end{enumerate}
Conversely, such pushforwards are fixed points.
\end{proposition}

\begin{remark}
Let $\sigma$ be the compact conjugation of $\Sp(2n,\C)$ whose fixed point subgroup is the compact symplectic group $\Sp(2n,\C)\cap U(n)$. Under the given embedding in $\GL(2n,\C)$, this is just the composition of conjugating, transposing and inverting, and it can be seen with a straightforward computation that it commutes with $\theta$ and $\tau$. Then $\sigma\theta$ is in the orbit of $\sigma\Int_S$ by the conjugation action of $\Int(\Sp(2n,\C))$, where $S$ is given by (\ref{eq-S-symplectic-form}). Since $\Int_S$ is just taking transpose and inverse on $\Sp(2n,\C)$, the antiholomorphic involution $\sigma\Int_S$ is just conjugation of matrices. The fixed point real subgroup $\Sp(2n,\C)^{\sigma\theta}$ is therefore a conjugate of $\Sp(2n,\R)$. The fixed point subgroup $\Sp(2n,\C)^{\sigma\tau}$ of the involution $\tau$ is equal to $\Sp(n/2,n/2)$ \cite[Chapter X, Section 2.3]{helgason}. Let $\NSp(2n,\R)$ and $\NSp(n/2,n/2)$ denote the respective normalizers in $\Sp(2n,\C)$. Their quotients by their connected components, which are $\Sp(2n,\R)$ and $\Sp(n/2,n/2)$ respectively, are both isomorphic to $\Z/2\Z$. According to \cite[Section 8.2]{PR} the non-abelian Hodge correspondence provides homeomorphisms 
\begin{align*}
    &\cM(X_L,\Sp(2n,\C)^{\theta},\lieg^{\theta}_{-1})\cong \calR(X_L,\Sp(2n,\R))\andd\\ &\cM(X_L,\Sp(2n,\C)^{\tau},\lieg^{\tau}_{-1})\cong \calR(X_L,\Sp(n/2,n/2)),
\end{align*}
where $\calR(X_L,G)$ denotes the character variety of $\pi_1(X_L)$ with values in $G$.
Combining this with the discussion above, we see that fixed points of the $\Z/2\Z$-action on $\mdl(X,\Sp(2n,\C))$ such that $-1$ sends $(E,\phi)$ to $(E\otimes L,-\phi)$ correspond to twisted equivariant $\Sp(2n,\R)$ and $\Sp(n/2,n/2)$-Higgs bundles on $X_L$.
\end{remark}

\subsection{Tensorization by a line bundle combined with dualization for \texorpdfstring{$G=\SL(n,\C)$}{G=SL(n,C)}}

If $n\ge 3$, the group of outer automorphisms of $\SL(n,\C)$ is isomorphic to $\Z/2\Z$. A lift of its generator $a$ is given by the automorphism $\theta$ sending $A$ to $A^{t-1}$. Given an arbitrary line bundle $L$, 
$$L\theta(L)=LL^*\cong\oo,$$
the trivial bundle. Hence, there is a homomorphism
$$\Gamma:=\Z/2\Z\to H^1(X,\C^*)\rtimes\Out(\SL(n,\C));\,-1\mapsto(L,a).$$
According to \cite[Chapter X, Section 2.3]{helgason}, every class of $H^1_{\theta}(\Z/2\Z,\PGL(n,\C))$ is represented either by $\theta$ or, in case that $n=2m$ is even, by $\nu:=\Int_S\theta$, where $S$ is given by (\ref{eq-S-symplectic-form}).

By \cite[Proposition 2.10]{PR}, the groups $\SL(n,\C)_{\theta}$ and $\SL(n,\C)_{\nu}$ defined in Section \ref{section-Gtheta} are equal to the normalizers $\NSO(n,\C)$ and $\NSp(2m,\C)$ of $\SO(n,\C)$ and $\Sp(2m,\C)$ in $\SL(n,\C)$, respectively. Recall from Section \ref{section-Gtheta} that $c_{\theta}$ is the homomorphism which sends $g\in \NSO(n,\C)$ to $\theta(g)g^{-1}$, an element of the centre of $\SO(n,\C)$, and we may define $c_{\nu}$ similarly. For the symplectic group consider the group $\mathcal A$ of invertible matrices
$$A_\zeta:=\begin{pmatrix}
I_m&0\\0&\zeta^{-1}I_m
\end{pmatrix},$$
where $\zeta\in\C^*$. They satisfy
$$c_{\nu}(A_\zeta)=\zeta I_n,$$
so that $c_{\nu}$ restricts to an isomorphism $\mathcal A\cong\C^*$. Since $\C^*$ coincides with the centre of $\GL(2m,\C)$, every coset of the normalizer of $\Sp(2m,\C)$ in $\GL(2m,\C)$ must be represented by a matrix $A_{\zeta}$. Classes in the quotient $\NSp(2m,\C)/\Sp(2m,\C)$ must then be in bijection with elements of $\mathcal A$ with determinant 1, which are exactly the matrices $A_\zeta$ such that $\zeta$ is an $m$-th root of unity. Consider the homomorphism
$$\Int_A:\Z/m\Z\to\Int(\Sp(2m,\C));\,\zeta\mapsto \Int_{A_{\zeta}},$$
where $\zeta$ is regarded as an $m$-th root of 1.
Then we obtain an isomorphism
\begin{equation}\label{eq-normalizer-sp}
    \NSp(2m,\C)\cong \Sp(2m,\C)\times_{(\Int_{A},1)}\Z/m\Z=\Sp(2m,\C)\rtimes_{\Int_{A}}\Z/m\Z.
\end{equation}

For the orthogonal group, given $\zeta\in \C^*$, 
$$c_{\theta}(\zeta I_n)=\zeta^{-2}I_n.$$
Therefore, by varying $\zeta$ we obtain the whole centre of $\GL(n,\C)$. Classes in the quotient $\NSO(n,\C)/\SO(n,\C)$ are represented by scalar matrices with determinant 1, which are just $n$-th roots of unity. When $n$ is odd, taking the $-2$-power is an automorphism of the group of $n$-th roots $\Z/n\Z$, so that $\NSO(n,\C)$ is a semidirect product of $\SO(n,\C)$ with $\Z/n\Z$ given by conjugation by an element of the centre. Since this conjugation is trivial, $\NSO(2m+1,\C)\cong\SO(2m+1,\C)\times \Z/n\Z$.
However, when $n=2m$ is even, the restriction of $c_{\theta}$ to $\Z/n\Z$ is not an isomorphism and so a 2-cocycle appears. In this case we may restrict to $n$-th roots of unity with argument less than $\pi$ to give a bijection between $\NSO(2m,\C)$ and $\SO(2m,\C)\times\Z/m\Z$. Since conjugation by elements in the centre is trivial, so is the characteristic homomorphism of the extension $\NSO(2m,\C)$ of $\SO(2m,\C)$. The cocycle $c$ is defined as follows: define a map
$$\sigma:\Z/n\Z\to\Z/2\Z$$
which assigns an $n$-th root to $1$ if its argument is less than $\pi$ and $-1$ otherwise. Then set
$$c:\Z/n\Z\times\Z/n\Z\to \Z/2\Z;\, (\delta,\delta')\mapsto \sigma(\delta^{\frac 12}\delta'^{\frac 12}),$$
where $\delta^{\frac 12}$ is the unique square root of $\delta$ with argument less than $\pi$.
Then there are isomorphisms
\begin{equation}\label{eq-normalizer-so}
    \NSO(n,\C)\cong\begin{cases}  \SO(n,\C)\times \Z/n\Z    &   \text{if $n$ is odd}\\
                                \SO(n,\C)\times_{(1,c)} \Z/m\Z   &   \text{if $n=2m$,}
\end{cases}
\end{equation}
whose inverse isomorphism sends $(M,\zeta)\in \SO(n,\C)\times \Z/n\Z$ (or $\SO(n,\C)\times_{(1,c)} \Z/m\Z$) to $\zeta^{-\frac12}M$, where $\zeta^{-\frac12}$ is the unique root of $\zeta^{-1}$ in $\Z/n\Z$ (or the one with argument less than $\pi$, respectively).

\begin{remark}
Note that the same reasoning that gives the even case in (\ref{eq-normalizer-so}) also provides an isomorphism
\begin{equation*}
    \NSp(2m,\C)\cong \Sp(2m,\C)\times_{(1,c)}\Z/m\Z.
\end{equation*}
Hence there are two descriptions of $\NSp(2m,\C)$, one with trivial 2-cocycle and one with a trivial lift of the characteristic homomorphism.
\end{remark}

Let $X_L$ be the étale cover of $X$ associated to $L$, which has Galois group $\Gamma_L\cong\Z/r\Z$. Note that $\Gamma_L$ is canonically embedded in $\C^*$ as the subgroup of $r$-th roots of 1, since it is the Galois group of a reduction of structure group of the $\C^*$-bundle corresponding to $L$. We denote by $\mdl(X_L,\SO(n,\C),{\Gamma_L})$ and --- if $n=2m$ is even --- $\mdl(X_L,\Sp(2m,\C),{\Gamma_L})$ the varieties of $\SO(n,\C)$ and $\Sp(2m,\C)$-Higgs bundles whose corresponding orthogonal or symplectic vector Higgs bundle $(V,\psi)$ is equipped with a $\Gamma_L$-equivariant action such that 
\begin{equation}\label{eq-form-multiplied}
    \langle v\cdot\zeta,v'\cdot\zeta\rangle=\zeta\langle v,v\rangle
\end{equation}
for each $v,v'\in V$ and $\zeta\in\Gamma_L\cong\Z/r\Z$, where $\langle \cdot,\cdot\rangle$ denotes the orthogonal or symplectic form. Let $\wcM(X_L,\SO(n,\C),{\Gamma_L})$ and $\wcM(X_L,\Sp(2m,\C),{\Gamma_L})$ be their images in $\cM(X,\SL(n,\C))$ given by taking the quotient by the $\Gamma_L$-equivariant action. 

\begin{proposition}
    Let $\Gamma:=\Z/2\Z$ act on $\mdl(X,\SL(n,\C))$ in such a way that the generator of $\Gamma$ sends $(E,\phi)$ to $(E^*\otimes L,\phi^*)$, where $L$ is a line bundle of order $r$. Then $\mdl_*(X,\SL(n,\C))^{\Gamma}$ is empty unless $r$ divides $n$, if $n$ is odd, or unless $r$ divides $m=n/2$ otherwise. In this case
    \begin{equation*}
        \wcM(X_L,\SO(n,\C),{\Gamma_L})\cup \wcM(X_L,\Sp(2m,\C),{\Gamma_L})\subset \mdl(X,\SL(n,\C))^{\Gamma}
    \end{equation*}
    and
    \begin{equation*}
        \mdl_*(X,\SL(n,\C))^{\Gamma}\subset \wcM(X_L,\SO(n,\C),{\Gamma_L})\cup \wcM(X_L,\Sp(2m,\C),{\Gamma_L}),
    \end{equation*}
    where $m:=n/2$ and $\wcM(X_L,\Sp(2m,\C),{\Gamma_L})$ is empty unless $n$ is even.
\end{proposition}

\begin{proof}
    Throughout this proof we deal with orthogonal and symplectic bundles at the same time by making use of square brackets, assuming that $n=2m$ is even in the symplectic case. Consider a twisted $\Gamma_L$-equivariant $\SO(n,\C)$ [$\Sp(2m,\C)$]-Higgs bundle $(F,\bullet,\psi)$ over $X_L$. Here we are omitting the twisting data. This data is $(1,1)$ or $(1,c)$ for $\SO(n,\C)$ depending on whether $n$ is odd or even respectively [$(\Int_A,1)$ for $\Sp(2m,\C)$]. The corresponding vector Higgs bundle $(F(\C^n),\psi)$ inherits a $\Gamma_L$-equivariant action by Proposition \ref{prop-associated-bundle-equivariant}, given by 
    \begin{equation*}
        F(\C^n)\ni(e,v)\cdot\zeta:=(e\bullet\zeta,\zeta^{\frac12}v)\,[(e\bullet\zeta,A_{\zeta}^{-1} v)]
    \end{equation*} 
    for each $\zeta\in\Gamma_L$.
    The linear transformation $\zeta^{\frac12}I_n$ [$A_\zeta^{-1}$] multiplies the usual orthogonal [symplectic] form on $\C^n$ by $\zeta$, hence the $\Gamma_L$-equivariant action on $F(\C^n)$ satisfies \ref{eq-form-multiplied}.

    Conversely, take a vector Higgs bundle $(V,\psi)$ equipped with an orthogonal [symplectic] form and a $\Gamma_L$-equivariant action satisfying \ref{eq-form-multiplied}. The bundle of frames $P$ of $V$ inherits a $\Gamma_L$-equivariant action, which we denote by $*$. It also has a reduction of structure group $F$ to $\SO(n,\C)$ [$\Sp(2m,\C)$], given by the orthogonal [symplectic] form. By (\ref{eq-form-multiplied}), $F* \zeta=F\zeta^{\frac 12}I_n$ [$FA_{\zeta}^{-1}$] for each $\zeta\in \Gamma_L$. This is because the coset of $\SO(n,\C)$ [$\Sp(2m,\C)$] in $\NSO(n,\C)$ [$\NSp(2m,\C)$] consisting of matrices which multiply the standard orthogonal [symplectic] form by $\zeta$ is equal to $\SO(n,\C)\zeta^{\frac 12}I_n$ [$\Sp(2m,\C)A_{\zeta}^{-1}$]. Therefore, the corresponding automorphism of $P$ sending $p\in P$ to $p\bullet\zeta:=(p*\zeta) \zeta^{-\frac 12}I_n$ [$(p*\zeta) A_{\zeta}$] preserves $F$. By Proposition \ref{prop-associated-bundle-equivariant}, $F$ inherits a twisted $\Gamma_L$-equivariant action.

    The proposition now follows from Theorems \ref{th-fixed-points-oscar-ramanan-higgs} and \ref{th-prym-narasimhan-ramanan-alpha-trivial-higgs}, Equations (\ref{eq-normalizer-sp}) and (\ref{eq-normalizer-so}) and Proposition \ref{prop-associated-bundle-equivariant}.
\end{proof}

\begin{remark}
    If $r$ is odd, $\Gamma_L$-equivariant actions on an orthogonal or symplectic bundle satisfying (\ref{eq-form-multiplied}) are in bijection with $\Gamma_L$-equivariant actions preserving the form, via rescaling the action of each $\zeta\in\Gamma_L$ by $\zeta^{-\frac12}$. Note that the square root homomorphism is well defined on $\Gamma_L$ in this case, since $\Gamma_L$ does not contain $-1$. 
\end{remark}

\section{Fixed points for general \texorpdfstring{$\Gamma$}{Gamma}}\label{section-general}

Throughout this section we fix a finite group $\Gamma$, a compact Riemann surface $X$ and a connected semisimple complex Lie group $G$ with centre $Z$ and Lie algebra $\lie g$.

\subsection{Étale covers and lifts of \texorpdfstring{$a$}{a}}\label{section-group-theory}
Let $\eta:\Gamma\to\Aut(X)$ and $a:\Gamma\to\Out(G)$ be homomorphisms. 
Pick a homomorphism
$$\theta:\ker\eta\to\Aut(G)$$
lifting $a\vert_{\ker\eta}$. With definitions as in Section \ref{section-Gtheta} there is an extension
\begin{equation}\label{eq-extension-connected-component-gt}
    1\to\gt_0\to\gt\xrightarrow{\pt}\gamtz\to1,
\end{equation}
restricting (\ref{eq-extension-connected-component}), where $\gt_0$ is the connected component of the identity and $\gamtz$ is a (finite) subgroup of $\gamtt$. Let $\otau:\gamtz\to\Out(\gt_0)$ be the characteristic homomorphism of (\ref{eq-extension-connected-component-gt}). By \cite{BGGM} there exists a lift $\tau:\gamtz\to\Aut(\gt_0)$ of the characteristic homomorphism of (\ref{eq-extension-connected-component-gt}), and consequently by Proposition \ref{prop-extensions-isomorphic-twisted-group} we can find $c\in\zzgt$ such that there is an isomorphism $\gt\cong\gt_0\times_{(\tau,c)}\gamtz$ as extensions of $\gt_0$, i.e. there is a commutative diagramme
\begin{equation*}
    \begin{tikzcd}[ar symbol/.style = {draw=none,"#1" description,sloped},
  isomorphic/.style = {ar symbol={\cong}},
  equals/.style = {ar symbol={=}},
  ]
        1\arrow[r]\ar[equal]{d}  &   \gt_0\arrow[r]\ar[equal]{d} &   \gt\arrow[r]\arrow[d,"\sim" labl]   & \gamtz  \arrow[r]\ar[equal]{d} & 1\ar[equal]{d}\\
        1 \arrow[r]  &   \gt_0 \arrow[r]  &   \gt_0\times_{(\tau,c)}\gamtz \arrow[r]    &     \gamtz \arrow[r]  &   1
    \end{tikzcd}.
\end{equation*}

Let $\pt:\gt\to \gamtz$ be the natural surjection. Fix a subgroup $\Lambda\le \gamtz$, set $\gtl:=\pt^{-1}(\Lambda)$ and let $p:Y\to X$ be a connected $\Lambda$-bundle over $X$. Consider the subgroup $\wgam\le\Aut(Y)$ lifting $\eta(\Gamma)$. This contains $\Lambda=\gal(Y/X)$ as a normal subgroup, since $\Lambda$ is the kernel of the projection $\wgam\to\eta(\Gamma)$. Let 
\begin{equation*}
    \wgame:=\{(\gamma,\wga)\in\Gamma\times\wgam\suhthat \etag=p(\wga)\}.
\end{equation*}
This contains the subgroup $\ker\eta\times 1$, a copy of $\ker\eta$. Let $p_{\Gamma}:\wgame\to\Gamma$ be the projection on the first factor. The following commutative diagramme holds:
\begin{equation}\label{eq-extension-wgam}
    \begin{tikzcd}[
      ar symbol/.style = {draw=none,"#1" description,sloped},
      equals/.style = {ar symbol={=}},
      ]
        &   &   1\arrow[d]&   1\arrow[d]&   \\
        &   &   \Lambda\arrow[d]\ar[equal]{r}&   \Lambda\arrow[d]&   \\
        1\arrow[r]&  \ker\eta\ar[equal]{d}\arrow[r] &   \wgame\arrow[d,"p_{\Gamma}"]\arrow[r]&   \wgam\arrow[d,"p_{\eta(\Gamma)}"]\arrow[r]&   1\\
        1\arrow[r]&   \ker\eta\arrow[r]&   \Gamma\arrow[d]\arrow[r,"\eta"]&   \eta(\Gamma)\arrow[d]\arrow[r]&   1\\
        &   &   1&   1&   \\
    \end{tikzcd},
\end{equation}
whose rows and columns are exact.


We say that a map $\tilde\tau:\wgame\to\Aut(G)$ such that $\tilde\tau(\wgam)$ preserves $Z(\gt_0)$ is a \textbf{$c$-twisted homomorphism} if it satisfies
\begin{equation}\label{eq-c-twisted-hom}
    \tilde\tau_{\gamma,\wga}\tilde\tau_{\gamma',\wga'}=\Int_{c(\gamma,\gamma')}\tilde\tau_{\gamma\gamma',\wga\wga'}
\end{equation}
for every $(\gamma,\wga)$ and $(\gamma',\wga')\in\wgame$, where $c\in Z^2_{\tilde\tau}(\wgam,Z(\gt_0))$.
Equivalently, according to the notation of Section \ref{section-twisted-G-Gamma-action}, it determines a $c$-twisted left $(G,\Gamma)$-action on $G$, where $G$ acts on itself by conjugation. Similarly, there is a notion of $c$-twisted homomorphism $\wgam\to \Aut(\gt)$. Denote by $\Hom_c(A,B)$ the set of $c$-twisted homomorphisms from $A$ to $B$ and let $\Aut(G)^{\gtl}$ be the set of automorphisms which preserve $\gtl$. There is a restriction map $r_{\ker\eta}:\Hom_c(\wgame,\Aut(G))\to\Hom(\ker\eta,\Aut(G))$. The image consists of homomorphisms because $c$ is trivial on $\ker\eta$ because of (\ref{eq-extension-wgam}). Let $\Hom_{\theta,c}(\wgame,\Aut(G)):=\Hom_c(\wgame,\Aut(G)^{\gtl})\cap r_{\ker\eta}^{-1}(\theta)$. 

Given an element $\tilde\tau\in\Hom_{\theta,c}(\wgame,\Aut(G))$, the restriction of $\tilde\tau(\gamma,\hat\gamma)$ to $\gtl$ for each $(\gamma,\hat\gamma)\in\wgame$ provides a map $\tilde\tau\vert_{\gtl}:\wgame\to \Aut({\gtl})$. Moreover,  $\tilde\tau\vert_{\gtl}$ factors through the projection $\wgame\to\wgam$. Indeed, given $\tilde\tau\in \Hom_{\theta,c}(\wgame,\Aut(G))$, $(\gamma,\wga)\in\wgame$, $\gamma'\in\ker\eta$ and $g\in\gtl$,
\begin{equation*}
    \tilde\tau_{(\gamma\gamma',\wga)}(g)=\Int_{c(\wga,1)}^{-1}\tilde\tau_{(\gamma,\wga)}\tilde\tau_{\gamma'}(g)=\tilde\tau_{(\gamma,\wga)}\theta_{\gamma'}(g)=\tilde\tau_{(\gamma,\wga)}(g).
\end{equation*} 
Therefore, there is a map $r_{\gtl}:\Hom_{\theta,c}(\wgame,\Aut(G))\to\Hom_c(\wgam,\Aut(\gtl))$.

Note that any automorphism of $\gtl$ preserves the connected component $\gt_0$ of $\gtl$, hence there is also a map $r_{\gt_0}:\Hom_c(\wgam,\Aut(\gtl))\to\Hom(\wgam,\Aut(\gt_0))$, where the fact that $Z(\gt_0)$ acts trivially on $\gt_0$ by conjugation implies that the image consists of (honest) homomorphisms. 
Conversely, given a homomorphism $\tau:\wgam\to\Aut(\gt_0)$ and a 2-cocycle $c\in Z^2_{\tau}(\wgam,Z(\gt_0))$ such that $\gt_0\times_{(\tau,c)}\Lambda$ and $\gtl$ are isomorphic as extensions of $\gt_0$, there is an extension of $\tau$ to a map $e_{\tau,c}:\wgam\to\Aut(\gtl)$ given by
\begin{equation*}
    \begin{tikzcd}[ar symbol/.style = {draw=none,"#1" description},
    isomorphic/.style = {ar symbol={\cong}},
      ]
        \wgam\arrow[r]\arrow[rrrr,bend right=15,"e_{\tau,c}", swap] & \gt_0\times_{(\tau,c)}\wgam\arrow[r] & \Int(\gt_0\times_{(\tau,c)}\wgam)\arrow[r] & \Aut(\gt_0\times_{(\tau,c)}\Lambda)\ar[isomorphic]{r} &  \Aut(\gtl),
    \end{tikzcd}
\end{equation*}
where the first map sends $\gamma\in\Gamma$ to $(1,\gamma)\in\gt_0\times_{(\tau,c)}\wgam$ and the existence of the third map follows from the fact that $\Lambda$ is normal in $\wgam$. The map $\etc$ is a $c$-twisted homomorphism, since $\Int_{(1,\gamma)}\Int_{(1,\gamma')}=\Int_{c(\gamma,\gamma')}\Int_{(1,\gamma\gamma')}$ by definition of the group multiplication on $\gt_0\times_{(\tau,c)}\wgam$.

In summary, there is the following diagramme:
\begin{equation}\label{eq-restriction-extension maps}
    \begin{tikzcd}
        \Hom(\wgame,\Out(G)) &
        \Hom_{\theta,c}(\wgame,\Aut(G)) \arrow[l,"q_*"]\arrow[r,"r_{\gtl}"] & \Hom_c(\wgam,\Aut(\gtl))\arrow[d,"r_{\gt_0}"]\\
        \Hom(\Gamma,\Out(G))\arrow[u,"p_{\Gamma}^*"]&       &   \Hom(\wgam,\Aut(\gt_0))
    \end{tikzcd},
 \end{equation}
where $q_*$ and $p_{\Gamma}^*$ are induced by the natural surjection $q:\Aut(G)\to\Out(G)$ and the projection $p_{\Gamma}:\wgame\to\Gamma$, respectively. For each homomorphism $\tau:\wgam\to\Aut(\gt_0)$ and each 2-cocycle $c\in Z^2_{\tau}(\wgam,Z(\gt_0))$ as above we set 
\begin{align*}
    &\Hom_{\theta,\tau,c}(\wgame,\Aut(G)):=r_{\gt_{\Lambda}}^{-1}(e_{\tau,c})\subset\Hom_{\theta,c}(\wgame,\Aut(G))\andd\\
    &\Hom_{\theta,\tau,c}(\wgame,\Out(G)):=q_*(\Hom_{\theta,\tau,c}(\wgame,\Aut(G))).
\end{align*}

The natural surjection $\Aut(G)^{\gtl}\to \Aut(G)^{\gtl}/\Int^G(\gt_0)$ restricts to a natural homomorphism $\gt_0\to\Aut(G)^{\gtl}$ given by conjugation, whose image is a normal subgroup of $\Aut(G)^{\gtl}$ since $\gt_0$ is preserved by every element of $\Aut(G)^{\gtl}$. Consider the map
\begin{equation*}
    p_{\theta,\tau,c}:\Hom_{\theta,\tau,c}(\wgame,\Aut(G))\to \Hom(\wgame,\Aut(G)^{\gtl}/\Int^G(\gt_0)).
\end{equation*}
This is well defined because the image of $c$ is in $\gt_0$, and so $c$-twisted homomorphisms are sent to (honest) homomorphisms.
We set 
\begin{equation*}
    \Hom_{\theta,\otau}(\wgame,\Aut(G)^{\gtl}/\Int^G(\gt_0)):=p_{\theta,\tau,c}(\Hom_{\theta,\tau,c}(\wgame,\Aut(G))).
\end{equation*}
Alternatively, $\Hom_{\theta,\otau}(\wgame,\Aut(G)^{\gtl}/\Int^G(\gt_0))$ consists of the homomorphisms $\nu:\wgame\to\Aut(G)^{\gtl}/\Int^G(\gt_0)$ such that the restriction of $\nu\vert_{\Lambda}$ to $\gt_0$ is equal to $\otau$ and $\nu\vert_{\ker\eta}$ is equal to the composition 
\begin{equation*}
    \ker\eta\xrightarrow{\theta}\Aut(G)^{\gt_0}\to \Aut(G)^{\gtl}/\Int^G(\gt_0).
\end{equation*}


Now let $\mu:\Gamma\to\C^*$ be a character, and let $\tilde\tau:\wgame\to\Aut(G)$ be a $c$-twisted homomorphism in $\Hom_{\theta,\tau,c}(\wgame,\Aut(G))$ preserving $\liegm$. Then there exists an associated left $(\tau,c)$-twisted $(\gt_0,\wgame)$-action on $\liegm$, given by
\begin{equation}\label{eq-action-tildetau}
    \rhotm:=\pg^*(\omu)\tilde\tau:\wgame\to\GL(\liegm);\,(\gamma,\wga)\mapsto (v\mapsto\omu_{\gamma}\tilde\tau_{\gamma,\wga}(v)).
\end{equation}
Moreover, since by definition of $\liegm$ and the fact that $r_{\ker\eta}(\tilde\tau)=\theta$ the action is trivial on $\ker\eta$, and $c$ only depends on the coset in $\wgam$, this factors through a $(\tau,c)$-twisted $(\gt_0,\wgam)$-action on $\liegm$ which we also call $\rhotm$. 

\begin{remark}\label{remark-not-preserve-liegm}
    In Sections \ref{section-fixed-points-general-alpha-trivial} and \ref{section-general-theorem} we will encounter $c$-twisted homomorphisms $\tilde\tau$ living in the set $\Hom_{\theta,\tau,c}(\wgame,\Aut(G))$ which do not necessarily preserve $\liegm$. However, we still obtain a well defined linear homomorphism 
\begin{equation*}
    \rhotm:\wgam\to\Hom(\liegm,\lie g).
\end{equation*}
Given a $\gt_0$-bundle $F$ over $Y$, this induces a map $H^0(Y,F(\liegm)\otimes K_Y)\to H^0(Y,F(\lie g)\otimes K_Y)$. Hence there is a notion of \textbf{$(\tau,c,\rhotm)$-twisted $\wgam$-equivariant $(\gt_0,\liegm)$-Higgs pair} over $Y$, given precisely as in Definition \ref{def-twisted-equivariant-higgs-pairs} by regarding the images of the Higgs field under the action of $\wgam$ as sections of $F(\lie g)\otimes K_Y$ and imposing that they actually live in $F(\liegm)\otimes K_Y$ and are equal to the Higgs field.
\end{remark}

\subsection{Fixed points for trivial \texorpdfstring{$\alpha$}{alpha}}\label{section-fixed-points-general-alpha-trivial}

Fix a homomorphism
$$\Gamma\to\Aut(X)\times\Out(G)\times\C^*;\,\gamma\mapsto(\fg,\ag,\mug).$$ Throughout this section we keep the notation of Section \ref{section-group-theory}. What follows may be regarded as a refinement of the results of Section \ref{section-alpha-trivial}.

\begin{proposition}\label{prop-fixed-points-reduction-alpha-trivial-higgs}
Let $\Lambda\le \gamtz$ be a subgroup. Consider a lift $\theta:\ker\eta\to\Aut(G)$ of $a\vert_{\ker\eta}$, a connected $\Lambda$-bundle $p:Y\to X$ and the group $\wgam\le\Aut(Y)$ fitting in (\ref{eq-extension-wgam}). Let $\tau:\wgam\to\Aut(\gt_0)$ be a homomorphism and $c\in\zzgtw$ a 2-cocycle such that there is an isomorphism of extensions $\gtl:=\pt^{-1}(\Lambda)\cong \gt_0\times_{(\tau,c)}\Lambda$. Assume that $\pg^*a\in\outc$ --- defined as in Section \ref{section-group-theory}
--- and pick $\tilde\tau\in \homtc$ such that $q_*(\tilde\tau)=\pg^*a$.

Let $(F,\psi)$ be a $(\tau,c,\rhotm)$-twisted $\wgam$-equivariant $(\gt_0,\liegm)$-Higgs pair over $Y$. Then $(F,\psi)$ can be regarded as a $(\gtl,\liegm)$-Higgs pair over $X$ via Proposition \ref{prop-twisted-equivariant-bundles-one-to-one}, and its extension of structure group $(E,\phi)$ to $G$ is isomorphic to $(E,\phi)\cdot\gamma$ for each $\gamma\in\Gamma$.
\end{proposition}

\begin{proof}
Throughout the proof we denote by $\tau:\wgam\to\Aut(\gt)$ the extension $e_{\tau,c}$ of the given $\tau:\wgam\to\Aut(\gt_0)$ by abuse of notation. We have to show that $E\cong E\cdot\gamma$ for every $\gamma\in\Gamma$. Pick $\wga\in\wgam$ such that $\eta(\gamma)=p(\wga)$ --- in other words, $(\gamma,\wga)\in\wgame$. Consider the automorphism of $F$ given by $\wga$. We want to know how it interacts with the $\gtl$-action on $F$. Consider the twisted product $\gt_0\times_{(\tau,c)}\wgam$. For each $e\in F$ and $(g,\lambda)\in\gt_0\times_{(\tau,c)}\Lambda\cong \gtl$,
\begin{align*}
    (e(g,\lambda))\bullet\wga
    &=((eg)\bullet\lambda)\bullet\wga
    \\&=((((eg)\bullet\wga)\bullet\wga^{-1})\bullet\lambda)\bullet\wga
    \\&=((eg)\bullet\wga)\bullet(1,\wga^{-1})(1,\lambda)(1,\wga)
    \\&=
    (e\bullet\wga)\tau^{-1}_{\wga}(g)\tau^{-1}_{\wga}(1,\lambda)
    \\&=(e\bullet\wga)\tau^{-1}_{\wga}(g,\lambda),
\end{align*}
therefore $\wga$ induces an isomorphism 
\begin{equation*}
    \hg:F\xrightarrow{\sim}\etag^{*-1}\tau_{\wga}(F);\,e\mapsto \etag^{*-1}e\bullet \wga,
\end{equation*}
where $F$ is regarded as a $\gtl$-bundle over $X$. 

Now let $\tilde\tau\in q_*^{-1}(\pg^*a)\subset\homtc$. Using the extension $\tilde\tau_{\gamma,\wga}\in\Aut(G)$ of $\tau$ we may extend the above isomorphism to $E$, thus obtaining $E\cong \etag^*\tilde\tau_{\gamma,\wga}(F)$. Since $\tilde\tau_{\gamma,\wga}$ is a lift of $a_{\gamma}$, we obtain an isomorphism $\hg:E\xrightarrow{\sim} E\cdot\gamma$ for each $\gamma\in\Gamma$.
It is left to show that it respects the higgs field.  
Let $(e,v)\otimes k$ be a local expression for $\psi$, where $e\in F$, $v\in\liegm$ and $k\in K_Y$. By $\wgam$-invariance of $\psi$,
\begin{equation*}
    \psi=\psi\cdot\wga=(e\bullet\wga,\omu_{\gamma}\tilde\tau_{\gamma,\wga}v)\otimes\etag^*k=\omu_{\gamma}\etag^{*-1}\tilde\tau_{\gamma,\wga}(\hg(\psi)).
\end{equation*}
We thus conclude that $\hg(\psi)=\mug\etag^{*}\tilde\tau_{\gamma,\wga}^{-1}(\psi)$ and so $\hg$ sends $\phi$ to $\phi\cdot\gamma$ as required.
\end{proof}

\begin{proposition}\label{prop-simple-fixed-points-oscar-ramanan-alpha-trivial-principal}
Let $E$ be a simple $G$-bundle over $X$ isomorphic to $E\cdot\gamma$ for every $\gamma\in\Gamma$. Then there exist a lift $\theta$ of $a\vert_{\ker\eta}$ and a reduction of structure group $F$ of $E$ to $\gt$ satisfying the following: let $p:Y\to X$ be a connected component of the $\gamt$-bundle $F/\gt_0\to X$ with Galois group $\Lambda:=\gal(Y/X)$, and $\wgam$ the subgroup of $\Aut(Y)$ lifting $\eta(\Gamma)$. Then there is a homomorphism $\otau:\wgam\to\Out(\gt_0)$ such that, for every lift $\tau:\wgam\to\Aut(\gt_0)$, 
we can find a 2-cocycle $c\in Z^2_{\tau}(\wgam,Z(\gt_0))$ such that the following statements hold.
\begin{enumerate}
    \item There is an isomorphism $\gtl:=\pt^{-1}(\Lambda)\cong\gt_0\times_{(\tau,c)}\Lambda$ of extensions of $\gt_0$.
    \item $\pg^*a\in\outc$ --- defined as in Section \ref{section-group-theory}.
    \item The tautological reduction of $p^*F$ to $\gt_0$ is a $(\tau,c)$-twisted $\wgam$-equivariant $\gt_0$-bundle.
\end{enumerate}
\end{proposition}

\begin{proof}
Let $E$ be a simple $G$-bundle isomorphic to $E\cdot\gamma$ for each $\gamma\in\Gamma$. According to Proposition \ref{prop-simple-fixed-points} there is a lift 
$$\theta:\ker\eta\to\Aut(G)$$
of $a\vert_{\ker\eta}$ and a reduction $F$ with structure group
$\gt$. Another lift of $a\vert_{\ker\eta}$ gives such a reduction if and only if it is in the orbit of $\theta$ under the conjugation action of $\Int(G)$.
On the other hand, 
there exists a homomorphism
$$\wt:\Gamma\to\Aut(G)$$
lifting $a$. For each $\gamma\in \ker\eta$ we set $\tg=\Int_{\sg}\wtg$, where 
$$\Int_s:\ker\eta\to\Int(G);\,\gamma\mapsto \Int_{\sg}$$
is an element of $Z^1_{\wt}(\ker\eta,\Int(G))$ by Lemma \ref{lemma-lifts-vs-non-abelian-cohomology}. Fix an element $\beta\in\Gamma$. We claim that $\wtb(\gt)=G^{\theta'}$, where \begin{equation}\label{eq-def-theta'}
    \theta':\ker\eta \to\Aut(G);\,\gamma\mapsto\tg':=\Int_{\wtb(s_{\beta\gamma\beta^{-1}})}\wtg.
\end{equation}

Indeed, for every $\gamma\in\ker\eta$ and $g\in\gt$,
\begin{align*}
    \Int_{\wtb(s_{\beta\gamma\beta^{-1}})}\wtg\wtb(g)&=
    \Int_{\wtb(s_{\beta\gamma\beta^{-1}})}\wtb\wt_{\beta\gamma\beta^{-1}}(g)\\&=
    \wtb\Int_{s_{\beta\gamma\beta^{-1}}}\wt_{\beta\gamma\beta^{-1}}(g)\\&=\wtb\theta_{\beta\gamma\beta^{-1}}(g)\\&=\wtb(g),
\end{align*}
where the last equation follows from $\ker\eta$ being a normal subgroup of $\Gamma$, so the inclusion $\wtb(\gt)\subseteq G^{\theta'}$ follows. The same reasoning applied to $\beta$ instead of $\beta^{-1}$ and $\theta'$ instead of $\theta$ gives $\wt_{\beta}(\gtt)\subseteq\gt$, so $\gtt\subseteq\wtb(\gt)$ also follows. Moreover, we can see that $\theta'$ is a homomorphism, or equivalently, by Lemma \ref{lemma-lifts-vs-non-abelian-cohomology}, that the map 
$$\ker\eta\to\Int(G);\gamma\mapsto\Int_{\wtb(s_{\beta\gamma\beta^{-1}})}$$ is an element of $Z^1_{\wt}(\ker\eta,\Int(G))$. Indeed, for every $\gamma$ and $\gamma'\in\ker\eta$,
$$\Int_{\wtb(s_{\beta\gamma\beta^{-1}})\wtg(\wtb(s_{\beta\gamma'\beta^{-1}}))}
=\Int_{\wtb(s_{\beta\gamma\beta^{-1}}\wt_{\beta\gamma\beta^{-1}}(s_{\beta\gamma'\beta^{-1}}))}
=\Int_{\wtb(s_{\beta\gamma\gamma'\beta^{-1}})},$$
where the last equality follows from the fact that $\Int_s\in Z^1_{\wt}(\ker\eta,\Int(G))$.

For each $\gamma\in\Gamma$ there is an isomorphism
$$\hg:E\to \etag^{*-1}\wtg E.$$
By simplicity of $E$, this isomorphism is unique up to multiplication by an element of $Z$.
Consider the subbundle $F':=\fb^{*}\hb(F)$ of $E$. For every $e\in E$ and $g\in G$ the equation
$\fb^{*}\hb(eg)=\fb^{*}\hb(e)\wtb(g)$ holds,
and so by the previous paragraph $F'$ is a reduction of $E$ with structure group $\wtb(\gt)=\gtt$. 
We have thus shown that $F'$ is a reduction of structure group of $E$ to $\gtt$. Since $\theta'$ is a lift of $a\vert_{\ker\eta}$, it must be a conjugate of $\theta$ by an element of $\Int(G)$. Moreover, $F't_{\beta}=F$ for some element $t_{\beta}\in G$ by Proposition \ref{prop-simple-fixed-points}. Note that, given another element $t'_{\beta}\in G$ such that $F't'_{\beta}=F$, there must exist $g\in\gt$ such that $t_{\beta}=t'_{\beta}g$.

Now let $p:Y\to X$ be the étale cover of $X$ defined by $F/\gt_0$. For simplicity we assume that it is connected, so that $\Lambda=\gamtz$ and $\gtl=\gt$. The general case follows using the same argument after taking a connected component of $Y$ and the corresponding monodromy group $\Lambda$. The map $t$ determines a homomorphism from $\Gamma$ to the group of automorphisms $\Aut(Y)$ of $Y$, where $\Aut(Y)$ is equipped with its trasposed multiplication, as follows: for each $\gamma\in\Gamma$ take an element of $Y=F/\gt_0$, choose a representative in $F$, apply the map
$\etag^{*}\hg(\cdot)t_{\gamma}:F\to F$
and take the projection on $F/\gt_0$. Different choices of $t$ provide different choices of the map $\Gamma\to\Aut(Y)$ differing by elements of $\gamtz=\gal(Y/X)$. Consider the subgroup $\wgamm_Y\le\Aut(Y)$ consisting of the lifts of $\eta(\Gamma)$ to $Y$. By the previous discussion this is the subgroup generated by $\gamtz$ and the image of $\Gamma$ under any of the maps $\Gamma\to\Aut(Y)$ that we have defined. Recall that we are regarding the $\wgamm_Y$-action on $Y$ as a right action --- in fact $\gamtz$ acts naturally on the right, since its action is a principal bundle action. The rest of the proof is committed to defining a suitable action of $\wgamm_Y$ on the tautological reduction of $p^*F$ to $\gt_0$, which we also call $F$ because they have the same total space.

Let $\wgame\subseteq\Gamma\times\wgamm_Y$ be the subset of pairs $(\gamma,\wga)$ such that $\eta(\Gamma)=p(\wga)$. First note that for each pair $(\gamma,\wga)\in \wgame$ we can define an automorphism 
$\etag^{*}\hg(\cdot)t_{\gamma,\wga}$ of $F$ lifting $\wga$, where $t_{\gamma,\wga}\in G$ is chosen suitably. This is an automorphism of the total space of $F$ as a complex variety, and in general it does not preserve the $\gt_0$-action. Note that another such choice $t'_{\gamma,\wga}$ is equal to $t_{\gamma,\wga}g$ for some $g\in\gt_0$. We claim that the map $$\wt\Int_t:\wgame\to\Aut(G);\,(\gamma,\wga)\mapsto \wtg\Int_{t_{\gamma,\wga}},$$ 
induces a homomorphism $\wgamm_Y\to\Out(\gt_0).$
Indeed, for every $(\gamma,\wga)\in \wgame$, $e\in F$ and $g\in\gt_0$ one has
\begin{align*}
    \etag^{*}\hg(eg)t_{\gamma,\wga}=\etag^{*}\hg(e)t_{\gamma,\wga}(\Int^{-1}_{t_{\gamma,\wga}}\wtg^{-1}(g))=\etag^{*}\hg(e)t_{\gamma,\wga}((\wtg\Int_{t_{\gamma,\wga}})^{-1}(g))
\end{align*}
and so, since both $\etag^{*}\hg(e)t_{\gamma,\wga}$ and $\etag^{*}\hg(eg)t_{\gamma,\wga}$ lie in $F$, the automorphism $\wtg\Int_{t_{\gamma,\wga}}$ of $G$ must preserve $\gt_0$. Moreover, the image of the restriction $\wtg\Int_{t_{\gamma,\wga}}\vert_{\gt_0}$ in $\Out(\gt_0)$ does not depend on $\gamma$: given another $\gamma'\in\Gamma$ such that $\eta(\gamma')=p(\wga)$, the element $\gamma^{-1}\gamma'$ must lie in $\ker\eta$. By simplicity of $(E,\phi)$, the morphism
$$h_{\gamma^{-1}\gamma'}(\cdot)s_{\gamma^{-1}\gamma'}:(E,\phi)\to\eta_{\gamma^{-1}\gamma'}^{*-1} \theta_{\gamma^{-1}\gamma'}(E,\mu_{\gamma^{-1}\gamma'}\phi)=\theta_{\gamma^{-1}\gamma'}(E,\mu_{\gamma^{-1}\gamma'}\phi)$$
must be equal to the isomorphism induced by the identity on $(F,\psi)$ up to an element of $Z$ --- see Proposition \ref{prop-reduction}. Hence the restrictions of $\etag^{*}\hg(\cdot)t_{\gamma,\wga}$ and $\eta_{\gamma'}^{*}h_{\gamma'}(\cdot)t_{\gamma',\wga}$ to $F$ differ by an element of $G$ and, since they both preserve $F$ and lift $\wga$, this element must actually be in $\gt_0$. The compatibility of the $\gt_0$-actions then implies that $\wtg\Int_{t_{\gamma,\wga}}$ and $\wt_{\gamma'}\Int_{t_{\gamma',\wga}}$ must differ by an element of $\Int^G(\gt_0)$, as required. 

Therefore, we have obtained a map $\wgamm_Y\to\Out(\gt_0)$. To prove the claim it is left to show that it is a homomorphism. For every $(\gamma,\wga)$ and $(\gamma',\wga')\in \wgame$, one has
\begin{align*}
    (\eta_{\gamma\gamma'}^{*}h_{\gamma\gamma'}(F))t_{\gamma\gamma',\wga\wga'}
    &=F
    \\&=\eta_{\gamma}^{*} h_{\gamma}(F)t_{\gamma,\wga}
    \\&=\eta_{\gamma'}^{*} h_{\gamma'}(\eta_{\gamma}^{*-1} h_{\gamma}(F)t_{\gamma,\wga})t_{\gamma',\wga'}\\&=
(\eta_{\gamma\gamma'}^{*}h_{\gamma\gamma'}(F))\wt_{\gamma'}^{-1}(t_{\gamma,\wga})t_{\gamma',\wga'}z
\end{align*}
for some element $z\in Z$. Hence there exists $g\in\gt_0$ such that $\wt_{\gamma'}(t_{\gamma,\wga})t_{\gamma',\wga'}z=t_{\gamma\gamma',\wga\wga'}g,$
and so
$$\wtg\Int_{t_{\gamma,\wga}}\wt_{\gamma'}\Int_{t_{\gamma',\wga'}}=\wt_{\gamma\gamma'}\Int_{\wt_{\gamma'}^{-1}(t_{\gamma,\wga})t_{\gamma',\wga'}}=\wt_{\gamma\gamma'}\Int_{t_{\gamma\gamma',\wga\wga'}}\Int_g,$$
as required.

We have thus obtained a homomorphism 
$\otau:\wgamm_Y\to\Out(\gt_0).$
Lift $\otau$ to a homomorphism $\tau:\wgamm_Y\to\Aut(\gt_0)$. In other words, we may rechoose the map $t:\wgame\to G$ so that
\begin{equation}\label{eq-def-t}
    \tau:=\wt\Int_t\vert_{\gt_0}
\end{equation}
is a homomorphism factoring though $\wgam$. More precisely, for each coset $\wga\in\wgamm_Y\cong\wgame/\ker\eta$ we may choose a representative $(\gamma,\wga)\in\wgame$ and define $t_{\gamma,\wga}$ so that $\tau_{\gamma}:=\wtg\Int_{t_{\gamma,\wga}}\vert_{\gt_0}$ and then, for each $\gamma'\in\gamma\ker\eta$, take the unique $t_{\gamma',\wga}$ that fits into the equation 
\begin{equation}\label{eq-c-doesnt-depend-on-ker}
    \etag^{*}\hg(\cdot)t_{\gamma,\wga}\vert_{F}=\eta_{\gamma'}^{*}h_{\gamma'}(\cdot)t_{\gamma',\wga}\vert_{F},
\end{equation}
which exists because of simplicity of $E$ and because $\ker\eta$ acts trivially on $\gt_0$ via $\theta$. Let $c:\wgamm_Y\times\wgamm_Y\to Z(\gt_0)$ be the (unique) map satisfying
\begin{equation}\label{eq-def-c}
    \eta_{\gamma'}^{*}h_{\gamma'}(\etag^{*}\hg(\cdot )t_{\gamma,\wga}) t_{\gamma',\wga'}\vert_{F}= h_{\gamma\gamma'}(\cdot c(\wga,\wga'))t_{\gamma\gamma',\wga\wga'}\vert_{F}
\end{equation}
for each $\wga,\wga'\in\wgamm_Y$ and any $\gamma,\gamma'\in\Gamma$ such that $\eta(\gamma)=p(\wga)$ and $\eta(\gamma')=p(\wga')$. Note that $c$ is well defined: both sides of (\ref{eq-def-c}) are independent of the choice of $\gamma$ and $\gamma'$ by (\ref{eq-c-doesnt-depend-on-ker}), and they are both $\gt_0$-equivariant with respect to the $\gt_0$ action given by $\wt_{\gamma}\Int_{t_{\gamma}}\wt_{\gamma'}\Int_{t_{\gamma'}}=\wt_{\gamma\gamma'}\Int_{t_{\gamma\gamma'}}$, hence they differ by an element of $\gt_0$ commuting with $\gt_0$ --- in other words, an element in $Z(\gt_0)$. Because of associativity of the composition of homomorphisms of $\gt_0$ bundles, $c\in Z^2_{\tau}(\wgamm_Y,Z(\gt_0))$ is a 2-cocycle. 

Define a map
\begin{equation}\label{eq-def-twisted-equivariant-actiom-F}
     F\times\wgamm_Y\to F;\,(e,\wga)\mapsto e\bullet\wga:= \etag^{*-1}\hg(e)t_{\gamma,\wga},\,\eta(\gamma)=p(\wga).
\end{equation}
This is independent of the choice of $\gamma$ by (\ref{eq-c-doesnt-depend-on-ker}) and it descends to the action of $\wgamm_Y$ by the construction of $t$. Moreover, it is a $(\tau,c)$-twisted $\wgamm_Y$-equivariant action on $F$ by (\ref{eq-def-t}) and (\ref{eq-def-c}). 

This finishes the proof of (3). Statements (1) and (2) follow by construction, so we are done.
\end{proof}

\begin{proposition}\label{prop-simple-fixed-points-oscar-ramanan-alpha-trivial-higgs}
Let $(E,\phi)$ be a simple $G$-Higgs bundle over $X$ isomorphic to $(E,\phi)\cdot\gamma$ for every $\gamma\in\Gamma$. Then there exist a lift $\theta$ of $a\vert_{\ker\eta}$ and a reduction of structure group $(F,\psi)$ of $(E,\phi)$ to $\gtl:=\pt^{-1}(\Lambda)$ satisfying the following: let $p:Y\to X$ be the étale cover associated to the $\Lambda$-bundle $F/\gt_0\to X$ and $\wgam$ the subgroup of $\Aut(Y)$ lifting $\eta(\Gamma)$. Then there is a homomorphism $\otau:\wgam\to\Out(\gt_0)$ such that, for every lift $\tau:\wgam\to\Aut(\gt_0)$, 
we can find a 2-cocycle $c\in Z^2_{\tau}(\wgam,Z(\gt_0))$ such that the following statements hold.
\begin{enumerate}
    \item There is an isomorphism $\gtl\cong\gt_0\times_{(\tau,c)}\Lambda$ of extensions of $\gt_0$.
    \item $\pg^*a\in\outc$, defined as in Section \ref{section-group-theory}.
    \item There exists $\tilde\tau\in\homtc$ 
    such that $q_*(\tilde\tau)=\pg^*a$ and the tautological reduction of $p^*(F,\psi)$ to $\gt_0$ is a $(\tau,c,\rhotm)$-twisted $\wgam$-equivariant $(\gt_0,\liegm)$-Higgs pair.
\end{enumerate}
\end{proposition}

\begin{proof}
The argument builds on top of the proof of Proposition \ref{prop-simple-fixed-points-oscar-ramanan-alpha-trivial-principal}. However, we need to check several facts related to the Higgs field.

First recall that at the beginning of the proof we used Proposition \ref{prop-simple-fixed-points} to obtain a lift $\theta:\ker\eta\to\Aut(G)$ of $a\vert_{\ker\eta}$ and a reduction of structure group $F$ of $E$ to $\gt$. In this case we obtain a reduction of structure group $(F,\psi)$ of $(E,\phi)$ to $(\gt,\liegm)$. Take any lift $\wt:\Gamma\to\Aut(G)$ of $a$. By Lemma \ref{lemma-lifts-vs-non-abelian-cohomology} there is a 1-cocycle $\Int_s\in Z^1_{\wt}(\ker\eta,\Int(G))$ such that $\theta=\Int_s\wt$. By assumption there is an isomorphism 
$$\hg:(E,\phi)\to\etag^{*-1}\wtg(E,\omu_{\gamma}\phi)$$
for every $\gamma\in\Gamma$. In the proof of Proposition \ref{prop-simple-fixed-points-oscar-ramanan-alpha-trivial-principal} we showed that $F':=\etag^*\hg(F)$ is a reduction of $E$ to $G^{\theta'}$, where $\theta'$ is given by (\ref{eq-def-theta'}). It is left to show that there exists a section $\psi'\in H^0(X,F'(\lieg^{\theta'}_{\mu})\otimes K_X)$ such that the Higgs pair $(F',\psi')$ is a reduction of structure group of $(E,\phi)$ --- in other words, that the Higgs field $\phi$ actually lies in $H^0(X,F'(\lieg^{\theta'}_{\mu})\otimes K_X)$.

Because of the equation $\hb(\phi)=\mub^{-1}\fb^{*}\wt_{\beta}(\phi)$, we know that $\phi=\mub\wtb\fb^{*}\hb(\phi)\in H^0(X,E(\lieg)\otimes K).$ Given $x\in X$, $e\in F_x$, $k\in K_x$ and an element $v\in \liegm$ such that $\phi_x=(e,v)\otimes k$, the equation
$$\phi_{\eta_{\beta}(x)}=[(\fb^{*}\hb(e),\mub \wtb(v))\otimes \fb^{*}k]_{\fb(x)}$$ holds,
hence it is enough to show that $\wtb(\liegm)\subseteq\lieg^{\theta'}_{\mu}$ --- in fact, one may prove that they are equal. For every $\gamma\in\ker\eta$ and $v\in\liegm$, one has
\begin{align*}
    \ttg \wtb(v)
    &= \Ad_{\wtb(s_{\beta\gamma\beta^{-1}})}\wtg \wtb(v)\\&=\wtb \Ad_{s_{\beta\gamma\beta^{-1}}} \wt_{\beta\gamma\beta^{-1}} (v)\\&=\wtb \theta_{\beta\gamma\beta^{-1}} (v)\\&=
    \mu_{\beta\gamma\beta^{-1}} \wtb (v)\\&=\mug \wtb(v),
\end{align*}
where the last equality follows from the fact that $\C^*$ is abelian and $\mu$ is a homomorphism, as required. 

Next we defined an étale cover $Y:=F/\gt_0$, which we assume for simplicity that is connected, and a $(\tau,c)$-twisted $\wgam$-action on $F$, thought of as a $\gt_0$-bundle over $Y$. 
According to Section \ref{section-group-theory}, in order to define the action of $\wgamm_Y$ on $H^0(Y,F(\liegm)\otimes K_X)$ we need $\tau$ to be the restriction of a $c$-twisted homomorphism $\tilde{\tau}:\wgame\to\Aut(G)$ such that $q_*\tilde{\tau}=p_{\Gamma^*}a$ --- this is statement (2). We claim that $\tilde\tau:=\wt\Int_t:\wgame\to\Aut(G)$ is a $c$-twisted homomorphism. Indeed, there is a map
\begin{equation*}
     E\times\wgame\to E;\,(e,(\gamma,\wga))\mapsto e\bullet (\gamma,\wga):= \etag^{*}\hg(e)t_{\gamma,\wga}.
\end{equation*}
By simplicity of $E$,
$$eg_{(\gamma,\wga),(\gamma',\wga')}\bullet (\gamma,\wga)\bullet (\gamma',\wga')=e\bullet (\gamma\gamma',\wga\wga')$$
for each $e\in E$, $(\gamma,\wga)$ and $(\gamma',\wga')\in\wgame$ and some $g_{(\gamma,\wga),(\gamma',\wga')}\in G$ depending on $(\gamma,\wga)$ and $(\gamma',\wga')$ but not on $e$. But this action restricts to (\ref{eq-def-twisted-equivariant-actiom-F}) on $F$, hence in fact $g_{(\gamma,\wga),(\gamma',\wga')}=c(\wga,\wga')$. The claim follows by considering the $G$-action on $E$ and noting that $q_*(\tilde\tau)=\pg^*a$ is implied by the fact that $\tilde\theta$ lifts $a$. 

We have thus constructed a $c$-twisted homomorphism $\tilde\tau\in\homtc$ 
such that $q_*(\tilde\tau)=\pg^*a$. Moreover, $F$ is a $(\tau,c)$-twisted $\wgam$-equivariant $\gt_0$-bundle. To finish the proof of (3) we have to show that the Higgs field $\psi$ such that $(F,\psi)$ is a reduction of structure group of $(E,\phi)$ is $\wgam$-invariant with respect to the twisted equivariant action of $\wgam$ on $F$ and $\rhotm$ --- see Section \ref{section-group-theory} for the definition of the equivariant action. Let $(e,v)\otimes k$ be a local expression for $\psi$, where $e\in F$, $v\in\liegm$ and $k\in K_Y$, the canonical bundle of $Y$. Keeping the notation of the proof of Proposition \ref{prop-simple-fixed-points-oscar-ramanan-alpha-trivial-principal},
\begin{align*}
    ((e,v)\otimes k)\cdot\wga
    &=(e\bullet\wga,\rhotm^{-1}(v))\otimes \etag^{*}k
    \\&=(e\bullet\wga,\mug \wtaug^{-1}(v))\otimes \etag^{*}k
    \\&=(\etag^{*}\hg(e)t_{\gamma,\wga},\mug\Ad_{t_{\gamma,\wga}}^{-1}\wtg(v))\otimes \etag^{*}k
    \\&=
    (\etag^{*}\hg(e),\mug \wtg(v))\otimes \etag^{*-1}k
    \\&=(e,v)\otimes k
\end{align*}
for every $(\gamma,\wga)\in\wgame$, where the last equation follows from the definition of $h$. 

This finishes the proof of (3). Statements (1) and (2) follow by construction, so we are done.
\end{proof}

\begin{proposition}\label{prop-prym-narasimhan-ramanan-alpha-trivial-higgs}
Consider a lift $\theta:\ker\eta\to\Aut(G)$ of $a\vert_{\ker\eta}$, a subgroup $\Lambda\le\gamtz$, a connected étale cover $p:Y\to X$ with Galois group $\Lambda$ and the subgroup $\wgam\le\Aut(Y)$ lifting $\eta(\Gamma)\le\Aut(X)$. Let $\tau$ be a homomorphism $\tau:\wgam\to\Aut(\gt_0)$ and $c\in\zzgtw$ a 2-cocycle such that there is an isomorphism of extensions $\gtl=\pt^{-1}(\Lambda)\cong \gt_0\times_{(\tau,c)}\Lambda$. Assume that $\pg^*a\in\outc$ --- defined as in Section \ref{section-group-theory} --- and pick $\tilde\tau\in \homtc$ such that $q_*(\tilde\tau)=\pg^*a$.

There is a morphism
\begin{equation}\label{eq-PNR+extension-morphism-general}
    \cM(Y,\gt_0,\wgam,\tau,c,\liegm,\rhotm)\to \cM(X,G),
\end{equation}
given by Proposition \ref{prop-prym-narasimhan-ramanan} and extension of structure group. 
\end{proposition}
\begin{proof}
Let $(E,\phi)$ be an equivariantly polystable $(\tau,c,\rhotm)$-twisted $\wgam$-equivariant $(\gt_0,\liegm)$-Higgs pair over $Y$. By Theorem \ref{EH1-equivariant} there exists a $\wgam$-invariant metric satisfying (\ref{hitchin-equation-pairs}). In particular it is $\gamtz$-invariant, hence by Theorem \ref{EH1-equivariant} the underlying twisted $\Lambda$-equivariant Higgs pair is equivariantly polystable and, by Propositions \ref{prop-polystability-extension-of-structure-group-non-connected} and \ref{prop-prym-narasimhan-ramanan}, the $(\gt,\liegm)$-Higgs pair over $X$ given by Proposition \ref{prop-twisted-equivariant-bundles-one-to-one} and extension of structure group is polystable. By Proposition \ref{prop-polystability-extension-structure-group} the $G$-Higgs bundle given by extension of structure group to $G$ is polystable as required.
\end{proof}



Denote by $\wcM(Y,\gt_0,\wgam,\tau,c,\liegm,\rhotm)$ the image of (\ref{eq-PNR+extension-morphism-general}).

\begin{theorem}\label{th-prym-narasimhan-ramanan-alpha-trivial-higgs}
Fix $\theta\in\Hom(\ker\eta,\Aut(G))$ lifting $a$. The following relations between moduli spaces hold:
\begin{enumerate}
    \item $$\bigcup_{[\beta],Y,[\tau^{\beta\theta}],[c^{\beta\theta}],[\tilde\tau]}\wcM(Y,G^{\beta\theta}_0,\wgam,\tau^{\beta\theta},c^{\beta\theta},\lie g^{\beta\theta}_{\mu},\rhotm)
    \subset\cM(X,G)^{\Gamma}. $$
    
    \item $$\cM_*(X,G)^{\Gamma}\subset
    \bigcup_{[\beta],Y,[\tau^{\beta\theta}],[c^{\beta\theta}],[\tilde\tau]}\wcM(Y,G^{\beta\theta}_0,\wgam,\tau^{\beta\theta},c^{\beta\theta},\lie g^{\beta\theta}_{\mu},\rhotm).$$
    
\end{enumerate}
Here $[\beta]$ runs through $H^1_{\theta}(\ker\eta,\Int(G))$, $Y$ runs over connected étale covers of $X$ with Galois group
$\Lambda\le\widehat{\Gamma}^{\beta\theta}$, $[\tau^{\beta\theta}]\in \Hom(\wgam,\Out(G^{\beta\theta}_0))$ and $[c^{\beta\theta}]\in H^2_{\tau^{\beta\theta}}(\wgam,Z(G^{\beta\theta}_0))$ are such that $\pg^*a\in \Hom_{\beta\theta,\tau^{\beta\theta},c^{\beta\theta}}(\wgame,\Out(G))$, and their restrictions to $\Lambda$ satisfy $G^{\beta\theta}_0\times_{(\tau^{\beta\theta},c^{\beta\theta)}}\Lambda\cong \gtl$ as extensions. Moreover, for each choice of $[\beta]$, $[\tau^{\beta\theta}]$ and $[c^{\beta\theta}]$, $\tilde\tau$ is an element of the set $ \Hom_{\beta\theta,\tau^{\beta\theta},c^{\beta\theta}}(\wgame,\Aut(G))$, $q_*\tilde\tau=\pg^*a$ and
\begin{equation*}
[\tilde\tau]:=p_{\beta\theta,\tau^{\beta\theta},c^{\beta\theta}}(\tilde\tau)\in\Hom_{\theta,\otau}(\wgame,\Aut(G)^{\gtl}/\Int^G(\gt_0)).
\end{equation*}
To make sense of all the notation, see Section \ref{section-group-theory}.
\end{theorem}

\begin{proof}
Follows from Propositions \ref{prop-fixed-points-reduction-alpha-trivial-higgs}, \ref{prop-simple-fixed-points-oscar-ramanan-alpha-trivial-higgs} and \ref{prop-prym-narasimhan-ramanan-alpha-trivial-higgs}.
\end{proof}

\subsection{The general theorem}\label{section-general-theorem}

Now we tackle the general case. Let $X$ be a compact Riemann surface, $G$ a connected semisimple complex Lie group with centre $Z$ and $\Gamma$ a finite subgroup of $H^1(X,Z)\rtimes(\Aut(X)\times\Out(G))\times\C^*$. By Section \ref{section-action} there is a right action of $\Gamma$ on $\mdl(X,G)$. The projections on each factor provide homomorphisms $\eta:\Gamma\to\Aut(X)$, $a:\Gamma\to\Out(G)$ and $\mu:\Gamma\to\C^*$, together with a 1-cocycle $\alpha\in Z^1_{a,\eta}(\Gamma,H^1(X,Z))$, where the action of $\Gamma$ on $H^1(X,Z)$ is determined by $a$ --- via extension of structure group --- and $\eta$ --- via pullback. In other words, this is a map $\alpha:\Gamma\to H^1(X,Z)$ satisfying
\begin{equation}\label{eq-1-cocycle-eta-a}
    \alpha_{\gamma\gamma'}=\alg\etag^{*-1}\ag(\alpha_{\gamma'})
\end{equation}
for each $\gamma$ and $\gamma'\in\Gamma$.
The restriction $\alpha\vert_{\ker\eta}$ is a 1-cocycle in $Z^1_{a}(\ker\eta,H^1(X,Z))\cong H^1(X,Z^{1}_a(\ker\eta,Z))$, hence any of its connected components provides an étale cover $X_{\alpha,\eta}\to X$ with Galois group $\Gamma_{\alpha,\eta}\le Z^{1}_a(\ker\eta,Z)$. 

Now pick a lift $\theta:\ker\eta\to\Aut(G)$ of $a\vert_{\ker\eta}$. Let $p:Y\to X$ be a connected component of a $\gamtt$-bundle in $\qqt^{-1}(\alpha\vert_{\ker\eta})$ --- see (\ref{eq-def-qt}) ---, and set $\olambda:=\gal(Y/X)\le \gamtt$. Consider the subgroup $\halpha\le H^1(X,Z)$ generated by the image of $\alpha$. This is finite because both $\Gamma$ and $Z$ are finite. Its image $p^*\halpha\le H^1(Y,Z)$ via pullback is also a finite subgroup. It determines an element of $\Hom(\halpha,H^1(Y,Z))\cong H^1(Y,\Hom(\halpha,Z))$, which in turn defines a connected étale cover $p_{Y_{\alpha}}:Y_{\alpha}\xrightarrow{p_{\halpha}} Y\to X$. Moreover, $Y_{\alpha}$ is equipped with a projection $Y_{\alpha}\to\halpha\in H^1(X,\Hom(\halpha,Z))$, where $\halpha$ is regarded as an étale cover of $X$ via the isomorphism $H^1(X,\Hom(\halpha,Z))\cong\Hom(\halpha,H^1(X,Z))$. Let $\Lambda:=\gal(Y_{\alpha}/X)$, denote by $\wgam\le\Aut(Y_{\alpha})$ the group of automorphisms of $Y_{\alpha}$ lifting $\eta(\Gamma)\le\Aut(X)$ and let $\wgame:=\{(\gamma,\wga)\in\Gamma\times\wgam\suhthat \eta(\gamma)=p(\wga)\}$. The commutative diagramme (\ref{eq-extension-wgam}) still holds and it has exact rows and columns. There is also a diagramme (\ref{eq-restriction-extension maps}), with the same notation. We may also define $\homtc$. 

Given $\taut:\wgam\to\Aut(\gt_0)$ and $\ct\in Z^2_{\taut}(\wgam,Z(\gt_0))$ whose restrictions to $\Lambda$ factor through $\olambda$ and satisfy $\gtll\cong \gt_0\times_{(\taut,\ct)}\olambda$, together with an element $\tilde\tau\in\homtc$, there is a $(\taut,\ct)$-twisted $ \wgam$-right action on $\liegm$ defined by (\ref{eq-action-tildetau}), which we denote by $\rhotm:\wgam\to\Hom(\liegm,\lieg)$.

By Proposition \ref{prop-prym-narasimhan-ramanan}, there is a morphism
\begin{equation*}
    \cM(Y_{\alpha},G^{\theta}_0,\wgam,\tau^{\theta},c^{\theta},\lie g^{\theta}_{\mu},\rhotm).
    \to\cM(X,G).
\end{equation*}
Let $\wcM(Y_{\alpha},G^{\theta}_0,\wgam,\tau^{\theta},c^{\theta},\lie g^{\theta}_{\mu},\rhotm)
    $ be the image of this morphism.

\begin{theorem}\label{th-prym-narasimhan-ramanan-general}
Fix $\theta\in\Hom(\ker\eta,\Aut(G))$ lifting $a$. The following relations between moduli spaces hold:
\begin{enumerate}
    \item $$\bigcup_{[\beta],Y,[\tau^{\beta\theta}],[c^{\beta\theta}],[\tilde\tau]}\wcM(Y_{\alpha},G^{\beta\theta}_0,\wgam,\tau^{\beta\theta},c^{\beta\theta},\tilde\tau,\lie g^{\beta\theta}_{\mu},\rhotm)
    \subset\cM(X,G)^{\Gamma}. $$
    
    \item $$\cM_*(X,G)^{\Gamma}\subset
    \bigcup_{[\beta],Y,[\tau^{\beta\theta}],[c^{\beta\theta}],[\tilde\tau]}\wcM(Y_{\alpha},G^{\beta\theta}_0,\wgam,\tau^{\beta\theta},c^{\beta\theta},\tilde\tau,\lie g^{\beta\theta}_{\mu},\rhotm).$$
    
\end{enumerate}
Here $[\beta]$ runs through $H^1_{\theta}(\ker\eta,\Int(G))$, $Y$ runs over étale covers of $X_{\alpha,\eta}$ which are connected components of $\gamtt$-bundles in $\qt^{-1}(\alpha\vert_{\ker\eta})$, $[\tau^{\beta\theta}]\in \Hom(\wgam,\Out(G^{\beta\theta}_0))$ and $[c^{\beta\theta}]\in H^2_{\tau^{\beta\theta}}(\wgam,Z(G^{\beta\theta}_0))$ are such that their restrictions to $\Lambda$ factor through $\olambda$ and satisfy 
\begin{equation}\label{iso-Gbetatheta-twisted-product}
G^{\beta\theta}_0\times_{(\tau^{\beta\theta},c^{\beta\theta})}\olambda\cong \gtll.
\end{equation}
Moreover, for each choice of $[\beta]$, $[\tau^{\beta\theta}]$ and $[c^{\beta\theta}]$, $\tilde\tau$ is an element of $ \Hom_{\beta\theta,\tau^{\beta\theta},c^{\beta\theta}}(\wgame,\Aut(G))$, $q_*\tilde\tau=\pg^*a$ and
\begin{equation*}
[\tilde\tau]:=p_{\beta\theta,\tau^{\beta\theta},c^{\beta\theta}}(\tilde\tau)\in\Hom_{\theta,\otau}(\wgame,\Aut(G)^{\gtl}/\Int^G(\gt_0)).
\end{equation*} Finally, if $t:\olambda\to\gtll$ is the map sending $\lambda\in\olambda$ to $(1,\lambda)\in G^{\beta\theta}_0\times_{(\tau^{\beta\theta},c^{\beta\theta})}\olambda\cong \gtll$, then
\begin{equation}\label{eq-2-cocycle-vs-alpha}
    \tilde\tau_{\hat\gamma\lambda^{-1}\hat\gamma^{-1}}(c(\hat\gamma,\lambda)c(\hat\gamma\lambda,\hat\gamma^{-1}))\tilde\tau_{\gamma,\hat\gamma}(t_{\lambda})^{-1}c(\hat\gamma,\hat\gamma^{-1})t_{\hat\gamma\lambda\hat\gamma^{-1}}=\langle\alpha_{\gamma},\lambda\rangle,
\end{equation}
for every $\gamma\in \Gamma$ and $\lambda\in\Lambda$, where $(\gamma,\hat\gamma)\in\wgame$ and $\langle\alpha_{\gamma},\lambda\rangle$ is the evaluation of $p_{Y}\circ\lambda\in\Hom(\halpha,Z)$ at $\alpha_{\gamma}$.
\end{theorem}
\begin{remark}
    As in Remark \ref{remark-isotropy}, we may refine Theorem \ref{th-prym-narasimhan-ramanan-general} by decomposing the different components according to their local types.
\end{remark}

\begin{proof}[Sketch of the proof]
(1) follows by an argument analogous to the proof of Proposition \ref{prop-simple-fixed-points-oscar-ramanan-alpha-trivial-higgs}, except for (\ref{eq-2-cocycle-vs-alpha}). To prove (2) let $(E,\phi)$ be a simple $G$-bundle fixed by the $\Gamma$-action. Using Proposition \ref{prop-simple-fixed-points} we obtain a reduction of structure group $(F,\psi)$ to a $(G_{\beta\theta},\lie g^{\beta\theta}_{\mu})$-Higgs pair for some lift $\beta\theta$ of $a\vert_{\ker\eta}$. Let $Y$ be the étale cover of $X$ given by a connected component of the $\gamtt$-bundle $F(\gs/\gt_0)$. Using the equivalence between Higgs bundles on $X$ and $\Lambda$-equivariant Higgs bundles on $Y_{\alpha}$, together with the fact that $p_{Y_{\alpha}}^*\halpha$ is trivial on $Y_{\alpha}$, we obtain a twisted equivariant $\wgam$-action on the reduction $(F',\psi')$ of $p_{Y_{\alpha}}^*(F,\psi)$ to $\gt_0$, as in the proof of Proposition \ref{prop-simple-fixed-points-oscar-ramanan-alpha-trivial-principal}. 

We obtain (\ref{eq-2-cocycle-vs-alpha}) by considering each $p_{Y_{\alpha}}^*\alg$ as the trivial bundle equipped with its natural $\Lambda$-equivariant action twisted by the natural pairing of $\alg$ with the elements of the Galois group $\Lambda\subset \Hom(\halpha,Z)$, given by evaluation. Indeed, the pullback $p_{Y_{\alpha}}^*(E,\phi)$ features a natural $\Lambda$-equivariant action, and the twisted $\wgam$-equivariant action on $(F',\psi')$ induces a $\Lambda$-equivariant isomorphism 
\begin{equation*}
    h_{\gamma,\hat\gamma}:p_{Y_{\alpha}}^*(E,\phi)\xrightarrow{\sim} \hat\gamma^*\tilde\taug^{-1}p_{Y_{\alpha}}^*(E\otimes \alg,\mug \phi)
\end{equation*}
for each $(\gamma,\hat\gamma)\in\wgame$. Using the isomorphism $p_{Y_{\alpha}}^*\alg\cong\oo$, the trivial line bundle over $Y_{\alpha}$, these must fit in a commutative diagramme
\begin{equation}\label{eq-commutative-diagram-general-th-alpha-equivariant}
    \begin{tikzcd}
E\arrow[r,"h_{\gamma,\hat\gamma}"]\arrow[d,"\lambdaa"]  &   E\arrow[d,"\lambdaaa"]\\
        E\arrow[r,"h_{\gamma,\hat\gamma}"]   &   E,
    \end{tikzcd}
\end{equation}
where $*$ denotes the $\Lambda$-equivariant action. By \cite[Proposition 4.7]{BGGM},
\begin{equation*}
    e*\lambda=e\bullet\lambda t_{\lambda}^{-1}
\end{equation*}
for every $e\in F'\subset p_{Y_{\alpha}}^*E$, where $\bullet$ denotes the twisted $\Lambda$-equivariant action on $F'$. Therefore, for each $e\in F'$ one has
\begin{equation}\label{eq-general-th-lambda}
    e*(\hat\gamma\lambda\hat\gamma^{-1})\langle\alg,\lambda\rangle=e\bullet(\hat\gamma\lambda\hat\gamma^{-1}) t_{\hat\gamma\lambda\hat\gamma^{-1}}^{-1}\langle\alg,\lambda\rangle
\end{equation}
and
\begin{align}
    ((e\bullet\hat\gamma)*\lambda)(\bullet\hat\gamma)^{-1}&= 
    (((e\bullet\hat\gamma)\bullet\lambda t_{\lambda}^{-1})\bullet\hat\gamma^{-1})c(\hat\gamma,\hat\gamma^{-1})
    \label{eq-general-th-gammalambda}\\&=
    (((e\bullet\hat\gamma)\bullet\lambda)\bullet\hat\gamma^{-1})\tilde\tau_{\gamma,\hat\gamma}(t_{\lambda})^{-1}c(\hat\gamma,\hat\gamma^{-1})\nonumber\\&=
    ((ec(\hat\gamma,\lambda)c(\hat\gamma\lambda,\hat\gamma^{-1})))\bullet(\hat\gamma\lambda\hat\gamma^{-1})\tilde\tau_{\gamma,\hat\gamma}(t_{\lambda})^{-1}c(\hat\gamma,\hat\gamma^{-1})\nonumber\\&=
    (e\bullet(\hat\gamma\lambda\hat\gamma^{-1}))\tilde\tau_{\hat\gamma\lambda^{-1}\hat\gamma^{-1}}(c(\hat\gamma,\lambda)c(\hat\gamma\lambda,\hat\gamma^{-1}))\tilde\tau_{\gamma,\hat\gamma}(t_{\lambda})^{-1}c(\hat\gamma,\hat\gamma^{-1}).\nonumber
\end{align}
By (\ref{eq-commutative-diagram-general-th-alpha-equivariant}) the first terms of (\ref{eq-general-th-lambda}) and (\ref{eq-general-th-gammalambda}) are equal, hence comparing the last terms (\ref{eq-2-cocycle-vs-alpha}) follows.
\end{proof}

Let $M(X,G)$ be the moduli space of $G$-bundles over $X$. Let $M(Y_{\alpha},G^{\theta}_0,\wgam,\tau^{\theta},c^{\theta})$ be the moduli space of $(\tau^{\theta},c^{\theta})$-twisted $\wgam$-equivariant $G^{\theta}_0$-bundles over $Y_{\alpha}$. By Proposition \ref{prop-prym-narasimhan-ramanan}, there is a morphism
\begin{equation*}
M(Y_{\alpha},G^{\theta}_0,\wgam,\tau^{\theta},c^{\theta}).
    \to M(X,G).
\end{equation*}
Let $\widetilde M(Y_{\alpha},G^{\theta}_0,\wgam,\tau^{\theta},c^{\theta})\subset M(X,G)
    $ be the image of this morphism.

\begin{corollary}\label{cor-prym-narasimhan-ramanan-general-principal}
Fix $\theta\in\Hom(\ker\eta,\Aut(G))$ lifting $a$. The following relations between moduli spaces hold:
\begin{enumerate}
    \item $$\bigcup_{[\beta],Y_{\alpha},[\tau^{\beta\theta}],[c^{\beta\theta}]}\widetilde M(Y,G^{\beta\theta}_0,\wgam,\tau^{\beta\theta},c^{\beta\theta})
    \subset M(X,G)^{\Gamma}. $$
    
    \item $$M_*(X,G)^{\Gamma}\subset
    \bigcup_{[\beta],Y_{\alpha},[\tau^{\beta\theta}],[c^{\beta\theta}]}\widetilde M(Y,G^{\beta\theta}_0,\wgam,\tau^{\beta\theta},c^{\beta\theta}).$$
    
\end{enumerate}
Here $[\beta]$ runs through $H^1_{\theta}(\ker\eta,\Int(G))$, $Y$ runs over étale covers of $X$ whose Galois group is a subgroup
$\Lambda\le\widehat{\Gamma}^{\beta\theta}$, $[\tau^{\beta\theta}]\in \Hom(\wgam,\Out(G^{\beta\theta}_0))$ and $[c^{\beta\theta}]\in H^2_{\tau^{\beta\theta}}(\wgam,Z(G^{\beta\theta}_0))$ are such that $\pg^*a\in \Hom_{\beta\theta,\tau^{\beta\theta},c^{\beta\theta}}(\wgame,\Out(G))$ and their restrictions to $\Lambda$ factor through $\olambda$ and satisfy 
\begin{equation*}
    G^{\beta\theta}_0\times_{(\tau^{\beta\theta},c^{\beta\theta})}\olambda\cong \gtll.
\end{equation*}
Moreover, if $t:\olambda\to\gtll$ is the map sending $\lambda\in\olambda$ to $(1,\lambda)\in G^{\beta\theta}_0\times_{(\tau^{\beta\theta},c^{\beta\theta})}\olambda\cong \gtll$, then (\ref{eq-2-cocycle-vs-alpha}) holds.
\end{corollary}

If $\mu$ is trivial, we may use the notation of Sections \ref{section-twisted-equivariant-higgs-pairs-and-hitchin-equations} and \ref{section-character-variety-alpha-trivial} to achieve a generalization of Theorems \ref{main-rep} and \ref{th-prym-narasimhan-ramanan-character-varieties}.

\begin{theorem}\label{th-prym-narasimhan-ramanan-character-general}
Let $\mu:\Gamma\to\C^*$ be trivial. Fix $\theta\in\Hom(\ker\eta,\Aut(G))$ lifting $a$. The following relations between character varieties hold:
\begin{enumerate}
    \item $$\bigcup_{[\beta],Y,[\tau^{\beta\theta}],[c^{\beta\theta}]}\wcalR(Y,G^{\beta\theta}_0,\wgam,\tau^{\beta\theta},c^{\beta\theta})
    \subset\calR(X,G)^{\Gamma}. $$
    
    \item $$\calR_{\irr}(X,G)^{\Gamma}\subset
    \bigcup_{[\beta],Y,[\tau^{\beta\theta}],[c^{\beta\theta}]}\wcalR(Y,G^{\beta\theta}_0,\wgam,\tau^{\beta\theta},c^{\beta\theta}).$$
    
\end{enumerate}
Here $[\beta]$ runs through $H^1_{\theta}(\Gamma,\Int(G))$, $Y$ runs over étale covers of $X$ which are connected components of $\gamtt$-bundles in $\qt^{-1}(\alpha\vert_{\ker\eta})$, and $[\tau^{\beta\theta}]\in \Hom(\wgam,\Out(G^{\beta\theta}_0))$ and $[c^{\beta\theta}]\in H^2_{\tau^{\beta\theta}}(\wgam,Z(G^{\beta\theta}_0))$ are such that $\pg^*a\in \Hom_{\beta\theta,\tau^{\beta\theta},c^{\beta\theta}}(\wgame,\Out(G))$ and their restrictions to $\Lambda$ factor through $\olambda$ and satisfy 
(\ref{iso-Gbetatheta-twisted-product}).
Moreover, if $t:\olambda\to\gtll$ is the map sending $\lambda\in\olambda$ to $(1,\lambda)\in G^{\beta\theta}_0\times_{(\tau^{\beta\theta},c^{\beta\theta})}\olambda\cong \gtll$, then (\ref{eq-2-cocycle-vs-alpha}) holds.
\end{theorem}
\begin{proof}
Use Theorems \ref{equivariant-nahc} and \ref{th-prym-narasimhan-ramanan-general}.
\end{proof}

\providecommand{\bysame}{\leavevmode\hbox to3em{\hrulefill}\thinspace}

\end{document}